\documentclass[12pt,oneside,a4]{amsart}

\usepackage{typearea}   
\usepackage{setspace}
\usepackage{amsfonts}
\usepackage{amsmath}
\usepackage{amssymb}
\usepackage{amsthm}
\usepackage{mathrsfs}
\usepackage{hyperref}
\usepackage{color}
\usepackage[all]{xy}
\usepackage{mathpazo}

\usepackage{todonotes}

\newcommand{\C}{\mathbb{C}}
\newcommand{\Z}{\mathbb{Z}}

\newcommand{\IR}{\mathbb{R}}
\newcommand{\IC}{\mathbb{C}}
\newcommand{\IZ}{\mathbb{Z}}
\newcommand{\IN}{\mathbb{N}}
\newcommand{\IQ}{\mathbb{Q}}
\newcommand{\IG}{\mathbb{G}}

\newcommand{\IQbar}{\overline{\mathbb{Q}}}
\newcommand{\mat}[1]{\mathrm{Mat}_{#1}}
\newcommand{\gl}[1]{\mathrm{GL}_{#1}}

\newcommand{\lcm}[1]{\mathrm{lcm}({#1})}

\newcommand{\be}{{\boldsymbol{e}}}
\newcommand{\mf}{{f}}

\newcommand{\bzeta}{{\boldsymbol{\zeta}}}
\newcommand{\bfeta}{{\boldsymbol{\eta}}}
\newcommand{\bxi}{{\boldsymbol{\xi}}}
\newcommand{\ord}{\mathrm{ord}}

\newcommand{\vi}{{r}} 

\newcounter{maincounter}
\numberwithin{maincounter}{section}
\numberwithin{equation}{section}

\newtheorem{thm}[maincounter]{Theorem}

\newtheorem{lemma}[maincounter]{Lemma}
\newtheorem{corol}[maincounter]{Corollary}
\newtheorem{conj}[maincounter]{Conjecture}
\newtheorem{defin}[maincounter]{Definition}
\newtheorem{propo}[maincounter]{Proposition}

\newcommand{\disc}[1]{\mathrm{disc}({#1})}
\newcommand{\dist}[1]{\mathrm{dist}({#1})}
\newcommand{\gal}[1]{\mathrm{Gal}({#1})}

\newcommand{\cD}{\mathcal{D}}
\newcommand{\cE}{\mathcal{E}}
\newcommand{\cBprim}{{\mathcal{B}}}
\newcommand{\ssm}{\backslash}

\newcommand{\hproj}[1]{h({#1})}

\newcommand{\ma}[1]{m({#1})}
\newcommand{\vol}[1]{{\mathrm{vol}}({#1})}
\newcommand{\sv}{r}

\newcommand{\rk}[1]{{\mathrm{rk}}({#1})}

\newcommand{\GammaN}{{(\IZ/N\IZ)^\times}}
\newcommand{\GammaM}{{(\IZ/M\IZ)^\times}}
\newcommand{\GammaE}{{(\IZ/E\IZ)^\times}}
\newcommand{\GammaNp}{{(\IZ/N'\IZ)^\times}}
\newcommand{\GammaNi}{{(\IZ/N_i\IZ)^\times}}
\newcommand{\Gammaf}{{(\IZ/f\IZ)^\times}}

\newcommand{\essatoral}{{essentially atoral}}  

\makeatletter
\def\imod#1{\allowbreak\mkern10mu({\operator@font mod}\,\,#1)}
\makeatother

\newif\ifextracontent
\extracontentfalse

\newcommand{\refcomment}[2]{{#2}}

\begin{document}

\title[Galois orbits and atoral sets]{Galois orbits of torsion points
near atoral sets}
\date{\today}
\author{V. Dimitrov and P. Habegger}

\address{Department of Mathematics, University of Toronto, 40 St. George Street, 
Toronto ON, M25 2E5, Canada}
\email{vesselin.dimitrov@gmail.com}
\address{Department of Mathematics and Computer Science, University of Basel, Spiegelgasse 1, 4051 Basel, Switzerland}
\email{philipp.habegger@unibas.ch}







\begin{abstract}
  We prove that the Galois equidistribution of torsion points of the
  algebraic torus $\mathbb{G}_{m}^d$ extends to the singular test
  functions of the form $\log{|P|}$, where $P$ is a Laurent polynomial
  having algebraic coefficients that vanishes on the unit real $d$-torus
  in a set whose Zariski closure in $\IG_m^d$ has codimension at least
  $2$. Our result includes a power
  saving quantitative estimate of the decay rate of the
  equidistribution. It refines an ergodic theorem of Lind, Schmidt, and
  Verbitskiy, of which it also supplies a purely Diophantine proof. As
  an application, we confirm Ih's integrality finiteness conjecture on
  torsion points for a class of atoral divisors of~$\mathbb{G}_m^d$.
\end{abstract}

\maketitle

\tableofcontents

\section{Introduction}

\subsection{Main results}

Let $d\ge 1$ be an integer and let $\IG_m^d$ denote the
$d$-dimensional algebraic torus with base field $\IC$. We will
identify $\IG_m^d$ with $(\IC\ssm\{0\})^d$, the group of its
$\IC$-points.

Let $\bzeta\in\IG_m^d$ be a \emph{torsion point}, \textit{i.e.}, a
point of a finite order. We define
\begin{equation}
\label{def:deltazeta}
\delta(\bzeta) = \inf \bigl\{|a| : a \in\IZ^d\ssm\{0\} \text{ with }\bzeta^a = 1\bigr\}
\end{equation}
where, here and throughout the article, $|\cdot|$ denotes the
maximum-norm; we refer to Section \ref{sec:notation} for the notation
$\bzeta^a$.

It is well-known that the Galois orbit $\{\bzeta^\sigma
: \sigma \in \gal{\IQ(\bzeta)/\IQ} \}$ becomes equidistributed in
$\IG_m^d$ with respect to the Haar measure as
$\delta(\bzeta)\rightarrow\infty$. More precisely, if $f:
\IG_m^d\rightarrow \IR$ is a continuous  function with compact support, then
\begin{equation}
  \label{def:equidistribution}
\frac{1}{[\IQ(\bzeta):\IQ]} \sum_{\sigma\in \gal{\IQ(\bzeta)/\IQ}}
f(\bzeta^\sigma) \rightarrow \int_{[0,1)^d} f(\be(x)) dx
\end{equation}
as $\delta(\bzeta)\rightarrow\infty$
where
\begin{equation}
  \label{eq:defbex}
 \be(x) = \left(e^{2\pi \sqrt{-1}x_1}, \ldots, e^{2\pi \sqrt{-1}x_d}\right)
\end{equation}
for $x=(x_1,\ldots,x_d)\in\IR^d$.

Our aim is to investigate the equidistribution result for test
functions $f = \log|P|$ where $P$ is a Laurent polynomial in $d$
unknowns and with algebraic coefficients. Such $P$ may vanish on
$(S^1)^d$, where $S^1 = \{z\in\IC : |z|=1\}$ is the unit circle, and
so $f$ is not defined everywhere. But for $\delta(\bzeta)$ large in
terms of $P$, Laurent's Theorem~\cite{Laurent} 
also known as the Manin--Mumford
Conjecture for $\IG_m^d$, implies that $P$ does not vanish at any
conjugate of $\bzeta$.
\refcomment{1}{See also \cite{SarnakAdams} for 
another proof by Sarnak and Adams.}
Moreover, the integral of $f$ over $(S^1)^d$
exists as the singularity is merely logarithmic. It is known as the
\textit{Mahler measure}
\begin{equation*}
  m(P) = \int_{[0,1)^d} \log |P(\be(x))| dx,
\end{equation*}
see for instance Section 3.4 in~\cite{Schinzel} for the convergence of
this integral for arbitrary $P\in \IC[X_1^{\pm 1},\ldots,X_d^{\pm
1}]\ssm\{0\}$.


A \textit{torsion coset} of $\IG_m^d$ is the translate of a connected algebraic
subgroup of $\IG_m^d$ by a point of finite order.
We call a torsion coset \textit{proper} if it does not equal $\IG_m^d$.

We call
  $P\in \IC[X_1^{\pm 1},\ldots,X_d^{\pm
1}]\ssm\{0\}$ \textit{\essatoral{}} if the Zariski closure of
\begin{equation*}
\{(z_1,\ldots,z_d) \in (S^1)^d : P(z_1,\ldots,z_d)=0\}
\end{equation*}
in  $\IG_m^d$ is a finite union of irreducible algebraic sets of
 codimension at
least $2$ and proper torsion cosets.

For example, if  $d=1$
then  $P$ is \essatoral{} if and only if it does not vanish
at any point of infinite multiplicative order in $S^1$.

\refcomment{3}{Lind--Schmidt--Verbitskiy
define the notion of an \textit{atoral} Laurent polynomial}
$P\in \IZ[X_1^{\pm 1},\ldots,X_d^{\pm 1}]\ssm\{0\}$
in  Definition
2.1~\cite{LSV:13}.
An atoral Laurent polynomial  is \essatoral{} in our sense.
Moreover, if $P$ is  irreducible
then it  is atoral if and only if the
intersection of its zero locus with $(S^1)^d$ has dimension at most
$d-2$ \refcomment{2}{as a
semi-algebraic set}, cf. by Proposition
2.2~\cite{LSV:13}.
A related, but not quite equivalent, definition of atoral Laurent
polynomials with complex coefficients was introduced earlier by
Agler--McCarthy--Stankus~\cite{AMS:06}.


Let $\IQbar$ denote the algebraic closure of $\IQ$ in $\IC$. We are
ready to state our first result.

\begin{thm}
\label{thm:main}
For each \essatoral{} 
  $P\in \IQbar[X_1^{\pm 1},\ldots,X_d^{\pm 1}]\ssm \{0\}$  
  there exists $\kappa >0$ with the following property. 
  Suppose $\bzeta\in\IG_m^d$ has finite  order
with $\delta(\bzeta)$  sufficiently large.
  Then $P(\bzeta^\sigma)\not=0$ for all $\sigma\in
\gal{\IQ(\bzeta)/\IQ}$ and
\begin{equation*}
\frac{1}{[\IQ(\bzeta):\IQ]} \sum_{\sigma \in \gal{\IQ(\bzeta)/\IQ}} \log |P(\bzeta^\sigma)|
   = m(P) +O(\delta(\bzeta)^{-\kappa})
\end{equation*}
as $\delta(\bzeta)\rightarrow\infty$, where the implicit constant
depends only on $d$ and $P$.
\end{thm}

Theorem \ref{thm:main2} below is a more precise version of this
result. 
In particular, we allow $\sigma$ to range over subgroups of
$\gal{\IQ(\bzeta)/\IQ}$ whose index and conductor grow sufficiently slow, the
conductor is defined in
Section \ref{quantequidistribution}.
Moreover, $\kappa$ 
 depends only on $d$ and the number of non-zero terms
appearing in $P$.  
Our method of proof allows one to determine an explicit value for
$\kappa$.

Torsion points in $\IG_m^d$ are characterized as the algebraic points of height
zero; see Section~\ref{sec:notation} for the definition of the height 
$h : \IG_m^n(\bar{\IQ}) \to\ [0,\infty)$.
Bilu~\cite{Bilu} proved that Galois orbits of algebraic points $\boldsymbol{\alpha} \in \IG_m^d$
of small height satisfy an analogous equidistribution statement as
(\ref{def:equidistribution}), asymptotically as $h(\boldsymbol{\alpha}) \to 0$
and $\delta(\boldsymbol{\alpha}) \to \infty$; the definition
(\ref{def:deltazeta}) extends naturally to non-torsion points and may
take infinity as a value. It is natural to ask whether Theorem
\ref{thm:main} admits a suitable generalization to points of small
height. Autissier's example~\cite{Autissier:06} rules out the verbatim
generalization  already for $\IG_m$. He constructed a sequence
$(\alpha_n)_{n\in\IN}$ of pairwise distinct algebraic numbers whose
height tends to $0$ but such that $\frac{1}{[\IQ(\alpha_n):\IQ]}
\sum_{\sigma} \log|\sigma(\alpha_n)-2|$ tends to $0$ for
$n\rightarrow\infty$. 
But the  integral of the corresponding test function
against the unit circle is $\log 2$. 
An interesting problem still arises if the test function has at
worst a logarithmic
singularity of real codimension at least $2$ on $(S^1)^d$. Suppose that
 $|f(z)|$  is  $O\!\left(|\!\log(|P(z)|^2+|Q(z)|^2)|\right)$ on an open neighborhood
of $(S^1)^d$  in $\IG_m^d$, where $P$ and $Q$ are
non-constant and coprime Laurent polynomials with algebraic
coefficients, and that $f$ vanishes on the complement of a compact
set in $\IG_m^d$. One may then ask about comparing the average of $f$ over the Galois orbit of 
 $\boldsymbol{\alpha} \in \IG_m^d(\bar{\IQ})$ with the
average of $f$ over $(S^1)^d$: is their difference bounded by
$\ll_f ( h(\boldsymbol{\alpha}) + \delta(\boldsymbol{\alpha})^{-1} )^{\kappa}$, 
for some $\kappa > 0$ depending only on $P$ and $Q$? We also mention Chambert-Loir and
Thuillier's Th\'eor\`eme 1.2~\cite{CLT:09} which is a general equidistribution result
for points of small height,  allowing $\log|P|$ as a test function if
the zero locus of $P$ in $\IG_m^d$ is a finite union of
torsion cosets. In this paper we allow $\log{|P|}$ as a test
function if $P$ is essentially atoral but we average over points
of finite order.

Our Theorem \ref{thm:main} recovers a variant of the result of
 Lind--Schmidt--Verbitskiy~\cite{LSV:13}. In their work,  the
 sum is not over the Galois orbit of a single point of finite order
 but rather over a finite subgroup $G$ of $\IG_m^d$. For this purpose
 we define
\begin{equation}
\label{def:deltaG}
  \delta(G) = \inf \bigl\{|a| : a\in\IZ^d\ssm\{0\} \text{ such that
  }\bzeta^a =1 \text{ for all }\bzeta\in G\bigr\}.
\end{equation}

Each finite subgroup of $\IG_m^d$ is a disjoint union of Galois
orbits. This observation allows us to recover the Theorem of Lind,
Schmidt, and Verbitskiy with an estimate on the decay rate.

\begin{thm}
\label{thm:LSV}
  Let $P\in \IQ[X_1^{\pm 1},\ldots,X_d^{\pm 1}]\ssm\{0\}$ be \essatoral{}.
  There exists $\kappa>0$ such that for any
  finite  subgroup $G\subset\IG_m^d$
  we have
  \begin{equation}
    \label{eq:LSVaverage}
    \frac{1}{\#G} \sum_{\substack{\bzeta\in G  \\ P(\bzeta)\not=0}}
    \log |P(\bzeta)| = m(P) + O(\delta(G)^{-\kappa})
  \end{equation}
  where the implicit constant
  depends only on $d$ and $P$.
\end{thm}

\refcomment{7}{To relate (\ref{eq:LSVaverage}) to the expression in
Lind, Schmidt, and Verbitskiy's Theorem 1.3~\cite{LSV:13} we refer to
Lemma 2.1~\cite{LSV:entropygrowth} as well as the comments on page
1063 and 1064~\cite{LSV:13}. Note that $G$ is $\Omega_\Gamma$ and
$\# G$ is
$|\IZ^d/\Gamma|$ in the notation of \cite{LSV:13}.}

\refcomment{4,5}{Lind, Schmidt, and Verbitskiy's approach is based on an in-depth
study~\cite{SV:sandpiles,LSV:entropygrowth,LSV:13} of an associated
dynamical system: the algebraic $\Z^d$-action on a closed,
shift-invariant subgroup of
$(S^1)^{\Z^d}$ whose  dual is $\IZ[X_1^{\pm 1},\ldots,X_d^{\pm
1}]/(P)$. The atoral condition, in the sense of~\cite{LSV:13}, turns
out to be equivalent to the existence of a non-trivial summable
homoclinic point.}

Theorem~\ref{thm:LSV} may be read as a strong quantitative
estimate on the growth of periodic points for such dynamical systems.
The refinement to Galois orbits, Theorem~\ref{thm:main}, does not seem
to be directly possible by the homoclinic method, nor does it seem to
follow formally from the case (\ref{eq:LSVaverage}) of finite subgroups,
which is where the dynamical method applies. 

Our method of proof draws its origins in work of
Duke~\cite{Duke:Combinatorial}. It differs from the method of Lind,
Schmidt, and Verbitskiy. However, it is striking that the notion of
atoral appears crucially in both approaches.


The first-named author~\cite{dimitrov} was able to prove Theorem \ref{thm:LSV} for a
general Laurent polynomial  when $G$ equals the group of
$N$-torsion elements in $\IG_m^d$. 


Let us return to Galois orbits. 
We believe that the hypothesis on $P$ being essentially atoral is
also unnecessary in  Theorem \ref{thm:main} on Galois orbits. The next conjecture sums up our
expectations. It is related to Schmidt's Conjecture 
  \cite[Remark 21.16(2)]{Schmidt:DSAO}. 

\begin{conj}
  \label{conj:galoisallP}
For each $P\in \IQbar[X_1^{\pm 1},\ldots,X_d^{\pm 1}]\ssm\{0\}$
  there exists $\kappa>0$ with the following property.  Suppose
  $\bzeta\in\IG_m^d$ has finite order with 
 $\delta(\bzeta)$  sufficiently large. Then
  $P(\bzeta^\sigma)\not=0$ for all $\sigma\in \gal{\IQ(\bzeta)/\IQ}$ and 
  \begin{equation*}
    \frac{1}{[\IQ(\bzeta):\IQ]} \sum_{\sigma\in \gal{\IQ(\bzeta)/\IQ}}
    \log |P(\bzeta^\sigma)| = m(P) + O(\delta(\bzeta)^{-\kappa})
  \end{equation*}
  as $\delta(\bzeta)\rightarrow\infty$,
  where the implicit constant
  depends only on $d$ and $P$.
\end{conj}

For $d=1$ this conjecture follows from work of M.~Baker, Ih, and Rumely
\cite{bakerihrumely}, see their statement around (6). They use a
version of Baker's deep estimates on linear forms in logarithms. 
Already the case $d=2$ and $P(X_1,X_2) = X_1 + X_1^{-1} + X_2 + X_2^{-1} - 3$
is open. 


\subsection{Ih's conjecture on integral torsion points}
As another application of our results we derive a 
special case of Ih's Conjecture~\cite{bakerihrumely} in the
multiplicative setting. 
Let $P \in \IQbar[X_1^{\pm 1},\ldots,X_d^{\pm 1}]$. A special case of Ih's Conjecture
predicts that the set of torsion points
$\bzeta \in \IG_m^d$ such that
$P(\bzeta)$ is an algebraic unit is not Zariski dense in
$\mathbb{G}_m^d$, unless the zero
set of $P$ in $\IG_m^d$ is itself a finite union of proper torsion cosets. 
M.~Baker, Ih, and Rumely~\cite{bakerihrumely} cover the case
$d = 1$ for arbitrary polynomials.
Their approach runs through a similar limiting statement as our
Theorem \ref{thm:main} for univariate polynomials. 

Here we solve a  case of Ih's Conjecture for
\essatoral{} polynomials  with integral coefficients.

\begin{corol}
  \label{cor:Ih}
  Let $K\subset\IC$ be a number field with ring of integers $\IZ_K$ and let
  $P \in \IZ_K[X_1^{\pm 1},\ldots,X_d^{\pm 1}]\ssm\{0\}$. 
  Suppose that the zero set of $P$ in $\IG_m^d$ is  not 
  a finite union of torsion cosets. 
  Suppose  in addition that $\tau(P)$ is \essatoral{} for all field
  embeddings $\tau:K\rightarrow\IC$. 
  \refcomment{9}{Then there exists $B\ge 1$ such that 
    if $\bzeta\in\IG_m^d$ has finite order and $P(\bzeta)$ is an algebraic
    unit, then
    $\delta(\bzeta)\le B$.}
\end{corol}

Ih's Conjecture expects 
the existence of $B$  without assuming that each $\tau(P)$ is \essatoral{}. Observe
that the result of M.~Baker, Ih, and Rumely is not a direct consequence
of this corollary, as we do not allow univariate polynomials that
vanish at a point of infinite multiplicative  order on the unit
circle.
Our approach does not depend on the theory of linear forms in logarithms.


A special class of atoral
polynomials, to which our results apply \emph{a fortiori}, are the
irreducible integer Laurent polynomials $P \in \IZ[X_1^{\pm
    1},\ldots,X_n^{\pm 1}]\setminus \{0\}$ that are not fixed up-to a
monomial factor and up-to a sign by the involution sending each $X_i$ to $1/X_i$. We call
these $P$ {asymmetric}. They are atoral in the sense of Lind--Schmidt--Verbitskiy, see the proof of
Proposition 2.2~\cite{LSV:13}. Hence an asymmetric Laurent polynomial is  \essatoral{}.  
The converse is false as the Laurent polynomial 
\begin{equation*}
  X_1+X_1^{-1}+X_2+X_2^{-1}-4.
\end{equation*}
is \essatoral{}; \refcomment{11}{indeed, its zero locus on $(S^1)^2$ consists of the
single point $(1,1)$.}

If $K=\IQ$, Corollary~\ref{cor:Ih} in the case of an asymmetric, and
thus necessarily irreducible Laurent polynomial $P$, can be deduced as
follows from the Manin--Mumford Conjecture for $\mathbb{G}_m^d$. Indeed,
if  $\gamma$ is a unit in the ring of algebraic integers of a cyclotomic field,
then
$\eta=\overline{\gamma} / \gamma$ is an algebraic integer whose Galois
conjugates lie on $S^1$. \refcomment{12}{So $\eta$ is a root of unity by Kronecker's
Theorem, see Theorem 1.5.9~\cite{BG}.} 
We consider the zero $(\eta,\bzeta)$ of
$P(X_1^{-1},\ldots,X_d^{-1})-X_0P(X_1,\ldots,X_d)$, which is
irreducible and defines an algebraic subset of $\IG_m^d$ none of whose
geometric irreducible components is a torsion coset. A similar
argument applies if $K$ is a totally real number field.

\subsection{Overview of the proof}

We close the introduction by describing the method of proof of
Theorem~\ref{thm:main}, which builds upon work of the second-named
author~\cite{hab:gaussian} and is related to the approach
of Duke~\cite{Duke:Combinatorial}.
The basic idea is to reduce the multivariate statement in
Theorem~\ref{thm:main} to the univariate case.
Whereas we worked with torsion points of prime order
in~\cite{hab:gaussian}, 
the main technical difficulty in this paper is that we
allow torsion points of arbitrary order.

Any torsion point $\bzeta\in\IG_m^d$ of order $N$ takes on the
form $(\zeta^{a_1},\ldots,\zeta^{a_d})$ where $\zeta=\be(1/N)$ is a
root of unity of order $N$ and $a = (a_1,\ldots,a_d)\in\IZ^d$.
The precise manner how the non-unique $a$ is chosen is delicate and
will be discussed
below. 
The notation $\bzeta = \zeta^a$ will be quite useful. A non-boldface
$\zeta$ denotes a root of unity and boldface $\bzeta$ suggests a torsion point of
$\IG_m^d$. 

If $P$ is as in Theorem~\ref{thm:main}, but for simplicity with coefficients in
$K=\IQ$, we define
the univariate polynomial
\begin{equation}
  \label{eq:univariatespecialization}
  Q(X) = P(X^a)=
P(X^{a_1},\ldots,X^{a_d}) \in \IQ[X^{\pm 1}].
\end{equation}
\refcomment{13}{Multiplying $Q$ by a power of $X$ turns out to be  harmless, so
 one can assume that $Q$ is a polynomial.}
 The values $|P(\bzeta^\sigma)|$
equal the values of $|Q(\zeta^{\sigma})|$ as $\sigma$ ranges over
$\gal{\IQ(\zeta)/\IQ}$. 

\medskip
    {\bf The univariate case and root separation (Section \ref{mahler}).}
Let us suppose for the moment that $\bzeta = \zeta$ is  a root of
unity. It is classical that the Galois conjugates of $\zeta$ are
equidistributed around the unit circle; we recall of these  facts in Section~\ref{quantequidistribution}. So
(\ref{def:equidistribution}) holds for $f(z) = \log|Q(z)|$ provided
$Q$ has no zero on the unit circle.
In  Proposition~\ref{univariatebound} we make convergence
quantitative for such $Q$. Roughly speaking, 
   for all $\epsilon > 0$ we have
\begin{equation}
  \label{eq:univariatecasesimplified}
\frac{1}{[\IQ(\zeta):\IQ]} \sum_{\sigma} \log |Q(\zeta^\sigma)| = m(Q)
+ O_{P,\epsilon} \left(\frac{|a|^{1+\epsilon}}{N^{1-\epsilon}}\right)
\end{equation}
where $\sigma$ runs over $\gal{\IQ(\zeta)/\IQ}$. 
Actually, the hypothesis on $Q$ is slightly weaker as we  allow it
to vanish at roots of unity, if all $Q(\zeta^\sigma)\not=0$. 
This hypothesis is ultimately a reflection of  the hypothesis that
the multivariate $P$ is \essatoral{} in Theorem~\ref{thm:main}.
Indeed, in the  univariate case, being  \essatoral{} boils down to not
vanishing at any
point of infinite multiplicative  order in $S^1$. 
The hypothesis on $Q$ is crucial for our method to work. The main
difficulty we encounter in the average
(\ref{eq:univariatecasesimplified})
are exceptionally small values of  $Q$ at some $\zeta^\sigma$.
The burden is to show, in a uniform sense, that no complex root $z$ of $Q$ can be too close
to  $\zeta^\sigma$ in a suitable sense.

If $z$ is itself a root of unity, doing this is straightforward
as $|z - 1|\gg 1/\ord(z)$.

The difficulty lies in the case when $z$ has infinite multiplicative  order. Here it is tempting to
apply a version of Baker's Theorem on linear forms in logarithms,
as did M.~Baker, Ih, and Rumely~\cite{bakerihrumely}.
However, and as already discussed  by Duke 
 in Section 3~\cite{Duke:Combinatorial} this seems unhelpful
 for the problem at hand.
\refcomment{14}{ Indeed,    estimates on linear forms in
 two logarithms such as \cite{LMN:95} lead to a factor
 $[\IQ(z):\IQ]^2
= O(|a|^2)$ in a bound for any member of the sum in 
(\ref{eq:univariatecasesimplified}).}
This is not good enough for our application as 
$|a|^2/[\IQ(\zeta):\IQ]$ may spoil the average in
(\ref{eq:univariatecasesimplified}).

Our solution is to use the banal inequality $|z-\zeta|\ge
\bigl||z|-1\bigr|$ which lies at the heart of the method
here and in~\cite{hab:gaussian}. 
As $z$ is no root of unity, and as $Q$ does not vanish at points of infinite
multiplicative  order on $S^1$, we have $|z|\not=1$ and so the banal inequality provides a 
non-trivial lower bound. We now explain how it leads to  a useful estimate on $|z - \zeta|$
via lower bounding $\bigl||z|-1\bigr|$. 

If $z$ is close to the unit circle, then $\bigl||z|-1\bigr|$ is
approximately $|z-1/\overline z|$.
In~\cite{hab:gaussian} a result  of Mahler~\cite{Mahler:DiscIneq}
on the separation of roots of an integer polynomial led to
a suitable lower bound for $|z-1/\overline z|$.
In that paper, the second-named author used his  counting result on 
approximations to a set definable in an o-minimal structure.
This allowed to make Mahler's estimate uniform over the
various zeros $z$ of $Q$.

The main tool of the present paper is a uniform generalization
of Mahler's inequality for the separation of
several pairs of roots of $Q$. Such a generalization was obtained by
Mignotte \cite{Mignotte:95}. 
In Section \ref{mahler} we give a variant of Mignotte's theorem that
is tailored to our application and is self-contained. We thus
bypass the o-minimal theory used in~\cite{hab:gaussian}.
We still require Bombieri, Masser, and Zannier's
Theorem~\cite{BMZGeometric} to be mentioned below.
Moreover, our Theorem \ref{thm:main} is effective in nature.

A possible approach towards Conjecture \ref{conj:galoisallP} lies in extending
 (\ref{eq:univariatecasesimplified}) to $Q$ that are allowed
to vanish at any point of $S^1$.
As observed,  we lack a suitable lower bound for
$|z-\zeta|$ if $z$ is an  algebraic number of infinite multiplicative
order on the unit circle.
\refcomment{15}{As suggested in the similar setting of Lemma 4.2~\cite{hab:gaussian},
it turns out that the $z$ of interest have small height $h(z)$. We
therefore propose the following conjecture.}

\begin{conj}\label{conj:linearforms}
For all $B\ge 1$ and  $\epsilon > 0$ there exists a constant
$c=c(B,\epsilon)>0$ with the following property. 
  Let $z \in \IC$ be an algebraic number with $|z|=1$ and
  $h(z)\le B/D$ where 
  $D =
  [\IQ(z):\IQ]$. If $\zeta\in\IC\ssm\{z\}$ is a root of unity of order $N$, then
  $\log |\zeta-z|\ge - c D^{1+\epsilon} N^{\epsilon}$. 
\end{conj}

The crux of this conjecture is its best-possible dependency on the degree $D$. 
In comparison, the state-of-the-art results in the theory of linear
forms in two logarithms of algebraic numbers in the $D$-aspect, such as Laurent, Mignotte, and
Nesterenko's Th\'eor\`eme~3~\cite{LMN:95}, have only a quadratic dependency on $D$. 

\medskip
{\bf Equidistribution of torsion points (Section
  \ref{quantequidistribution}).}
\refcomment{16}{We return to the  case $\bzeta=\zeta^a$ of a general torsion point in
$\IG_m^d$ of order $N$.} The exponent vector $a$ used to define $Q$
as in (\ref{eq:univariatespecialization})
depends on $\bzeta$. 
For this reason it is important that  the error term
in Proposition \ref{univariatebound} is explicit  in
terms of $Q$.
Moreover, it is important to choose $a$ with $|a|$ as small as possible. 
For fixed $\zeta$ the exponent $a$ is
well-defined up-to addition of an element in $N\IZ^d$. So clearly we may assume
$|a|\le N$, although this is not good enough in view of (\ref{eq:univariatecasesimplified}). Fortunately, there is a
second degree of freedom, namely we can replace $\bzeta$ by any Galois
conjugate of itself.

This leads us to classical questions of equidistribution of the Galois
orbit of $\bzeta$; we compile the necessary statements in Section
\ref{quantequidistribution}. Using the 
Erd\"os--Tur\'an Theorem and the theory of Gau\ss{} sums, 
 Lemma \ref{approx} produces $a$ with $|a| =
 O(N \delta(\bzeta)^{-1/(3d)})$ such that $\zeta^a$ is a Galois
 conjugate of $\bzeta$.

Let us return to  the
error term in (\ref{eq:univariatecasesimplified}). One factor $N$
cancels out and the error term becomes
$N^{2\epsilon} \delta(\bzeta)^{-(1+\epsilon)/(3d)}$.
\refcomment{17}{The innocuous $\epsilon$ in (\ref{eq:univariatecasesimplified}) is
ultimately responsible for the factor $N^{2\epsilon}$.
Although
$\delta(\bzeta)\le N$, there is no non-trivial bound in the reverse
direction and 
$N^{2\epsilon} \delta(\bzeta)^{-(1+\epsilon)/(3d)}$ could explode.
}

\medskip
{\bf Factoring $\bzeta$ (Section \ref{sec:lattice}).} The solution to this problem is described
in Section \ref{sec:lattice}. In Proposition
\ref{prop:geometryofnumbers} we factor $\bzeta$ into a product
$\bfeta\bxi$ where $\bxi$ has finite order $M$ such that
$\bxi=\be(a/M)$ where $|a| = O(M^{1-\kappa})$. Moreover, the order of
$\bfeta$ is  bounded from above
by a small power of $N$. The power saving obtained in the
exponent of $N$  is small even when compared to 
the saving obtained for $|a|$. The methods
employed come from the \refcomment{18}{geometry of numbers} and  
\refcomment{19}{slopes of lattices in $\IR^d$.} 

We will  replace
$\bzeta$ by $\bxi$ and the univariate polynomial $Q(X)=P(X^a)$ by
$P(\bfeta^{a} X^a)$. This last transformation does not change the height or
the monomial structure of $Q$. But it can change the field generated by
its coefficients
as the order of $\bfeta$ and hence its field of definition vary as
$\bzeta$ varies.
For this reason, we must keep track of the base field of $Q$ throughout
the whole argument. 

\medskip
{\bf Putting everything together (Sections \ref{sec:preliminary},
  \ref{sec:equidistribution}, \ref{sec:endgame}).} In Sections
\ref{sec:preliminary}, and \ref{sec:endgame} we put all ingredients
together to prove the final result. Here we apply a result of
Bombieri, Masser, and Zannier~\cite{BMZGeometric} on the intersections
of a subvariety in $\IG_m^d$ of codimension at least $2$ with all
$1$-dimensional algebraic subgroups of $\IG_m^d$. Roughly speaking,
this result shows that if $P$ is \essatoral{}, then for ``most''
choices of $a$ the univariate polynomial $Q$ as in
(\ref{eq:univariatespecialization}) does not vanish at any point of
infinite  multiplicative  order on $S^1$. Recall that this property of $Q$ was crucial
to deduce (\ref{eq:univariatecasesimplified}). Bombieri, Masser, and
Zannier's result is related to the study
\refcomment{21}{of}   unlikely intersections, for an
overview we refer to Zannier's book~\cite{ZannierBook}. Another tool
that makes an appearance is Lawton's Theorem~\cite{Lawton}.

The intermediate Section \ref{sec:equidistribution} contains a weak
version of a result of Hlawka~\cite{Hlawka:71} on the numerical
integration of a continuous, multivariate function. The results
obtained there are useful in connection with the function attaching
the Mahler measure to a non-zero polynomial.

\medskip
{\bf Appendices.} In Appendix~\ref{app:lawton} we give a quantitative
version of Lawton's Theorem~\cite{Lawton} regarding the convergence of
a sequence of Mahler measures. Unfortunately, we are not able to use
the very closely related theorem in~\cite{hab:gaussian} as we require
 additional uniformity. The arguments in this appendix follow
closely Lawton's strategy. 
 Finally, in the second Appendix we show how to
deduce Theorem \ref{thm:LSV}, the Theorem of
Lind--Schmidt--Verbitskiy, from our Theorem \ref{thm:main}.

\refcomment{22}{\subsection{Final remarks}}

The results mentioned above, in particular the theorem of Bombieri,
Masser, and Zannier, also play an important role in Le's
approach~\cite{Le:homologytorsion}. The question on how small a sum of
roots of unity can be was raised by
Myerson~\cite{Myerson:sumofrootsof1} in connection with a
combinatorial question~\cite{Myerson:CombI,Myerson:CombII} which was
later studied by Duke~\cite{Duke:Combinatorial}.
Dubickas~\cite{Dubickas:23rootsofunity} has more recent work in this
direction for sums of $2$ and $3$ roots of unity of prime order.

\subsection*{Acknowledgments} 
The authors thank Pierre Le Boudec for references regarding Gau\ss{}
sums, Peter Sarnak for the reference to
Le's~\cite{Le:homologytorsion}, and Shouwu Zhang for pointing out
Chambert-Loir and Thuillier's work~\cite{CLT:09}. We also thank the
referee for carefully reading this text and for providing many
valuable comments that led to improvements of the text and some
simplifications. The authors thank Fran\c cois Brunault, Antonin
Guilloux, Mahya Mehrabdollahi, and Riccardo Pengo for pointing out a
mistake in an earlier attempt to prove Lemma~\ref{lem:volSPep}(i)
and for providing the reference to Dobrowolski's work~\cite{Dobrowolski:Lawton}.
Vesselin Dimitrov gratefully acknowledges support
from the European Research Council via ERC grant GeTeMo 617129.
Philipp Habegger has received funding from the Swiss National Science
Foundation project n$^\circ$ 200020\_184623.

\section{Notation and preliminaries} \label{sec:notation}
Apart from the notation already introduced we use $\IN$ to denote the
natural numbers $\{1,2,3,\ldots\}$. If $x=(x_1,\ldots,x_m)$ with all
$x_i$ elements in an abelian group $G$ and if $A = (a_{i,j})_{i,j} \in
\mat{m,n}(\IZ)$ we write $x^A = (x_1^{a_{1,1}}\cdots
x_m^{a_{m,1}},\ldots,x_1^{a_{1,n}}\cdots x_m^{a_{m,n}})\in G^n$. So if
$B\in \mat{n,p}(\IZ)$, then $(x^{A})^{B} = x^{AB}$. For a commutative
ring $R$ with $1$ we let $R^\times$ denote its group of units. Euler's
function $\varphi$ maps $N\in\IN$ to the cardinality of
$(\IZ/N\IZ)^\times$. The group of all roots of unity in $\IC^\times$
is $\mu_\infty$. We often identify $\IG_m^d$ with the set
of its complex points $(\IC^\times)^d$ \refcomment{74}{and let $1$ denote the unit
element $(1,\ldots,1) \in\IG_m^d$.}
 If $\bzeta\in\IG_m^d$ is a torsion point, we write $\ord(\bzeta)$ for
its order. We write $\langle\cdot,\cdot\rangle$ for the Euclidean
inner product on $\IR^d$, $|\cdot|_2$ for the Euclidean norm on
$\IR^d$, and $|\cdot|$ for the maximum-norm on $\IR^d$
\refcomment{69}{and $\mat{m,n}(\IR)$}. 
We define $\log^+ x = \log \max\{1,x\}$ for all $x\ge 0$.

The constants implicit in Vinogradov's notation
$\ll_{x,y,z,\ldots},\gg_{x,y,z,\ldots},$ and in
$O_{x,y,z,\ldots}(\cdots)$ depend only on the values $x,y,z,\ldots$
appearing in the subscript.

Let $P\in \IC[X_1^{\pm 1},\ldots,X_d^{\pm 1}]\ssm\{0\}$, then
$|P|$ denotes the maximum-norm of the coefficient vector
of $P$ and we set $|0|=0$.
Recall that $m(P)$ is the Mahler measure of $P$. 
It follows from Corollaries 4 and 6 in Chapter 3.4~\cite{Schinzel}
that $\exp(m(P))$ is at most the Hermitian norm of the coefficient
vector of $P$. Suppose $P$ has at most $k\ge 1$ non-zero terms, we
find
\begin{equation}
  \label{eq:mPlogPub}
  m(P) \le \log|P| + \frac 12\log k. 
\end{equation}
The following result~\cite[Corollary 2]{DS:16} of Dobrowolski and Smyth
provides a reverse inequality of the same quality.

\begin{thm}[Dobrowolski--Smyth]
  \label{thm:dobsmyth}
Suppose $P\in \IC[X_1^{\pm 1},\ldots, X_d^{\pm 1}]\ssm\{0\}$ has at most
 $k\ge 2$ non-zero terms with $k$ an integer. Then
  $m(P)\ge \log|P| - (k-2)\log 2$. 
\end{thm}

Therefore,
\begin{equation}
  \label{eq:diffmahlerlog}
|m(P)-\log|P|| \ll k 
\end{equation}
with absolute implied constant. 
Observe that if
 $P$ is a polynomial, then 
$m(P) \ge \log |P| - \log(2)\sum_{i=1}^d \deg_{X_i}P$
by the classical Lemma 1.6.10~\cite{BG}. So (\ref{eq:diffmahlerlog})
is stronger when the number of terms in $P$ is known to be
bounded, which is often the case in our work. 

Let $x$ be an element of a number field $K$. The absolute logarithmic
Weil height, or just height, of $x$ is
\begin{equation}
  \label{def:heightx}
  \hproj{x} = \frac{1}{[K:\IQ]} \sum_{v} [K_v :\IQ_v] 
\log\max\{1,|x|_v\};
\end{equation}
here $v$ runs over all places of $K$ normalized such that $|2|_v=2$
for an infinite place $v$ and $|p|_v=1/p$ if $v$ lies above the rational
prime $p$, the completion of $K$ with respect to $v$ is $K_v$ and the
completion of $\IQ$ with respect to the restriction of $v$ is $\IQ_v$.
Let $P$ be a non-zero Laurent polynomial with  coefficients $x_0,\ldots,x_n\in K$.
The absolute logarithmic Weil height, or just height, of $P$ is 
\begin{equation}
\label{def:hprojP}
  \hproj{P} = \frac{1}{[K:\IQ]} \sum_{v} [K_v :\IQ_v] 
\log\max\{|x_0|_v,\ldots,|x_n|_v\}.
\end{equation}
See Chapter 1~\cite{BG} for
more details on heights. For example, $\hproj{x}$ and $\hproj{P}$ are
well-defined for $x\in\IQbar$ and $P\in
\IQbar[X_1^{\pm 1},\ldots,X_d^{\pm 1}]$, \textit{i.e.}, the values do
not depend on the number field $K$ containing $x$ and the coefficients
of $P$, respectively. Moreover 
$\hproj{P}=\hproj{\lambda P}$ for all $\lambda\in\IQbar^\times$.


\section{Quantitative Galois equidistribution for torsion points}  \label{quantequidistribution}

We need a strong enough quantitative version of the Galois
equidistribution of torsion points $\bzeta$ of
$\mathbb{G}_m^d$, with a power saving discrepancy  in
$\delta{(\boldsymbol{\zeta})}$ defined in (\ref{def:deltazeta}). 

Different approaches are possible and we opt to use the
Erd\"os--Tur\'an--Koksma bound. This reduces the problem to the
estimation of certain exponential sums, which happen to be Gau\ss{}
sums that can be explicitly evaluated.

Let $N\in\IN$. For a divisor $f\in\IN$ of $N$ we work with the
canonical surjective, homomorphism $\GammaN\rightarrow\Gammaf$ induced
by reducing modulo $f$.

The \textit{conductor} $\mf_G$ of a subgroup $G$ of $\GammaN$ is the
least positive integer $f\mid N$ such that $G$ contains
$\ker(\GammaN\rightarrow \Gammaf)$. Observe that
$[\GammaN:G]\le\varphi(\mf_G)$. 

Certainly, $\mf_G$ is well-defined  as
$\ker(\GammaN\rightarrow\GammaN)$ is the trivial subgroup. Moreover, $\mf_{\GammaN}=1$.
But one should take care \refcomment{25}{as the conductor} of $G=\{1\}$ is $N/2$ for $N\equiv 2\imod 4$. 

The group $\GammaN$ is naturally isomorphic to the Galois group of
$\IQ(\zeta)/\IQ$, where $\zeta$ is a root of unity of  order  $N$.
Let $L\subset \IQ(\zeta)$ be the fixed field of $G$. Then $L$ lies in the fixed field of
$\ker(\GammaN\rightarrow(\IZ/\mf_G\IZ)^\times)$ which equals
$\IQ(\zeta_{\mf_G})$ where $\zeta_{\mf_G}$ is a root of unity of   order $\mf_G$. 

Let $f\ge 1$ be an integer  and $\zeta_f$ of
 order $f$. We claim 
$L \subset \IQ(\zeta_f)$ if and only if $\mf_G \mid f$. 
Indeed, if the inclusion holds, then $L\subset\IQ(\zeta_f) \cap
\IQ(\zeta_{\mf_G})$. It is \refcomment{26}{well-known} 
that  the intersection is generated by a root of unity of
 order $\gcd(f,\mf_G)$. By minimality of $\mf_G$ we find
$\mf_G\mid f$.
The converse direction follows as 
$\IQ(\zeta_{\mf_G}) \subset \IQ(\zeta_f)$ if $\mf_G\mid f$.

So $\mf_G$ is the greatest common divisor of all $f$, for which
$L\subset \IQ(\zeta_f)$.
Equivalently $\mf_G$ is the greatest common divisor of all $f \mid N$,
for which $\ker(\GammaN \rightarrow\Gammaf)\subset G$.

By \refcomment{27}{class field theory}, $\mf_G$ is the finite part of the conductor of
the abelian extension $L/\IQ$.

The next lemma collects some classical facts on Gau\ss{} sums. We write
$\mf_\chi=\mf_{\ker\chi}$ for a character
$\chi:\GammaN\rightarrow\IC^\times$. We recall that $\be(\cdot)$ was
defined in (\ref{eq:defbex}).

\begin{lemma}
\label{lem:boundcharsum}
  Let $N\in\IN$ and say 
$\chi : \GammaN \rightarrow\IC^\times$ is a character. 
For $k\in\IZ$ we define $\tau = \sum_{\sigma\in\GammaN}
\chi(\sigma)\be(k\sigma /N)$, then the following hold true. 
\begin{enumerate}
\item [(i)] If $\gcd(k,N)=1$ then $|\tau|\le \mf_\chi^{1/2}$. 
\item[(ii)] For unrestricted $k$ we set $N' = N/\gcd(k,N)$. Then
  \begin{equation*}
    |\tau| \le \frac{\varphi(N)}{\varphi(N')} \mf_\chi^{1/2}. 
  \end{equation*}
\end{enumerate}
\end{lemma}
\begin{proof}
\newcommand{\sigmay}{\sigma}
  If $k=1$, part (i) follows directly
  from Lemma 3.1, Section 3.4~\cite{IwaniecKowalski}. The more general
  case $\gcd(k,N)=1$ follows as
  $\sum_{\sigma\in \GammaN} \chi(\sigma) \be(k\sigma/N)=\sum_{\sigmay\in \GammaN} \chi(k'\sigmay) \be(\sigmay/N)$
  where $kk'\equiv 1 \imod N$
  and since 
  $\chi$ is completely multiplicative. 
  
To prove (ii) set $N'=N/\gcd(k,N)$ and $k' = k/\gcd(k,N)$. Then
$\tau$ is
\begin{equation*}
 \sum_{\sigma\in\GammaN} \chi(\sigma)\be\left(\frac{k'}{N'}\sigma\right)
 = \sum_{\sigma'\in\GammaNp} 
\left(\sum_{\substack{\sigma\in\GammaN \\ \sigma\equiv \sigma'\imod {N'}}} \chi(\sigma)\right)\be\left(\frac{k'}{N'}\sigma\right).
\end{equation*}
The inner sum on the right
runs over a coset of the kernel of $\GammaN\rightarrow\GammaNp$.
 Since $\chi$ is a character, the
inner sum equals $0$ if the said kernel does not lie in the kernel
of $\chi$. In this case, $\tau=0$ and we are done. 

Otherwise, $\ker(\GammaN\rightarrow\GammaNp)\subset \ker \chi$,
and then $\mf_\chi  \mid N'$.
We find moreover that $\chi$
factors through  a character
$\chi' :\GammaNp\rightarrow\IC^\times$
and $\mf_{\chi'} \mid \mf_{\chi}$. 
As the kernel
of $\GammaN\rightarrow \GammaNp$ has order $\varphi(N)/\varphi(N')$
 we have
\begin{equation*}
  \tau = \frac{\varphi(N)}{\varphi(N')}\sum_{\sigma'\in\GammaNp}\chi'(\sigma')
\be\left(\frac{k'}{N'}\sigma\right). 
\end{equation*}
Part (ii) now follows from (i) since $\gcd(k',N')=1$. 
\end{proof}

\begin{lemma}
  \label{eq:ramsumbound}
  Let $N\in\IN$, let $G$ be a subgroup of $\GammaN$, and let
  $k\in\IZ$.
  We define $N'=N/\gcd(k,N)$, then
  \begin{equation*}
    \frac{1}{\#G}\left|\sum_{\sigma\in G} \be(k\sigma/N)\right|\le
    \frac{[\GammaN:G]}{\varphi(N')} \mf_G^{1/2}. 
  \end{equation*}
\end{lemma}
\begin{proof}
  Let $\chi'_1,\ldots,\chi'_m : \GammaN /G\rightarrow\IC^\times$ be all
  characters and $m=[\GammaN:G]$. 
Then $\sum_{i=1}^m \chi_i'(\sigma)=0$ for all $\sigma\in \GammaN/G$ expect for the
neutral element, where this sum equals $m$. Write $\chi_i$ for
$\GammaN\rightarrow \GammaN/G$ composed with 
$\chi'_i$. Then $\sum_{i=1}^m\chi_i(\sigma)=0$ if and only if 
$\sigma\in \GammaN\ssm G$, otherwise this sum is $m$. 
Therefore,
\begin{equation}
\label{eq:charfunctionG}
  \sum_{\sigma\in G} \be(k\sigma/N)
 = \frac 1m \sum_{i=1}^m   \sum_{\sigma\in \GammaN}\chi_i(\sigma) \be(k\sigma/N)
\end{equation}
and  lemma \ref{lem:boundcharsum}(ii) implies
\begin{equation*}
\left|\sum_{\sigma\in \GammaN} \chi_i(\sigma) \be(k\sigma/N)\right|
\le \frac{\varphi(N)}{\varphi(N')} \mf_{\chi_i}^{1/2}. 
\end{equation*}
Note that $G\subset\ker\chi_i$ because $\chi_i$ factors through 
$\GammaN\rightarrow\GammaN/G$. 
So $\mf_{\chi_i}\le \mf_G$, by the minimality of $\mf_{\chi_i}$. 
The current lemma now follows  from (\ref{eq:charfunctionG}). 
\end{proof}

\refcomment{29}{Let $d,n \in\IN$ and  $x_1,\ldots,x_n\in [0,1)^d$. 
The \textit{discrepancy} of $(x_1,\ldots,x_n)$ is
\begin{equation}
  \label{def:discrepancy}
 \cD (x_1,\ldots,x_n) = \sup_B \left|\frac{\#\{ i : x_i \in B\}}{n}
 - \mathrm{vol}(B)\right|
\end{equation}
where $B$ ranges over all products $\prod_{i=1}^d [\alpha_i,\beta_i)$
with $0\le \alpha_i <\beta_i\le 1$.} \refcomment{31}{Note that the
discrepancy lies in $[0,1]$.} In  some references such as \cite{Harman},
 the discrepancy is not
normalized by dividing by $n$ and can be greater than $1$.

In the next proposition we bound from above the discrepancy of the
Galois orbit of a point of finite order in $\IG_m^d$ using the Gau\ss{}
sum estimates above. 
Below,  $d_0(N)$ denotes the number of divisors of a natural number $N$.

\begin{propo}
\label{propo:discbound}
Let $\bzeta\in\IG_m^d$ have  order $N$ and let $G$ be a subgroup
of $\GammaN$ 
such that
$\{ \bzeta^\sigma : \sigma\in G\} = \{\be(x_i): 1\le i\le \#G\}$
with all $x_i$ in $[0,1)^d$.
\begin{enumerate}
\item [(i)] We have
\begin{equation*}
\cD(x_1,\ldots,x_{\# G}) \ll_d [\GammaN:G]\mf_G^{1/2} \frac{(\log
  2\delta(\bzeta))^{d-1} \log\log 3\delta(\bzeta)}{\delta(\bzeta)^{1/2}}.
\end{equation*}
\item[(ii)] If $d=1$, then
\begin{equation*}
  \cD(x_1,\ldots,x_{\#G}) \ll
[\GammaN:G]\mf_G^{1/2}
  \frac{\log(2N)d_0(N)}{\varphi(N)}.
\end{equation*}
\end{enumerate}
\end{propo}
\begin{proof}
  We abbreviate $n=\#G$. We fix $a\in\IZ^d$ with $\bzeta = \be(a/N)$.
  Then $N$ and the entries of $a$ are coprime. Let $H\ge 4$ be an
  integer. We use the Erd\"os--Tur\'an--Koksma inequality, Theorem
  5.21~\cite{Harman}, to bound the discrepancy $\cD = \cD
  (x_1,\ldots,x_{n})$ as follows
  \begin{equation}
    \label{eq:discboundgeneral}
    \cD \ll_d \frac{1}{H} + \sum_{\substack{b\in \IZ^d\ssm\{0\}
        \\ |b|\le H}} \frac{1}{r(b)} \left|\frac{1}{n}\sum_{\sigma\in
        G} \be\left(\frac{\langle a,b\rangle}{N}\sigma\right) \right|
  \end{equation}
  here $r(b_1,\ldots,b_d) = \max\{1,|b_1|\}\cdots\max\{1,|b_d|\}$.

  By Lemma \ref{eq:ramsumbound}, the expression inside the modulus is
  at most $C/\varphi(N/\gcd(\langle a,b\rangle,N))$ with
  $C=[\GammaN:G]\mf_{G}^{1/2}$. \refcomment{30}{We have $\varphi(M) \gg
    M / \log\log(3+M)$ for all integers $M\ge 1$ with an absolute and
    effective implicit constant, see for example Theorem 15~\cite{RS}.}
  Therefore,
  \begin{alignat*}1
    \cD &\ll_d
    \frac 1H +C\sum_{\substack{b\in \IZ^d\ssm\{0\} \\ |b|\le
        H}} \frac{1}{r(b)}\frac{\gcd(\langle
      a,b\rangle,N)}{N}\log\log(3+N/\gcd(\langle a,b\rangle,N)).
  \end{alignat*}

  If $b\in\IZ^d\ssm\{0\}$ with $|b|\le H$, then 
  \begin{equation*}
    \left\langle a,\frac {N}{\gcd(\langle a,b\rangle,N)}b\right\rangle
    = N\frac{\langle a,b\rangle}{\gcd(\langle a,b\rangle,N)} \in N\IZ
  \end{equation*}
  which implies $\bzeta^{bN/\gcd(\langle a,b\rangle,N)}=1$. So
  $N/\gcd(\langle a,b\rangle,N) \ge \delta/|b|>0$ where
  $\delta=\delta(\bzeta)$. 
  As $t\mapsto (\log\log(3+t))/t$ is decreasing on $t>0$ we
  find 
  \begin{alignat*}1 \cD &\ll_d \frac{1}{H}+ C\frac{1}{\delta}
    \sum_{\substack{b\in \IZ^d\ssm\{0\} \\ |b|\le H}}
    \frac{|b|}{r(b)}\log\log(3+\delta).
  \end{alignat*}
  \refcomment{31}{The sum of $|b|/r(b)$ over all $b \in\IZ^d$ with
    $1\le |b|\le H$ is $\ll_d H(\log H)^{d-1}$, so we find
    \begin{equation*}
      \cD \ll_d \frac{1}{H} +C \frac{\log\log(3\delta)}{\delta}H (\log
      H)^{d-1}.
    \end{equation*}
  }
  Part (i) follows by fixing $H$ to be the least integer with $H\ge
  \delta^{1/2}$ and $H\ge 4$.

  In part (ii) we have $d=1$. We may assume $N\ge 4$
  \refcomment{31}{as the discrepancy is at most $1$.} Here $a$ is
  coprime to $N$ and so $\gcd(ab,N)=\gcd(b,N)$. In
  (\ref{eq:discboundgeneral}) we take $H=N$ and use again Lemma
  \ref{eq:ramsumbound} with $C$ as before to find
  \begin{equation*}
    \cD \ll \frac 1N +  \sum_{b=1}^N \frac{C}{b \varphi(N/\gcd(b,N))}
    \refcomment{33}{\le}
    \frac 1N + \sum_{g\mid N}\frac {C}{g\varphi(N/g)} 
    \sum_{e =1}^{ N/g} \frac{1}{e}.
  \end{equation*}
  In the sum over $g$ we have $g\varphi(N/g)\ge \varphi(N)$ and the
  harmonic sum is $\ll \log N$. So $\cD \ll 1/N+ C (\log N)d_0(N)/
  \varphi(N)$, which implies (ii).
\end{proof}

A variant of the case $d=1$ already appears in Lemma
1.3~\cite{bakerihrumely}, it is attributed to Pomerance.

The discrepancy bound in (i) depends on $\delta(\bzeta)$. But
$\delta(\bzeta)$ is always bounded above by $N$. So estimates
involving $N$ are stronger than estimates involving $\delta(\bzeta)$.
However, there can be no upper bound for the discrepancy in terms of
the order $N$.

For $d=1$ we have $\delta(\bzeta)=N$. If $[\GammaN:G]$ and $\mf_G$
are fixed, the decay of the discrepancy is $1/N$ up-to terms of
subpolynomial growth. This fact will be important.

\refcomment{34}{
  The \textit{total variation} of a real valued function
$F$ whose domain contains the interval $[a,b]$ with $a\le b$ is
\begin{equation*}
  \mathrm{Var}_a^b(F) = \sup_{a\le x_0 \le \cdots \le x_m \le b}
  \sum_{i=1}^n |F(x_i)-F(x_{i-1})|.
\end{equation*}
For $a=0$ and $b=1$ we abbreviate
$\mathrm{Var}(F) = \mathrm{Var}_a^b(F)$.}

The next lemma requires Koksma's inequality.

\begin{lemma}\label{thm:bir}
  Let $F:[0,1]\rightarrow\IR$ be a function with
  $\mathrm{Var}(F)<\infty$. If
  $N\ge 1$ is an integer and $G$ is a subgroup of $\GammaN$ such that
  $\{ \zeta^\sigma : \sigma\in G\} = \{\be(x_i): 1\le i\le \#G\}$ with
  all $x_i$ in $[0,1)$, then
  \begin{eqnarray*}
    \left| \frac{1}{\#G} \sum_{i=1}^{\#G}
    F(x_i) - \int_{0}^1 F(x) dx \right|  \ll
    [\GammaN:G] \mf_G^{1/2}
    \frac{\log(2N) d_0(N)}{\varphi(N)}{\mathrm{Var}}(F).
  \end{eqnarray*}
\end{lemma}
\begin{proof}
  The claim follows from
  Theorems 1.3 and  5.1 in Chapter 2~\cite{KuipersNiederreiter}
  together with
  Proposition \ref{propo:discbound}(ii).
\end{proof}

\subsection{A univariate average}
\begin{lemma}
  \label{lem:VarFbound}
  \refcomment{35}{Let  $\alpha\in\IC$ and $r>0$. For $x\in [0,1]$ we define
  \begin{equation*}
    F_{\alpha,r}(x) = \log \max\left( r, |\be(x) -
      \alpha| \right).
  \end{equation*}
  Then
  $F_{\alpha,r}:[0,1]\rightarrow \IR$
  satisfies
  $\mathrm{Var}(F_{\alpha,r}) \le 3 \log(1+2/r)$.
  }
\end{lemma}
\begin{proof}
  We abbreviate $F = F_{\alpha,r}$.   
  By elementary geometry we can find $m\le 3$ and $0=x_0 < x_1<\cdots
  < x_m = 1$ 
   such that  $F$ is monotone on all $[x_{i-1},x_i]$.
   Then $\mathrm{Var}_{x_{i-1}}^{x_i}(F) =
  |F(x_{i})-F(x_{i-1})|$ and 
  $\mathrm{Var}(F) =\sum_{i=1}^m 
  \mathrm{Var}_{x_{i-1}}^{x_i}(F)$.
  We have $\log\max\{r,|\alpha|-1\}\le F(x)\le
  \log\max\{r , |\alpha|+1\}$ for all $x\in [0,1]$. Hence
  $\mathrm{Var}_{x_{i-1}}^{x_i}(F) \le \log
  \frac{\max\{r,|\alpha|+1\}}{\max\{r,|\alpha|-1\}}$ which we see is
  at most
  $\log(1+2/r)$ by considering the cases 
  $|\alpha|\ge 1+r$ and  $|\alpha|<1+r$.
  Thus $\mathrm{Var}(F) \le 3 \log(1+2/r)$.
\end{proof}
The value $r$ serves as a truncation parameter. We now  apply Koksma's
inequality to $F_{\alpha,r}$. 

\begin{lemma}  \label{logintegral}
  Let  $\zeta\in \mu_\infty$ have  order $N$
  and let $G$ be a subgroup of $\GammaN$. 
  Let $\alpha\in \IC$ and $r\in (0,1]$, then
  \begin{equation}\label{logintegralest}
    \frac{1}{\#G} \sum_{ \substack{ \sigma \in G\\
        |\zeta^{\sigma} - \alpha| > r  } } \log{|
      \zeta^{\sigma}-\alpha|}   = \log^+ |\alpha| +
    O\left(\left([\GammaN:G]\mf_G^{1/2}\frac{\log(2N)
          d_0(N)}{\varphi(N)}+  r\right)\left
        |\log \frac r 2\right|\right).
  \end{equation}
\end{lemma}
\begin{proof}
  We let $I$ denote the left-hand side of  (\ref{logintegralest}). Then
  $I=I_1+I_2$ with 
  \begin{equation*}
    I_1 = \frac{1}{\#G}\sum_{i=1}^{\# G}
    F_{\alpha,r}(x_i)\quad\text{and}\quad I_2=    \frac{1}{\#G} \sum_{\substack{i\\
        |\be(x_i) - \alpha| \le r  }}-\log r
  \end{equation*}
  with the $x_i\in [0,1)$ as in Lemma \ref{thm:bir} and $F_{\alpha,r}$
  as in Lemma \ref{lem:VarFbound}. 
  The integrals below
  are understood to be over subsets of $[0,1]$.  
  Applying Lemma \ref{thm:bir} to $F_{\alpha,r}$ yields 
  \begin{equation}
    \label{eq:galoisequiI}
    I_1 = \int_0^1  F_{\alpha,r}(x)  dx + 
    O\left([\GammaN:G]\mf_G^{1/2}\frac{\log(2N) d_0(N)}{\varphi(N)}
      \left|\log \frac r 2\right|\right).
  \end{equation}
  
  The set of $x\in [0,1]$ with 
  $|\be(x)-\alpha|\le r$ is 
  of the form $\emptyset, [a,b],$ or $[0,a]\cup [b,1]$.
    So its characteristic function 
  has total
  variation at most $2$. Lemma \ref{thm:bir} applied to this
  characteristic function yields
  \begin{equation}
    \label{eq:galoisequiIb}
    I_2 = 
    -\int_{|\be(x)-\alpha| \le r} \log r dx
    + 
    O\left([\GammaN:G]\mf_G^{1/2}\frac{\log(2N) d_0(N)}{\varphi(N)}
    \right).
  \end{equation}
  The sum of the integrals in (\ref{eq:galoisequiI}) and
  (\ref{eq:galoisequiIb}) equals
  \begin{equation*}
    \int_{0}^1  \log| \be(x)-\alpha|dx
    - \int_{|\be(x)-\alpha|\le r} \log|\be(x)-\alpha| dx.
  \end{equation*}
  Jensen's formula,
  \refcomment{38}{Proposition 1.6.5~\cite{BG},} implies that the
  first integral equals $\log^+|\alpha|$. To complete the proof it suffices to show that the second integral is
  $O(r|\!\log r/2|)$.

\refcomment{39}{The integral is non-positive as
    $r\le 1$  and
    we may assume that it is non-zero. First assume, $|\alpha|\le 1/2$, then
    $r\ge 1/2$. In this case $|\be(x)-\alpha|\ge 1/2$ and the
    integral is $O(r|\!\log r/2|)$. Second, say $|\alpha|>1/2$.
    Lemma 11.6.1~\cite{RahmanSchmeisser}
    implies  $|\be(x)-\alpha|\ge |\alpha|^{1/2}
    |\be(x)-\be(y)| \ge 2^{-1/2} |\be(x-y)-1|$ where $\alpha = |\alpha|\be(y)$ and
    $|x-y|\le 1/2$. There is an absolute and effectively computable
    constant $C>0$ with $|\be(x-y)-1|\ge C |x-y|$ and thus
    $|\be(x)-\alpha|\ge 2^{-1/2} C |x-y|$. In the integral we have
 $r\ge |\be(x)-\alpha|$ and so the desired bound follows from
    elementary analysis.} 
\end{proof}

\subsection{A Galois conjugate near $1$} We will also need an estimate
on the minimal distance of  a Galois conjugate  of a torsion point to
the unit element.

\begin{lemma}  \label{approx}
Let  $\bzeta\in\IG_m^d$ have 
order $N$ and let $G$
be a subgroup of $\GammaN$. 
There exist $\sigma \in G$
and $a\in\IZ^d$
with $\bzeta = \be(a\sigma /N),|a|<N,$ and
\begin{equation} 
\label{approximation}
\frac{|a|}{N} \ll_d \frac{[\GammaN:G]^{1/d} \mf_G^{1/(2d)}}
 {\delta(\boldsymbol{\zeta})^{1/(3d)}}.
\end{equation}
\end{lemma}
\begin{proof}
Let $\bzeta = \be(b/N)$ with $b\in\IZ^d$, the entries of $b$ and $N$ have no
common prime divisor. Suppose $x_1,\ldots,x_n$ are as in
Proposition \ref{propo:discbound} coming from the $\bzeta^\sigma$ as
$\sigma$ ranges over $G$ where $n=\# G$. There exists
$c(d)>0$ depending only on $d$ with $\cD(x_1,\ldots,x_n)\le c(d)
[\GammaN:G] \mf_G^{1/2} \delta(\bzeta)^{-1/3}$. We set $\kappa
=2c(d)^{1/d} [\GammaN:G]^{1/d}
\mf_G^{1/(2d)} \delta(\bzeta)^{-1/(3d)}$.
There is nothing to show if $\kappa\ge 1$. Otherwise, by the
 definition of the discrepancy
 the hypercube $[0,\kappa)^d$
  contains some  $x_i=a/N$. Hence $a$ satisfies, $|a|<N$,
  (\ref{approximation}), and $\be(a/N) = \bzeta^{\sigma^{-1}}$
  for some $\sigma\in G$. 
\end{proof}


\section{Theorem of Mahler--Mignotte}  \label{mahler}

In this section, we firstly establish 
the separation of pairs of roots of an integer polynomial. 
Theorem \ref{mahlerextended} below was shown by
Mahler~\cite{Mahler:DiscIneq} for the case $k = 1$ of a single pair of roots.
Mignotte~\cite{Mignotte:95} generalized Mahler's inequality to products over
several disjoint pairs of roots (see his
Theorem 1). We reproduce here a lightened version of Mignotte's theorem
that is suitable for our needs.
The proof is an adaptation of Mahler's original argument about a single pair, guided by the principle that
Liouville's Inequality bounds an algebraic number at an arbitrary set
of places in terms of the height. 
Let us also mention G\"uting's
proof~\cite{Gueting:61} of a less precise earlier result involving
the length of a polynomial instead of the Mahler measure.

Let $Q\in\C[X]$ be a non-zero univariate polynomial. By Jensen's formula its Mahler measure 
equals
\begin{equation}
\label{eq:mahlerjensen}
m(Q) =\log|a_0| + \sum_{i=1}^D \log^+|z_i|
\end{equation}
if $Q=a_0(X-z_1)\cdots(X-z_D)$ and where the $z_i$ are complex.
If $Q$ is non-constant, we let $\disc{Q}$ denote its discriminant as a
degree $\deg Q$ polynomial.

\begin{thm}  \label{mahlerextended}
Let $Q \in \C[X] \setminus \C$ be of degree $D$
and with no repeated roots. 
If
$z_1,\ldots,z_k, z_1',\ldots,z_k'$ are pairwise distinct complex roots of $Q$, then
 \begin{equation} \label{subproduct} \sum_{j=1}^k
- \log|z_j-z'_j| \le 
  \frac {D+2k}{2} \log D - \frac{k}{2}\log 3
+ (D-1) m(Q) - \frac{1}{2} \log |\disc{Q}|
\end{equation}
with strict inequality for $k\ge 1$. 
\end{thm}

\begin{proof}
We modify Mahler's argument as follows. 

 Both sides of (\ref{subproduct}) are invariant under multiplication $Q$ by
a non-zero scalar. So we may assume that $Q$ is monic.  After
possibly swapping $z_j$ with $z'_j$ we 
may assume $|z_j|\ge |z'_j|$ for all $j$.



We augment $z_1,\ldots,z_k$ to all complex
roots $z_1,\ldots,z_D$ of  $Q$. Then we consider the
Vandermonde determinant
$$
V = \det \left( \begin{array}{cccc}
              1 & 1 & \ldots & 1 \\
              z_1 & z_2 & \ldots & z_D \\
              \vdots & \vdots &  & \vdots \\
              z_1^{D-1} & z_2^{D-1} & \ldots & z_D^{D-1}
            \end{array} \right),
$$
which is non-zero as $z_1,\ldots,z_D$ are pairwise distinct.
For $j \in \{ 1,\ldots,k\}$, let $i_j>k$ be the index with $z_j' =
z_{i_j}$.  For these $j$, we subtract
 the $i_j$-th column from the $j$-th column
 and factoring each difference $z_j - z_{i_j}$ out of the
 determinant with the identities $z_j^m - z_{i_j}^m = (z_j -
 z_{i_j})(z_j^{m-1} + z_j^{m-2}z_{i_j} + \cdots + z_{i_j}^{m-1})$,
 $1 \leq m \leq D-1$.
We obtain an expression
\begin{equation} \label{detquotient}
V=W\prod_{j=1}^k (z_j - z_{i_j})=W\prod_{j=1}^k (z_j - z'_{j}),
\end{equation}
where $W\not=0$ is the determinant of the matrix having
$$
\left( \begin{array}{c}
         0 \\
         1 \\
         z_j + z'_{j} \\
         \vdots \\
         z_j^{D-2} + z_j^{D-3}z'_{j} + \cdots + {z'_{j}}^{D-2}
       \end{array} \right)
$$
for its $j$-th column, $j \in  \{1,\ldots,k\}$, and the same entries as in
the Vandermonde matrix in the remaining columns. 
By Hadamard's inequality, $|W|$ is bounded from above by the product of the
Hermitian norms of all these columns. The $j$-th column, for some $j
\in \{1,\ldots,k\}$, has Hermitian norm
\begin{alignat*}1
\sqrt{\sum_{m=0}^{D-2} |z_j^m + z_j^{m-1}z'_{j}+\cdots + {z'_{j}}^m|^2} 
&\leq\sqrt{\sum_{m=0}^{D-2} (m+1)^2 } \max\{1,|{z_j}|,|{z'_j}|\}^{D-2}
 \\
& <
\sqrt{D^3/3}\cdot \max\{1,|{z_j}|\}^{D-1}  
\end{alignat*}
where we used  $|z'_j|\le |z_j|$. 
The Hermitian norm of the $j$-th column with  $j \in \{k+1,\ldots,D\}$ 
is at most $\sqrt D \max\{1,|z_j|\}^{D-1}$. 

Applying Hadamard's inequality, using these two bounds, 
and taking the logarithm 
yields
\begin{alignat*}1
\log  |W| &\le \frac k2 \log\left(\frac{D^3}{3}\right)
+\frac{D-k}{2}\log D 
+(D-1)\sum_{j=1}^D \log^+|z_j|\\
&= \frac {D+2k}{2} \log D - \frac{k}{2}\log 3
+ (D-1) m(Q)
\end{alignat*}
\refcomment{40}{as $Q$ is monic; the inequality is strict for $k\ge
  1$. If $k=0$ no column operations are necessary and $V=W$.}

The squarefree polynomial $Q$ has discriminant 
 $\disc{Q} = V^2$. Consequently $|V| = |\disc{Q}|^{1/2}$, and  in
 view of (\ref{detquotient}) we have
\begin{alignat*}1
\sum_{j=1}^k -\log|z_j-z'_j| &= \log|W|-\log|V|\\
&\le \frac {D+2k}{2} \log D - \frac{k}{2}\log 3
+ (D-1) m(Q) - \frac{1}{2} \log |\disc{Q}|
\end{alignat*}
\refcomment{40}{with a strict inequality for $k\ge 1$.} This concludes the proof. 
\end{proof}

While Theorem~\ref{mahlerextended} suffices for our needs here, we remark
that it is possible to relax the
hypothesis to having $z_1,\ldots,z_k$ pairwise distinct and
$\{z_1,\ldots,z_k\}\cap\{z'_1,\ldots,z'_k\}=\emptyset$, at the cost of a slightly worse upper bound (\ref{subproduct}).

The following corollary holds for integral polynomials that are not
necessarily squarefree.

\begin{corol}
\label{cor:genmahler}
    Let $Q \in \Z[X] \setminus \Z$ be of degree $D$.
If
$z_1,\ldots,z_k, z_1',\ldots,z_k'$ are pairwise distinct complex roots of $Q$, then
  \begin{alignat}1
\label{eq:genmahlerbound}
  \sum_{j=1}^k - \log|z_j-z'_j| &\le 
\frac{D+2k}{2} \log D
-\frac k2 \log 3+ (D-1)m(Q)
  \end{alignat}
with strict inequality for $k\ge 1$. 
\end{corol}
\begin{proof}
Again we may assume $k\ge 1$. 
We begin by splitting off the squarefree part of $Q$.
More precisely, we  factor $Q=\widetilde QR$ where $\widetilde Q,R\in \IZ[X]$ and
$\widetilde Q$ is
squarefree 
and vanishes at all complex roots of $Q$. 
The discriminant $\disc{\widetilde Q}$
is a non-zero integer, and so $|\disc{\widetilde Q}|\ge 1$.
Moreover, $m(\widetilde{Q})\ge 0$. 
Theorem \ref{mahlerextended}  applied to $\widetilde Q$
 and $1\le \deg \widetilde Q\le D$
imply that the sum on the left of (\ref{eq:genmahlerbound}) 
is at most $\frac 12 (D+2k) \log D - \frac k2 \log 3 + (D-1)m(\widetilde Q)$. 
The corollary follows from $m(\widetilde Q) = m(Q)-m(R)\le m(Q)$. 
\end{proof}

\subsection{A repulsion property of the unit circle}
A key point in~\cite{hab:gaussian} is that while Mahler's theorem does not
give a strong enough
 bound for the distance of a complex root of $Q \in \Z[X]
\setminus \{0\}$ to an $N$-th root of unity (the product
$(X^N - 1) Q(X)$ has an exceedingly large degree), it can be used to bound the distance from the
unit circle to the locus of roots of $P$ lying off the unit circle. 
With  Corollary~\ref{cor:genmahler}, this repulsion property of the unit circle can be strengthened as follows.

\begin{lemma} \label{lem:repulsion}
Let $Q \in \Z[X] \setminus \Z$   and $Q = a_0(X-z_1)\cdots(X-z_D)$ where $z_1,\ldots,z_D\in\IC$. 
Then
\begin{equation}
  \label{circlebound}
\begin{aligned}
\sum_{\substack{j=1\\|z_j|\not=1}}^D \log^+ \frac{1}{\bigl||z_j|-1\bigr|} &\leq 
D\log\left(\frac{3+\sqrt{5}}{2}\right) + 2D\log(2D) + 4D m(Q) \\
&\leq 
4D \bigl(\log(2D) + m(Q)\bigr). 
\end{aligned}
\end{equation}
\end{lemma}
Before we come to the proof let us remark that 
 $||z|-1|$ is the distance $\dist{z,S^1}$ of $z\in\C$  to the unit circle
$S^1$. Thus inequality (\ref{circlebound}) can be restated as
 providing 
\begin{equation*}
\frac
 1D \sum_{\substack{j=1\\|z_j|\not=1}}^D \log^+
 \frac{1}{\dist{z_j,S^1}}
\le \log\left(\frac{3+\sqrt{5}}{2}\right) + 2\log(2D) + 4 m(Q).
\end{equation*}
Our result suggests that the unit circle repels roots of $Q$ that lie
off the unit circle.
Related estimates are implicit in work of
Dubickas~\cite{Dubickas:96}, \textit{cf.} his Theorem 2.
\begin{proof}
  The second bound in (\ref{circlebound}) is elementary, so it
  suffices to prove the first one.

  \refcomment{42}{
    Say $Q=a_0 Q_1\cdots Q_n$ where each $Q_i\in
    \IZ[X]$ is irreducible  of positive degree with $a_0\in\IZ$. 
    Observe that $m(Q_i)\le m(Q)$ and $\sum_i \deg Q_i = \deg
    Q$.
    So it suffices to prove (\ref{circlebound}) for $Q$ irreducible in
    $\IZ[X]$. We  may also assume $Q(0)\not=0$.
  }

  We will apply Corollary~\ref{cor:genmahler} to the polynomial
  $\widetilde Q\in\IZ[X]$ constructed from $Q$ in the following manner.
  If $Q(1/X)X^D \not= \pm Q$ we take $\widetilde Q = Q(X)Q(1/X)X^D$ and
  $\widetilde Q = Q$ otherwise. So $\widetilde D = \deg\widetilde Q =
  \delta D$ and $m(\widetilde Q)=\delta m(Q)$ with $\delta=2$ in the
  first case and $\delta=1$ in the second case. For any root $z$ of
  $\widetilde Q$ we also have $\widetilde Q(1/\overline z)=0$.

  The following basic observation for a complex number $z$ will prove
  useful. We have $|z-1/\overline{z}|\le 1$ if and only if $\phi^{-1}
  \le |z| \le \phi$ with $\phi = (1+\sqrt{5})/2$ the golden ratio.

  Let $w_1,\ldots,w_k$ be the roots of $\widetilde Q$ without
  repetition such that $\phi^{-1} \le |w_j|< 1$. Then
  $w'_j=1/\overline{w_j}$ is a root of $\widetilde Q$ for each $j\in
  \{1,\ldots,k\}$ with $|w'_j|>1$. Corollary~\ref{cor:genmahler} yields
  \begin{equation}
    \label{circleboundprep}
    \sum_{j=1}^k \log^+\frac{1}{\bigl|w_j -
      1/\overline{w_j}\bigr|} \le
    {\delta D}\log(\delta D) +
    \delta^2 D m(Q)
  \end{equation}
  because $k\le \widetilde D /2=\delta D/2$ and $m(Q)\ge 0$.

  Suppose $z_j$ is a root of $Q$ with $|z_j|\not=1$ and
  $\phi^{-1} \le |z_j|\le \phi$.
  Then 
  $z_j\in \{w_l,1/\overline{w_l}\}$ for some unique $l$.
  The mapping $j\mapsto l$ is at worst $2$-to-$1$
  and injective if $\delta=2$ as $Q$ is irreducible. \footnote{Indeed, if
    $z_j,z_k \in \{w_l,1/\overline{w_l}\}$ with $z_j\not=z_k$,
    then $z_j = 1/\overline{z_k}$. So $\widetilde Q=Q$ and hence $\delta=1$ in this case.}
  This leads to the factor $2/\delta$ in 
  \begin{equation}
    \label{circleboundprep2}
    \sum_{\substack{|z_j|\not=1\\1/\phi\le |z_j|\le \phi}}  \log^+\frac{1}{\bigl|z_j-1/\overline{z_j}\bigr|}
    \le \frac{2}{\delta}\sum_{l=1}^k \log^+\frac{1}{\bigl|w_l-1/\overline{w_l}\bigr|}
  \end{equation}

  For a complex number $z$ with $|z|\ge \phi^{-1}$ we have $|z -
  1/\overline{z}| = \frac{|z|+1}{|z|}\bigl||z| - 1\bigr|  \le (1+\phi)\bigl||z|-1\bigr|$.
  This allows us to get
  \begin{equation*}
    \sum_{\substack{|z_j|\not=1\\ 1/\phi\le |z_j|\le \phi}}  \log^+ \frac{1}{\bigl||z_j|-1\bigr|}
    \le s \log (1+\phi) +
    \sum_{\substack{|z_j|\not=1\\1/\phi\le |z_j|\le \phi}}
    \log^+\frac{1}{|z_j-1/\overline z_j|}
  \end{equation*}
  where $s$ is the number of terms in the first sum.
  There are at most $D-s$ other roots of $Q$ and if
  $|z_j|<\phi^{-1}$ or $|z_j|>\phi$ we get
  $\log^+1/\bigl||z|-1\bigr|\le\log (1+\phi)$. Together with
  (\ref{circleboundprep}) and (\ref{circleboundprep2}) we find
  \begin{equation*}
    \sum_{|z_j|\not=1}\log^+ \frac{1}{\bigl||z_j|-1\bigr|}
    \le D\log (1+\phi) + 2D\log({\delta D})  + 2\delta Dm(Q).
  \end{equation*}
  We have established (\ref{circlebound})  
  for $Q$ as $\delta\le 2$.
\end{proof}

Next we generalize our bound to a polynomial with coefficients in a
number field. Recall that $\hproj{Q}$ is the absolute logarithmic
projective Weil height of a non-zero polynomial $Q$ with algebraic
coefficients.

\begin{corol}
\label{cor:repulsionmult}
Let $F\subset\IC$ be a number field and
let $Q \in F[X] \setminus F$   and $Q = a_0(X-z_1)\cdots(X-z_D)$ where $z_1,\ldots,z_D\in\IC$. Then
\begin{equation*}
\sum_{\substack{j=1\\ |z_j|\not=1}}^D  \log^+ \frac{1}{\bigl||z_j|-1\bigr|}
\le 10D[F:\IQ]^2\bigl(\log(2D) + \hproj{Q}\bigr).
\end{equation*}
\end{corol}
\begin{proof}
Let $\widetilde Q$ be the product of the  $\IQ$-Galois conjugates of $Q$.
Then $\widetilde Q$ has rational coefficients
and degree  $\widetilde D\le D[F:\IQ]$. Let $\lambda\in\IN$ such that $\lambda\widetilde Q$ is integral with  content
$1$. For the projective height we find
$\hproj{\widetilde Q} = \log|\lambda \widetilde Q|$.
Together with Lemma 1.6.7~\cite{BG} we get $m(\lambda\widetilde Q) \le \frac 12
\log(1+\widetilde D) + \hproj{\widetilde Q}$.
As all $\IQ$-Galois conjugates of $Q$ have the same projective height
we use elementary estimates at local places to find
\begin{equation*}
\hproj{\widetilde Q} \le [F:\IQ]\log(1+D) + [F:\IQ]\hproj{Q}.  
\end{equation*}
By
Lemma \ref{lem:repulsion} applied to $\lambda\widetilde Q$, the sum $\sum_{j=1 : |z_j|\not=1}^D\log^+ 1/{\bigl||z_j|-1\bigr|}$ is at most 
\begin{equation*}
4\widetilde D\left(\log(2\widetilde D)+\frac
12 \log(1+\widetilde D) + [F:\IQ]\log(1+D)+[F:\IQ]\hproj{Q}\right). 
\end{equation*}
 We use 
$1+\widetilde D\le 2\widetilde D\le 2 D [F:\IQ] \le (2D)^{[F:\IQ]}$
to
complete the proof. 
\end{proof}

\subsection{Averages over roots of unity}

In this subsection we apply the repulsion property of the unit circle,
Corollary~\ref{cor:repulsionmult}, to estimate the norm of cyclotomic
integers of the form $Q(\zeta)$, \refcomment{43}{where $\zeta$ is a varying root of
unity and $Q$} is a moderately controlled univariate polynomial with
algebraic coefficients and without zeros in
$S^1 \setminus \mu_{\infty}$. This gives a fairly uniform solution of
the one dimensional \essatoral{} case and forms the basis for the
higher dimensional case to be taken up in the next sections.

\begin{propo}  \label{univariatebound}
  Let $F\subset\IC$ be  a number field and
let $Q \in F[X] \ssm \{0\}$ be of
 degree at most $D\ge 1$ with no roots in  $S^1\ssm\mu_\infty$.
Let $\zeta\in\mu_\infty$ be of  order $N$ and
 $G$  a subgroup of $\GammaN$
such that $Q(\zeta^\sigma)\not=0$ for all 
$\sigma \in G$. Then 
  \begin{equation}
  \begin{aligned}
\label{meanvalue2}
\frac{1}{\# G} \sum_{\sigma \in  G}
    &\log{|Q(\zeta^{\sigma})|}  =m(Q)  + \\
&O\left([F:\IQ]^2[\GammaN:G] \mf_G^{1/2}
D(\log(2D)+\hproj{Q})\frac{(\log 2N)^3 d_0(N)}{N}
\right).
\end{aligned}
\end{equation}
\end{propo}
\begin{proof}   
We may assume that $Q$ is non-constant and $D=\deg Q$. Let $Q =
a_0(X-z_1) \cdots (X-z_D)$. The idea is that each given root $z_j$ may
get within distance of \refcomment{49}{$\le 1/N^2$} to at most a
single conjugate of $\zeta$.


We call $z_j$ exceptional if $|\zeta^{\sigma_j}-z_j|\leq 1/N^2$ for
some $\sigma_j\in G$. As $|\xi-\xi'|\ge 4/N$ for distinct roots of
unity $\xi,\xi'$ of order $N$ we see that $\sigma_j$ is uniquely
determined by $z_j$. Note that $\zeta^{\sigma_j}\not=z_j$ because
$Q(z_j)=0\not= Q(\zeta^{\sigma_j})$.

We  apply Lemma~\ref{logintegral} with $\alpha = z_j$ and
\refcomment{49}{$r = 1/N^2$}.   Thus
\begin{equation}
\label{eq:sumtauG}
\begin{aligned}
& \frac{1}{\#G} \sum_{\sigma \in  G} \log{|\zeta^{\sigma}-z_j|}
 \\
& \quad=
\log^+{|z_j|} +  \frac{1}{\#G}\log{|\zeta^{\sigma_j}-z_j|} + O\left(
[\GammaN:G] \mf_G^{1/2}\frac{(\log 2N)^2 d_0(N)}{\varphi(N)} + \frac{\log 2N}{N^2}
    \right),
    \end{aligned}
\end{equation}
\refcomment{44}{ if $z_j$ is exceptional, otherwise
  the same bound without the term 
 $(\#G)^{-1}\log|\zeta^{\sigma_j}-z_j|$
holds true.}
As $1/N^2 \le 1/\varphi(N)$ we \refcomment{45}{merge} $(\log 2N)/N^2$ into the first term
of the error term.
 Summing (\ref{eq:sumtauG})
over all $j\in \{1,\ldots,D\}$ and adding
$\log|a_0|$ gives
\begin{equation}
\label{eq:univariateboundquasifinal}
 \frac{1}{\#G}\sum_{\sigma\in G} \log|Q(\zeta^\sigma)|= m(Q) - \frac{1}{\#G}\sideset{}{'}\sum_{j=1}^D\log\frac{1}{|\zeta^{\sigma_j}-z_j|}
 + O\left([\GammaN:G] \mf_G^{1/2}D\frac{(\log 2N)^2 d_0(N)}{\varphi(N)}\right),
\end{equation}
the dash signifies that we only sum over those $j$ for which
$z_j$ is exceptional.  

To bound the dashed sum we require Corollary~\ref{cor:repulsionmult}.
If $z_j$ is exceptional, then $|\zeta^{\sigma_j}-z_j|\le 1$.
Therefore, the dashed sum is non-negative.

\refcomment{47}{First, we consider the subsum over all exceptional $z_j\not\in\mu_\infty$.} Then
$|z_j|\not=1$ and
$|\zeta^{\sigma_j}-z_j|\ge ||z_j|-1|$ by the reverse triangle
inequality. 
By Corollary \ref{cor:repulsionmult} we find
\begin{equation}
\label{eq:univariateboundfinal2}
\begin{aligned}
  0&\le  \sideset{}{'}\sum_{\substack{j=1 \\ z_j\not\in\mu_\infty}}^D \log\frac{1}{|\zeta^{\sigma_j}-z_j|}
  \le  \sideset{}{'}\sum_{\substack{j=1 \\ z_j\not\in\mu_\infty}}^D \log^+\frac{1}{||z_j|-1|}
  = O\left([F:\IQ]^2 D (\log(2D) +  \hproj{Q})\right).
\end{aligned}
\end{equation}

\refcomment{47}{Second, we consider the subsum over all exception
 $z_j\in\mu_\infty$, which is harmless.}
\refcomment{48}{Recall that $\zeta^{\sigma_j}\not=z_j$.}
Since the  order of $z_j$ is $\ll
[\IQ(z_j):\IQ]^{2} \le (D[F:\IQ])^2$
and the order of $\zeta^{\sigma_j}$ is $N$
we find
$|\zeta^{\sigma_j}-z_j|\gg N^{-1}(D[F:\IQ])^{-2}$.
On the other hand,
$|\zeta^{\sigma_j}-z_j| \le N^{-2}$
and hence $N\ll (D[F:\IQ])^2$.
We obtain the crude estimate
$|\zeta^{\sigma_j}-z_j|\gg (D[F:\IQ])^{-4} \gg (2D)^{-4[F:\IQ]}$
and finally bound the at  most $D$ terms below separately
 to get
\begin{equation}
\label{eq:univariateboundfinal3}
  0\le  \sideset{}{'}\sum_{\substack{j=1 \\ z_j\in\mu_\infty}}^D \log\frac{1}{|\zeta^{\sigma_j}-z_j|}
  =  O\left( [F:\IQ]D\log(2D) \right).
\end{equation}

We divide the sum of (\ref{eq:univariateboundfinal2}) and
(\ref{eq:univariateboundfinal3}) by $\#G$ to find
\begin{equation*}
0 \le \frac{1}{\#
G} \sideset{}{'}\sum_{j=1}^D \log \frac{1}{|\zeta^{\sigma_j}-z_j|}
= O\left([F:\IQ]^2  D(\log(2D)+\hproj{Q})\frac{[\GammaN:G]}{\varphi(N)}\right).
\end{equation*}
The proposition follows from (\ref{eq:univariateboundquasifinal})
and $\varphi(N) \gg N / \log\log(3N)$, \refcomment{50}{a consequence of
Theorem 15~\cite{RS}.}
\end{proof}

Proposition~\ref{univariatebound} and ultimately Theorem
\ref{mahlerextended} may be viewed as our input from transcendence
theory. If this or a comparable bound held without the restrictive
condition that $Q$ has no roots in $S^1 \setminus \mu_{\infty}$ then
it could be used to attack Conjecture~\ref{conj:galoisallP}. We were
unable to prove or disprove that a suitable version of
Theorem~\ref{univariatebound} extends to general polynomials.
\refcomment{51}{Progress on
  Conjecture~\ref{conj:linearforms} could indicate a path towards this
  goal.}


\section{Geometry of numbers}  \label{sec:lattice} 

Let $d\ge 1$ and suppose $\bzeta\in\IG_m^d$ has order $N$. It would be
useful if $\bzeta$ had a Galois conjugate close to the unit
element $1$. If the  distance were at most a small
power of $N^{-1}$, this conjugate could be used to help reduce the
multivariate Theorem \ref{thm:main} to the univariate
Proposition \ref{univariatebound}, \textit{cf.}~\cite{hab:gaussian}.

Unfortunately, such a conjugate need not exist. Take for example
$\bzeta = \be(1/p,1/p^n)$ where $p$ is a prime and $n\in\IN$, here $N
= p^n$. Any conjugate of $\bzeta$ has distance $\gg 1/p$ to $1$
regardless of the value of $n$.
The problem is that $\bzeta$ is up-to a point of order $p$ contained
in the algebraic subgroup $\{1\}\times\IG_m$. 

We overcome this difficult by constructing a factorization $\bzeta
= \bfeta \bxi$ into torsion points $\bfeta$ and $\bxi$ that satisfy
the following properties for prescribed $\epsilon >0$. First, the
order of $\bfeta$ is small relative to $N$, more precisely it is
$O_{d,\epsilon}(N^\epsilon)$. Second, some Galois conjugate of $\bxi$
is at distance at most $O_{d,\epsilon}(N^{-\kappa(\epsilon)})$ to $1$.
Here $\kappa(\epsilon)$ is expected to be small for small $\epsilon$,
but we will see that $\kappa(\epsilon)/\epsilon$ is large. This is of
central importance for our application.

We use the \refcomment{18,52}{geometry of numbers}
to construct this factorization.
\refcomment{19,53}{An important tool is the slope of a lattice.}

A lattice $\Lambda$ in $\IR^d$ is a finitely 
generated and discrete subgroup of $\IR^d$.
The rank of $\Lambda$ is denoted by $\rk{\Lambda}$ and its determinant
by $\det(\Lambda)$. We
 consider the set
\begin{equation*}
A=  \{(r, \log\det(\Omega)) : r\in\IZ \text{ and }\Omega\text{ is a
    subgroup of $\Lambda$ with $\rk{\Omega}=r$}\}
\end{equation*}
and use the convention $\det(\{0\})=1$. 
In constrast to the convention in  Arakelov theory, we have no sign in front of
$\log\det(\Lambda)$. Observe that the second coordinate is bounded from
below on $A$. 
In Proposition 1, Stuhler \cite{stuhler:quadform} proved that for each $j\in \{0,\ldots,\rk{\Lambda}\}$ there
exists a lattice $\Lambda_j\subset \Lambda$, possibly non-unique,
 with $\log\det(\Lambda_j)$ minimal.
The lower boundary of the convex hull of $A$ is the graph of a piece-wise linear, continuous,
convex function \refcomment{55}{$P : [0,\rk{\Lambda}] \rightarrow\IR$}.
As $\Lambda_0 = \{0\}$ and $\Lambda_{\rk{\Lambda}}=\Lambda$ 
 we find $P(0) = 0$ and $P(\rk{\Lambda}) = \log
\det(\Lambda)$. 

For each $j\in \{1,\ldots,\rk{\Lambda}\}$, the slope of $P$ on
$[j-1,j]$ is
\begin{equation*}
  \mu_j(\Lambda) = P(j)-P(j-1). 
\end{equation*}
By convexity we have
\begin{equation*}
  \mu_1(\Lambda) \le \mu_2(\Lambda) \le\cdots\le\mu_{\rk{\Lambda}} (\Lambda).   
\end{equation*}
Moreover, $\mu_1(\Lambda) + \cdots + \mu_{j}(\Lambda) =
P(j)-P(0)=P(j)$ 
for all $j$ as $P(0)=\log \det(\Lambda_0)=0$.

Assume $\Lambda\not=\{0\}$ and
 let $\nu \in (0,1/2]$ be a parameter.
  Suppose that
\begin{equation*}
 \mu_j(\Lambda)  < \nu^{\rk{\Lambda} - j + 1} \log\det(\Lambda)
\end{equation*}
for all $j\in \{1,\ldots,\rk{\Lambda}\}$. Taking the sum yields
\begin{equation*}
  \log\det(\Lambda) < (\nu + \nu^2+\cdots + \nu^{\rk{\Lambda}}) \log\det(\Lambda). 
\end{equation*}
As $\nu\in (0,1/2]$ we must have $\det(\Lambda)  <  1$.

Let us now assume $\det (\Lambda) \ge 1$, then there exists
\refcomment{59}{a unique} \refcomment{58}{$j_0 \in
\{0,\ldots,\rk{\Lambda}-1\}$} such that
\begin{alignat}1
\label{eq:mujplus1}
  \mu_k(\Lambda) &< \nu^{\rk{\Lambda} - k + 1}\log\det(\Lambda)\text{
    for all }
  {1\le k\le j_0} \quad\text{and}\quad
  \mu_{j_0+1}(\Lambda) \ge  \nu^{\rk{\Lambda} -j_0}\log\det(\Lambda).
\end{alignat}

We write $\Lambda(\nu)$ for the rank $j_0$ lattice $\Lambda_{j_0}$,
indicating its dependency on $\nu$. It satisfies
$\rk{\Lambda/\Lambda(\nu)} \ge 1$.

\refcomment{56}{
  Note that
  $\mu_{j_0}(\Lambda(\nu)) < \mu_{j_0+1}(\Lambda(\nu))$ if $j_0\ge 1$.
  Therefore, $\Lambda(\nu)$ appears in the 
  Harder--Narasimhan filtration of $\Lambda$
  as considered by Stuhler~\cite{stuhler:quadform} and
  Grayson~\cite{grayson:red}, if we
  include $\{0\}$ as a member of the filtration.
  In particular, $\Lambda(\nu)$ is the unique
  lattice in $\Lambda$ of rank $\rk{\Lambda(\nu)}$ and minimal
  determinant. 
}

Here are two simple properties.

First, for the Euclidean norm $|\cdot|_2$ we claim 
\begin{equation}
  \label{eq:vlowerbound}
  \log|v|_2 \ge \nu^{\rk{\Lambda/\Lambda(\nu)}}\log\det(\Lambda)\quad \text{for
    all}\quad v\in \Lambda\ssm\Lambda(\nu). 
\end{equation}
Indeed, the lattice
$\Lambda'$ generated by $\Lambda(\nu)$ and $v$ contains $\Lambda(\nu)$ strictly.
We must have $\rk{\Lambda'} > \rk{\Lambda}$, as $\det(\Lambda')$ would otherwise
be strictly less than $\det(\Lambda)$. (This shows in particular that
$\Lambda/\Lambda(\nu)$ is torsion free; a well-known property of the
Harder--Narasimhan filtration.)
So
$\rk{\Lambda'}=\rk{\Lambda}+1$ and by convexity of $P$ we find
$\log\det(\Lambda') \ge \log\det(\Lambda(\nu)) + \mu_{j_0+1}(\Lambda)$. 
On the other hand, $\det(\Lambda')\le \det(\Lambda') \det(\Lambda(\nu)\cap v\IZ) \le
\det(\Lambda(\nu))\det(v\IZ)$ \refcomment{65}
{is well-known, for a proof see Proposition 2~\cite{stuhler:quadform}}. We conclude $\log \det(v\IZ) \ge
\mu_{j_0+1}(\Lambda)$. Now $\det(v\IZ)=|v|_2$, so  (\ref{eq:vlowerbound}) follows from (\ref{eq:mujplus1}). 

Second, (\ref{eq:mujplus1}) implies 
\begin{equation}
  \label{eq:detLambdaub}
 \log\det(\Lambda(\nu)) \le \mu_1(\Lambda)+ \cdots + \mu_{j_0}(\Lambda) 
\le 2\nu^{1+\rk{\Lambda/\Lambda(\nu)}} \log\det(\Lambda).
\end{equation}

We now make things more concrete. 
 Let $\bzeta \in \IG_m^d$
have order $N$ and set
\begin{equation*}
  \Lambda_\bzeta = \{ u \in\IZ^d : \bzeta^u = 1\}. 
\end{equation*}
\refcomment{66,\# 67}{We consider the homomorphism
  $\IZ^d\rightarrow\IG_m$ defined by $u \mapsto \bzeta^u$ and see that
 $\IZ^d/\Lambda_\bzeta$ is isomorphic
  to the finite subgroup of $\IG_m$ generated by the coordinates of
  $\bzeta$. 
So $\IZ^d/\Lambda_\bzeta$ is cyclic of order $N$.
In particular,
 $\Lambda_\bzeta$ is a lattice in $\IR^d$ of rank $d$ with
$\det(\Lambda_\bzeta)=[\IZ^d:
  \Lambda_\bzeta]=N\ge 1$.}
The 
 saturation
\begin{equation}
\label{def:tildeLambdazeta}
  \widetilde\Lambda_\bzeta(\nu) = \{ u\in\IZ^d : \text{there is $n\in\IZ\ssm\{0\}$
  such that  $nu\in\Lambda_\bzeta(\nu)$}\}
\end{equation}
of $\Lambda_\bzeta$ in $\IZ^d$ will also be useful for us. 
It is a lattice of the same rank as $\Lambda_\bzeta(\nu)$.

For any lattice $\Lambda\subset \IR^d$ of positive rank we set
\begin{equation}
\label{def:firstminimum}
\lambda_1(\Lambda) = \min \left\{ |u| :
u\in\Lambda\ssm\{0\}\right\}
\end{equation}
\refcomment{69}{where as usual $|\cdot|$ denotes the maximum-norm.} 
It is convenient to define $\lambda_1(\{0\})=\infty$.

\begin{propo}
\label{prop:geometryofnumbers}
Let $\nu\in (0,1/4]$ and let $\bzeta \in \IG_m^d$ be of  order $N$. There exists
  $V\in \gl{d}(\IZ)$ 
  and a
decomposition $\bzeta = \bfeta\bxi$ with $\bfeta$ and $\bxi$ in
$\IG_m^{d}$ of finite order $E$ and $M$, respectively, such that the
following holds. We abbreviate $\vi = \rk{\Lambda_\bzeta/\Lambda_\bzeta(\nu)} \in
  \{1,\ldots,d\}$. 
\begin{enumerate}
\item[(i)] We have $E\mid N,M\mid N,$ and $E\le N^{2\nu^{1+\vi}}$. In particular, 
$\IQ(\bfeta,\bxi)= \IQ(\bzeta)$. 
\item[(ii)] We have $|V|\ll_d N^{2\nu^{1+\vi}}$ 
with  $\bxi^V = (1,\ldots,1,\bxi')$ and $\bxi'\in \IG_m^{\vi}$.
\item[(iii)] If $G$ is a subgroup of $\GammaM$ there exist
$a\in\IZ^{\vi}$ and  $\sigma\in G$
  such that ${\bxi'} = \be(a\sigma /M),$  
\begin{equation}
  \label{eq:conjugateisnear1}
  |a|<M,\quad\text{and}\quad 
\frac{|a|}{M} \ll_d \frac{[\GammaM:G]\mf_G^{1/2}}{N^{\nu^{\vi}/(6d
    )}}.
\end{equation}
\item[(iv)] With the definition (\ref{def:deltazeta}) we have
$\delta(\bxi) \ge d^{-1/2}\min\{\lambda_1(\widetilde
\Lambda_\bzeta(\nu)),N^{\nu^{d}/2}\}$.
\end{enumerate}
\refcomment{71}{Moreover, if $\vi = d$, or equivalently $\Lambda(\nu)=\{0\}$, then $V$ is the identity matrix.}
\end{propo}
\begin{proof}
  We abbreviate $\Lambda=\Lambda_\bzeta$ as well as
  $\Lambda(\nu)=\Lambda_\bzeta(\nu)$ and   $\widetilde\Lambda(\nu)=\widetilde\Lambda_\bzeta(\nu)$. 
Note $\det(\Lambda)=N$.

We can find a collection of $d-\vi =\rk{\widetilde \Lambda(\nu)}$ linearly independent vectors in
$\widetilde\Lambda(\nu)$ whose norms are at most
$\ll_d \det(\widetilde\Lambda(\nu))$ by
applying Minkowski's Second Theorem, \refcomment{72}{see
  Theorem~V in
Chapter~VIII~\cite{cassels:geonumbers}, and using
$\lambda_1(\Lambda)\ge 1$.} By appending
suitable standard basis vectors of $\IZ^d$ we find $d$ linearly
independent vectors in $\IZ^d$. By Corollary 2, 
Chapter I.2~\cite{cassels:geonumbers} applied to $\IZ^d$
\refcomment{73}{and these
vectors}  we get a basis of $\IZ^d$
whose entries have norm $\ll_d\det(\widetilde\Lambda(\nu))$.
\refcomment{73}{By the said corollary, the original linearly independent vectors
can be expressed via an triangular matrix in terms of the new basis
vectors.} So 
the first $\rk{\widetilde \Lambda(\nu)}$
entries of this basis are a basis of the saturated group
$\widetilde\Lambda(\nu)$. Thus there exists $V\in \gl{d}(\IZ)$ whose first
$\rk{\widetilde \Lambda(\nu)}$ columns constitute a basis of $\widetilde\Lambda(\nu)$ and
\begin{equation}
\label{eq:normBbound}
|V|\ll_d \det (\widetilde \Lambda(\nu)).
\end{equation}
As $\det(\widetilde\Lambda(\nu))\le\det(\Lambda(\nu))$, the bound for $|V|$ in (ii)
follows from (\ref{eq:detLambdaub}).

We write $\bzeta^V = (\bfeta',\bxi')$
where $\bfeta' \in \IG_m^{d-\vi}$ and $\bxi'\in  \IG_m^{\vi}$ both have finite
order dividing $N$. We take $\bfeta$ and $\bxi$ from the assertion to
equal $(\bfeta',1,\ldots,1)^{V^{-1}}$ and
$(1,\ldots,1,\bxi')^{V^{-1}}$, respectively.

Observe that $[ \widetilde\Lambda(\nu):\Lambda(\nu)]
\widetilde\Lambda(\nu)\subset \Lambda(\nu) \subset\Lambda$. So the first $\rk{\widetilde{\Lambda}(\nu)}$
entries of $\bzeta^{[\widetilde \Lambda(\nu):\Lambda(\nu)]V}$ are
${\bfeta'}^{[\widetilde \Lambda(\nu):\Lambda(\nu)]}=1$. This implies
that $E=\ord(\bfeta)$ from the assertion satisfies $E \mid
[\widetilde \Lambda(\nu):\Lambda(\nu)]$ and thus $E\le
\det(\Lambda(\nu)) \le N^{2\nu^{1+\vi}}$ by (\ref{eq:detLambdaub}).

To verify (iii) let us fix $v\in\IZ^{\vi}\ssm\{0\}$ such that
${\bxi'}^{v}=1$ and $|v|=\delta(\bxi')$. Then $\bxi^{V' v}=1$ where
$V'\in \mat{d\vi}(\IZ)$ consists of the final $\vi$ columns of $V$.
Raising to the $E$-th power to kill $\bfeta$ yields $\bzeta^{EV'
v}=1$. Therefore, $EV'v \in \Lambda$. Note that
$EV'v\not\in \Lambda(\nu)$, indeed otherwise $V'v$ would lie in the
saturation $\widetilde\Lambda(\nu)$. This is impossible as no
non-trivial linear combination of columns of $V'$ lies in
$\widetilde\Lambda(\nu)$ which is generated by the first
$\rk{\widetilde\Lambda(\nu)}$ columns of $V$. Thus
(\ref{eq:vlowerbound}) implies $|EV'v|_2 \ge N^{\nu^{\vi}}$. By
(\ref{eq:normBbound})
\begin{equation*}
 |EV'v|\ll_d E |V'||v|\ll_d E|V||v|\ll_d [\widetilde \Lambda(\nu):\Lambda(\nu)]
\det(\widetilde\Lambda(\nu))|v| = \det (\Lambda(\nu))|v|
\end{equation*}
we conclude $N^{\nu^\vi} \ll_d \det(\Lambda(\nu))|v|$. 
The determinant bound in
(\ref{eq:detLambdaub}) gives
\begin{equation*}
\delta(\bxi')=  |v| \gg_d N^{\nu^{\vi}-2\nu^{1+\vi}}
  \gg_d N^{\nu^{\vi}/2}
\end{equation*}
and the last inequality used $\nu\le 1/4$. 

To complete the proof of (iii) let $G$ be a subgroup of $\GammaM$
where $M=\ord(\bxi)=\ord(\bxi')$. By Lemma
\ref{approx} applied to $\bxi'$ there are $a\in\IZ^{\vi}$
and $\sigma\in G$
with $\bxi'=\be(a\sigma/M), |a|<M,$ and 
\begin{equation*}
  \frac{|a|}{M} 
  \ll_d \frac{[\GammaM:G]^{1/i}\mf_G^{1/(2\vi)}} {\delta(\bxi')^{1/(3\vi)}}\ll_d
  \frac{[\GammaM:G]\mf_G^{1/2}} {N^{\nu^{\vi}/(6d)}}.
\end{equation*}  

It remains to check (iv). Say $v\in\IZ^d\ssm\{0\}$
with  $\bxi^v = 1$ and $|v|=\delta(\bxi)$. 
Then $\bzeta^v = \bfeta^v \bxi^v = \bfeta^v$. Thus ${Ev} \in \Lambda$
and there are two cases to consider.
If $v\in \widetilde\Lambda(\nu)$, then
 $|v|_2\ge \lambda_1(\widetilde\Lambda(\nu))$ by definition.
Otherwise, $v\not\in \widetilde\Lambda(\nu)$ in which case
 $Ev\not\in \Lambda(\nu)$ by saturation. Here we can use 
(\ref{eq:vlowerbound})  and the bound for $E$ from (i) to conclude
$|v|_2\ge E^{-1} N^{\nu^{\vi}}\ge N^{\nu^{\vi}-2\nu^{1+\vi}}\ge
N^{\nu^{\vi}/2}$. 
So $|v|\ge |v|_2/\sqrt{d}\ge N^{\nu^\vi/2}/\sqrt d$, as claimed
in (iv).
\end{proof}

\refcomment{71}{The situation simplifies in the following two cases.
  If $\vi=d$, then
 $\bxi=\bzeta,\bfeta=1,M=N, E=1$, and $V$ is the identity matrix.
If $N$ is a prime, then
 $E=1$ as $E\mid N$ and $E\le N^{2\nu^{1+\vi}}<N$ by
part (i) above. Thus again $\bxi=\bzeta$ and $\bfeta=1$.}


\section{A preliminary result}\label{sec:preliminary}

Let $d\ge 1$ be an integer. 

\begin{defin}
  \label{def:Bprim}
We use the convention $\inf\emptyset = \infty$.
For  $u\in\IZ^d$ we define 
\begin{equation}
\label{def:rhoa}
  \rho(u) = \inf \bigl\{ |v| : v\in\IZ^d\ssm\{0\} \text{ and }\langle u,v\rangle=0
  \bigr\}.
\end{equation}
For a Laurent polynomial $P\in \IQbar[X_1^{\pm 1},\ldots,X_d^{\pm 1}]$ we
define
\begin{equation*}
\begin{aligned}
 \cBprim(P) = \inf\bigl\{B\in\IN :   &\text{ if $\bfeta
  \in(\mu_\infty)^d,z\in S^1\ssm\mu_\infty$ is algebraic,
and $u\in\IZ^d$ with $P(\bfeta z^u)=0$}
\\ &
\text{ then $\rho(u)\le B$}\bigr\}.
\end{aligned}
\end{equation*}
\end{defin}

Let us spell this out for $d=1$. Then \refcomment{typo}{$\rho(u)=1$} for $u=0$ and
$\rho(u)=\infty$ otherwise. If $P$ vanishes at a point $S^1$ of
infinite order, then $\cBprim(P)=\infty$. Conversely, if $P$ does not
vanish at any point of $S^1\ssm\mu_\infty$ then we have $\cBprim(P)=1$.
In particular,  if $d=1$ and $P$ is \essatoral{},  then $\cBprim(P)=1$.


Let $\bzeta\in\IG_m^d$ have  order $N$ and say  $\nu \in (0,1/2]$. 
Below we make  use of the canonically determined lattice  $\Lambda_\bzeta(\nu)$
 attached to $(\bzeta,\nu)$   as  in Section \ref{sec:lattice}.
Recall that 
$\lambda_1(\widetilde\Lambda_\bzeta(\nu))$ is the least positive
Euclidean norm of a vector in the saturation of $\Lambda_\bzeta(\nu)$
in $\IZ^d$. 
For technical reasons we work with
\begin{equation}
\label{def:tildelambdanu}
\widetilde\lambda(\bzeta;\nu)
= \min\left\{\lambda_1(\widetilde\Lambda_\bzeta(\nu)), N^{\nu^{d}/2}\right\}.
\end{equation}
For example, if  $\Lambda_\bzeta(\nu)$ is $\{0\}$, then  the minimum equals $N^{\nu^d/2}$.

An important goal is to generalize Proposition \ref{univariatebound}
to multivariate polynomials. Proposition \ref{prop:etaxifactorization}
below is a step in this direction.

\begin{propo}
\label{prop:etaxifactorization}
Let $K\subset\IC$ be a number field, $0<\nu\le 1/(128d^2)$, and suppose
 $P\in K[X_1,\ldots,X_d]\ssm \{0\}$ has at most $k$ non-zero terms for an integer
 $k\ge 2$ and satisfies $\cBprim(P)<\infty$. Let $\bzeta\in\IG_m^d$ have
  order $N$ and suppose
 $G$ is a subgroup of $\GammaN$ with
  $P(\bzeta^\sigma)\not=0$ for all
  $\sigma\in G$. Then the following properties hold true
  with  $\vi=d-\rk{\Lambda_\bzeta(\nu)}\ge 1$.
 \begin{enumerate}
  \item [(i)] If $d=1$, then
\begin{alignat*}1
&\frac{1}{\#G}\sum_{\sigma\in G}\log|P(\bzeta^\sigma)| =
 m(P) + O_{d,k}\left(\frac{[K:\IQ]^2 [\GammaN:G]^2
 \mf_G
  \deg(P)^2(1+\hproj{P})}{N^{\nu^{\vi}/(20d)}}\right). 
  \end{alignat*}
 \item [(ii)] If $d\ge 2$ and
 $\widetilde\lambda(\bzeta;\nu)> 
 d^{1/2} \max\{\cBprim(P),\deg P\}$,
then
\begin{equation}
\label{eq:prelimpropii}
  \begin{aligned}   
\frac{1}{\#G}\sum_{\sigma\in G}&\log|P(\bzeta^\sigma)| = m(P) +\\
& O_{d,k,\nu}\left(\frac{[K:\IQ]^2 [\GammaN:G]^2 \mf_G
  \deg(P)^2(1+\hproj{P})}{N^{\nu^{\vi}/(20d)}}+\frac{\deg(P)^{16d^2}}{\widetilde\lambda(\bzeta;\nu)^{1/(16(k-1))}}\right). 
  \end{aligned}
\end{equation}
\end{enumerate}
\end{propo}
\begin{proof}
We may assume that $P$ is non-constant. Part (i) follows with ample
margin from Proposition \ref{univariatebound} with $Q=P$ and $F=K$.
Indeed, we use require the standard estimate $d_0(N) \ll_\epsilon
N^\epsilon$ which holds for all $\epsilon > 0$. We refrain from
stating better bounds in (i) for the \refcomment{76}{purpose of
better} comparability with the bounds in part (ii).
  
We split the proof of part (ii) up into 5 steps.

\smallskip
{\bf Step 1: Reduction to the univariate case.}
We write $L$ for the fixed field of $G$ in $\IQ(\bzeta)$. 
Note that $G$ is the Galois group of $\gal{\IQ(\bzeta)/L} =
\gal{L(\bzeta)/L}$. 

By Proposition \ref{prop:geometryofnumbers} \refcomment{77}{applied to} $\bzeta$ we obtain
$V \in \gl{d}(\IZ)$ and 
a decomposition $\bzeta = \bfeta\bxi$.
Let $E=\ord(\bfeta)$ and $M=\ord(\bxi)$.
By (i) of Proposition \ref{prop:geometryofnumbers}  we find
\begin{equation}
\label{eq:Eboundgon}
 E \le N^{2\nu^{1+\vi}}  \quad\text{and thus}\quad
 M\ge N/E \ge N^{1-2\nu^{1+\vi}}.
\end{equation}

The group used in Proposition \ref{prop:geometryofnumbers}(iii) is obtained as follows;
we denote it with $H$ to avoid a clash of notation with $G$ from
above.
\refcomment{84}{Let $H$ be the subgroup of 
$\GammaM$ corresponding to
$\gal{\IQ(\bxi)/\IQ(\bxi)\cap L(\bfeta)}$.}
By Galois theory, \refcomment{78}{see for example Theorem~VI.1.12~\cite{LangAlgebra},} the restriction homomorphism
 $\gal{L(\bxi)/L}\rightarrow \gal{\IQ(\bxi)/\IQ(\bxi)\cap L}$
 is an isomorphism.
\refcomment{84}{Using this isomorphism we will 
 identify $H$  with 
$\gal{L(\bxi)/L(\bxi)\cap L(\bfeta)}$.}

For future reference we estimate the conductor of $H\subset\GammaM$. The fixed field
of $H$ in $\IQ(\bxi)$ is $\IQ(\bxi)\cap L(\bfeta)$. By the characterization of $\mf_G$,
the field $L$ is contained in $\IQ(\be(1/\mf_G))$.
So $\IQ(\bxi)\cap L(\bfeta) \subset
\IQ(\be(1/M))\cap \IQ(\be(1/\mf_G),\be(1/E))$
since $\bxi$ has order $M$ and $\bfeta$ has order $E$.
This final intersection is generated by a root of unity of 
order
$\gcd(M,\lcm{\mf_G,E})$.
We conclude
\begin{equation}
\label{eq:condHbound}
\mf_H \le \lcm{\mf_G,E} \le \mf_G E \le \mf_G N^{2\nu^{1+\vi}}
\end{equation}
having used (\ref{eq:Eboundgon}).

We use basic Galois theory to compute
\begin{equation*}
\frac{1}{\#G} \sum_{\sigma\in G} \log
|P(\bzeta^\sigma)| 
= \frac{1}{[L(\bfeta):L]} \sum_{\tau\in \gal{L(\bfeta)/L}}
\sum_{\substack{\sigma\in\gal{L(\bxi)/L} \\
\tau|_{L(\bfeta)\cap
L(\bxi)} =
\sigma|_{L(\bfeta)\cap
  L(\bxi)}}}
\frac{1}{\#H} \log|P(\bfeta^\tau \bxi^\sigma)|.
\end{equation*}
Observe that the inner sum is over a coset of $\widetilde\tau H$ of
$H$ inside $\GammaM$; here $\widetilde\tau \in \GammaM$ restricts to
the restriction  $\tau|_{L(\bfeta)\cap L(\bxi)}$.
Below,  $\tau$ is as in the outer sum. The inner sum equals 
\begin{alignat}1
\label{eq:Stau1}
S_\tau &= \frac{1}{\#H}\sum_{\sigma\in \widetilde\tau
H} \log|P(\bfeta^\tau \bxi^\sigma)|
= \frac{1}{\#H}\sum_{\sigma\in 
H} \log|P(\bfeta^\tau \bxi^{\widetilde\tau \sigma})|,
\end{alignat}
that is 
\begin{equation}
\label{eq:galoisdecompositionsum2}
\frac{1}{\#G}\sum_{\sigma\in G} \log|P(\bzeta^\sigma)| =
\frac{1}{[L(\bfeta):L]} \sum_{\tau \in\gal{L(\bfeta)/L}} S_\tau.
\end{equation}

By Proposition \ref{prop:geometryofnumbers}(iii) applied to $H$
we get 
 $a\in\IZ^{\vi}$ satisfying (\ref{eq:conjugateisnear1})
 and   $\sigma_0\in H$ \refcomment{80}{with
 $\bxi^V = (1,\ldots,1,\be(a\sigma_0/M))$}. 
We extend $a$ to the left by $d-\vi$ zeros
and obtain a row vector $(0,a)\in\IZ^d$.
We set $u=(0,a)V^{-1}\in\IZ^d$ 
and use Proposition \ref{prop:geometryofnumbers}(ii) 
to get
 $\bxi= \be(u\sigma_0/M)$. Let us set
\begin{equation}
\label{eq:QPequality}
Q = P(\bfeta^\tau X^u)X^l
\end{equation}
in the unknown $X$; it depends on $\tau$ and the \refcomment{81}{exponent $l$ is} chosen to make sure that $Q$ is
a polynomial. 
So $0\not=|P(\bfeta^\tau \bxi^{\widetilde \tau \sigma})| =
|Q(\be(\widetilde \tau \sigma\sigma_0/M))|$
and in particular $Q\not=0$. We may assume that $Q(0)\not=0$. 
The coefficients of $Q$ 
lie in  $F=K(\bfeta)$ and $Q$ has at most $k$ non-zero terms as $P$ has \refcomment{82}{at most} this
many non-zero terms. 
All this allows us to rewrite (\ref{eq:Stau1}) using a univariate
polynomial, $\sigma_0$ above is absorbed by the sum
\begin{equation}
\label{eq:StauinQ}
S_\tau = \frac{1}{\#H}\sum_{\sigma\in 
H} \log|Q(\be(\widetilde\tau \sigma/M))|.
\end{equation}

\smallskip
{\bf Step 2: Non-vanishing of $Q$ on $S^1\ssm \mu_\infty$.}
Suppose $w\in\IZ^d\ssm\{0\}$ satisfies $\langle u,w\rangle = 0$
and $|w|=\rho(u)$. 
 Recall that $\bxi =\be(u \sigma_0 /M)$, so $\bxi^w =1$.
Thus $|w|\ge\delta(\bxi)$ and Proposition
\ref{prop:geometryofnumbers}(iv)
together with  (\ref{def:tildelambdanu})
yield
\begin{equation}
\label{eq:deprimub2}
 \rho(u)=
|w|\ge d^{-1/2} \widetilde\lambda(\bzeta;\nu).
\end{equation}

Let $z\in S^1\backslash\mu_\infty$ be algebraic. If $Q(z)=0$
then $P(\bfeta^\tau z^u)=0$ by
(\ref{eq:QPequality}). 
\refcomment{83}{By Definition~\ref{def:Bprim} we have $\rho(u)\le \cBprim(P)$.}
This and (\ref{eq:deprimub2}) contradict the lower bound
$\widetilde\lambda(\bzeta;\nu)>d^{1/2} \cBprim(P)$ in the hypothesis.
Hence $Q(z)\not=0$.

Thus $Q$, having algebraic coefficients, does not vanish at any point
of $S^1\ssm\mu_\infty$. As $\rho(u)>1$ we also have $u\not=0$.

\smallskip
{\bf Step 3: Bounding quantities in prepartion of Proposition \ref{univariatebound}.}
This step is mainly bookkeeping. We aim to  apply Proposition \ref{univariatebound} to
$Q,$ the root of unity $\be(\widetilde\tau/M)$, and the subgroup $H\subset \GammaM$ to
determine the asymptotic behavoir of $S_\tau$.
To proceed we bound the  various
quantities  below separately: 
\begin{equation}
\label{eq:6bounds}
\begin{tabular}{rl}
$[\GammaM:H]$ & $\le [\GammaN:G] N^{2\nu^{1+\vi}},$  \\
$\mf_H$ & $\le \mf_G N^{2\nu^{1+\vi}},$ \\ 
$\deg(Q)$ & $\ll_d \deg(P) \min\{[\GammaN:G]\mf_G^{1/2}
N^{1-\nu^{\vi}/(10d)},N^2\},$ \\
$\hproj{Q}$ & $=  \hproj{P}$, \\
$[K(\bfeta):\IQ]=[F:\IQ]$ & $ \le [K:\IQ] N^{2\nu^{1+\vi}},$
\end{tabular}
\end{equation}

\refcomment{84}{Note that $\# H = [\IQ(\bxi):\IQ(\bxi)\cap L(\bfeta)]
 = [\IQ(\bxi):\IQ]/[\IQ(\bxi)\cap L(\bfeta):\IQ] \ge
 [\IQ(\bxi):\IQ]/[L(\bfeta):\IQ]$ and since
 $[\IQ(\bxi):\IQ] = \# \GammaM$ we find 
 $[\GammaM:H]\le [L(\bfeta):\IQ] \le [L:\IQ] E$.
 The first bound follows from (\ref{eq:Eboundgon}) and as
 $[L:\IQ] = [\GammaN:G]$.}


We already proved the bound for $\mf_H$ in (\ref{eq:condHbound}). 

Next comes $\deg(Q)$. 
Observe that 
\begin{alignat*}1
\deg(Q) &\ll_{d}  |a| |V^{-1}|\deg(P)\ll_d |a| |V|^{d-1}\deg(P)
 \\ &\ll_d
 [\GammaM:H] \mf_H^{1/2} \deg(P)
 N^{1+2(d-1)\nu^{1+\vi}-\nu^{\vi}/(6d)}
 \\ &\ll_d
 [\GammaN:G] \mf_G^{1/2}  \deg(P)
 N^{1+2\nu^{1+\vi} + \nu^{1+\vi}+2(d-1)\nu^{1+\vi}-\nu^{\vi}/(6\vi)}
 \end{alignat*}
 having used the bounds in Proposition \ref{prop:geometryofnumbers},
 $M\le N$,
 and the first two bounds in (\ref{eq:6bounds}). 
As  ${\nu\le 1/(128 d^2)}$
the exponent of $N$ is at most 
$1 + (2d+1)\nu^{1+\vi}-\nu^{\vi}/(6d)\le
1- \nu^{\vi}/(10d)$ and thus we obtain 
\begin{equation*}
\deg(Q) \ll_d [\GammaN:G] \mf_G^{1/2}\deg(P)  N^{1-\nu^{\vi}/(10
d)}
\end{equation*}
which is part of the third inequality in (\ref{eq:6bounds}).
The bound $\deg(Q) \ll_d \deg(P)N^2$ is proved similarly, but requires
only the trivial estimate $|a|<M\le N$ from (\ref{eq:conjugateisnear1}) and $|V^{-1}| \ll_d
N^{2d\nu^{1+\vi}}$. 

 We claim  that the coefficients of $P(\bfeta^\tau X^u)$
are equal to the coefficients of $P$ up-to multiplication by a root of
unity. 
In view of the definition of the height (\ref{def:hprojP}) this will imply the fourth claim in (\ref{eq:6bounds}).
Indeed, it suffices to rule out that two distinct monomials
in $P$  lead to the same power of $X$ after the substitution. 
Hence it suffices to verify $\rho(u) > \deg P$. 
But this follows from (\ref{eq:deprimub2}) and as
$\widetilde\lambda(\bzeta;\nu)>d^{1/2}\deg P$ by hypothesis.

The degree of the number field $F$ containing the coefficients of
$Q$ satisfies
\begin{equation*}
[F:\IQ] = [K(\bfeta):\IQ]\le [K:\IQ][\IQ(\bfeta):\IQ]\le
[K:\IQ]E
\le [K:\IQ] N^{2\nu^{1+\vi}}
\end{equation*}
where we used (\ref{eq:Eboundgon}). This implies 
the fifth claim in (\ref{eq:6bounds}).

\smallskip
{\bf Step 4: Applying Proposition \ref{univariatebound} in
the univariate case.}
Our aim is to determine the asymptotics of (\ref{eq:StauinQ}). 
We use the bounds from the last step to control  the
 error
term in (\ref{meanvalue2}) arise in Proposition \ref{univariatebound}
applied to $Q, \be(\widetilde\tau/M),$ and $H$. By
(\ref{eq:6bounds})
the error is
\begin{alignat*}1
&\ll  [F:\IQ]^2 [\GammaM:H]\mf_H^{1/2} \deg(Q)(\log(2\deg Q)
+\hproj{Q}) \frac{(\log 2M)^3 d_0(M)}{M} \\
&\ll_d [K:\IQ]^2 [\GammaN:G]^2 \mf_G \deg(P)(\log (2N^2\deg P)
+\hproj{P}) N^{9\nu^{1+\vi}+1-\nu^{\vi/(10d)}}\frac{(\log 2M)^3 d_0(M)}{N}
\end{alignat*}
where we use $\deg Q \ll N^2\deg P$ to bound $\log(2\deg Q)$ from
above and the lower bound for $M$ in (\ref{eq:Eboundgon}). 

The exponent of $N$ is $9\nu^{1+\vi}-\nu^{\vi}/(10d)\le
-\nu^{\vi}/(19d)$ as $\nu \le 1/(128d^2) \le 1/(256d)$. As $M\mid N$
we find $d_0(M)\le d_0(N)$. It is well-known that $d_0(N)\ll_\epsilon
N^\epsilon$ for all $\epsilon$. \refcomment{85}{We also anticipate $\log(2N^2)$ coming
from $\log(2N^2\deg P)$ to find}
\begin{equation*}
\log(2N^2)N^{9\nu^{1+\vi} - \nu^{\vi/(10d)}} (\log 2M)^3 d_0(M)\ll_{d,\nu}
N^{-\nu^{\vi}/(20d)}.
\end{equation*}
Using the crude inequality $\log \deg P\le \deg P$ the error term is thus
\begin{equation*}
\ll_{d,\nu}
[K:\IQ]^2 [\GammaN:G]^2 \mf_G \deg(P)^2
(1+\hproj{P})N^{-\nu^{\vi}/(20d)}.
\end{equation*}

Applying Proposition \ref{univariatebound} and recalling
$m(Q)=m(P(\bfeta^\tau X^u))$ we  find
\begin{alignat}1
\label{eq:stauasymptotics}
S_\tau &= m(P(\bfeta^\tau X^u)) + O_{d,\nu}\left(\frac{[K:\IQ]^2
  [\GammaN:G]^2\mf_G \deg(P)^2(1+\hproj{P})}{N^{\nu^{\vi}/(20d)}}\right).
\end{alignat}

\smallskip
{\bf Step 5: Applying a quantitative version of Lawton's Theorem.}
To determine the asymptotics of the Mahler measure
we apply our quantitative variant of Lawton's Theorem,
Theorem \ref{thm:lawtonquant}
to $P(\bfeta^\tau (X_1,\ldots,X_d))\not=0$.
This polynomial has the same degree and number of terms as $P$. 
The exponent vector  satisfies $\rho(u)\ge d^{-1/2} \widetilde\lambda(\bzeta;\nu)$ by (\ref{eq:deprimub2}). Our hypothesis implies $\rho(u)> \deg P$, as required by
Theorem \ref{thm:lawtonquant}. We find
\begin{equation}
\label{eq:lawtonconclusion} 
m(P(\bfeta^\tau X^u)) = m(P(\bfeta^\tau(X_1,\ldots,X_d))) +
O_{d,k} \left(\frac{\deg(P)^{16d^2}}{\widetilde{\lambda}(\bzeta;\nu)^{1/(16(k-1))}}\right). 
\end{equation}

The Mahler measure of $P$ and $P(\bfeta^\tau(X_1,\ldots,X_d))$ are
equal as translating by $\bfeta^\tau\in(S^1)^d$ does not affect the
value of the integral.

By combining (\ref{eq:stauasymptotics}) and
(\ref{eq:lawtonconclusion}) we conclude
\begin{alignat*}1 
S_\tau  = m(P) &+ O_{d,k,\nu}\left(\frac{[K:\IQ]^2
  [\GammaN:G]^2\mf_G 
  \deg(P)^2(1+\hproj{P})}{N^{\nu^{\vi}/(20d)}}
  + 
  \frac{\deg(P)^{16d^2}}{\widetilde{\lambda}(\bzeta;\nu)^{1/(16(k-1))}}\right).
  \end{alignat*}
The proposition  follows from
(\ref{eq:galoisdecompositionsum2}). 
\end{proof}

We now explain why the  situation
simplifies  when the order $N$ of $\bzeta$ is a prime number. 
In this case, after the proof of Proposition \ref{prop:geometryofnumbers} we observed
that   $\bfeta=1$ and $\bzeta=\bxi$. In the proof above, inequality
(\ref{eq:deprimub2}) can be replaced by
$ \rho(u) \ge \delta(\bzeta)$.
So the hypothesis on $\bzeta$ in (ii) of the proposition
can be replaced by $\delta(\bzeta) > \max\{\cBprim(P),\deg P\}$; see also
the argument near (\ref{eq:lawtonconclusion}). This is
certainly satisfied for $\delta(\bzeta)\rightarrow\infty$. 
Moreover, $\widetilde\lambda(\bzeta;\nu)$ can be replaced by
$\delta(\bzeta)$ in (\ref{eq:prelimpropii}). 
From this point it is not difficult
to deduce Theorem \ref{thm:main}  when 
$N$ is a prime.

The remaining argument is required to treat general $N$. We need  to keep track of extra information such as $[K:\IQ], [\GammaN:G],\mf_G,$ and the dependency on $P$ to anticipate a
monomial change of coordinates.


\section{Equidistribution}\label{sec:equidistribution}

Proposition \ref{prop:etaxifactorization} closes in on Theorem \ref{thm:main}.
Indeed, suppose that for some choice of $\nu$ the value
 $\widetilde\lambda(\bzeta;\nu)$ grows polynomially in
 $\delta(\bzeta)$. Then the error term of (\ref{eq:prelimpropii}) tends to $0$ as
 $\delta(\bzeta)\rightarrow \infty$ and we are done. 

However, consider the following example, already found in the
beginning of Section \ref{sec:lattice}.  
Suppose $n\ge 2$ and $\bzeta_p$ and $\bzeta_{p^n}$ are roots of unity of
order $p$ and $p^n$, respectively. Say $\bzeta =
(\bzeta_p,\bzeta_{p^n})$, it has order $p^n$.
The lattice $\Lambda_{\bzeta}$ contains $(p,0)$ and this vector has
minimal positive norm in $\Lambda_{\bzeta}$.  
For $n$ large enough in terms of $\nu$ we have $\Lambda(\nu) =
(p,0)\IZ$ and $\widetilde\Lambda(\nu)=(1,0)\IZ$. Thus
$\lambda_1(\widetilde\Lambda_{\bzeta}(\nu))=1$ and
this yields $\widetilde\lambda(\bzeta;\nu)=1$.

This example suggests a  monomial change of coordinates which we will
do in the next section. In the
current section we lay the groundwork for this change of coordinates. 


\subsection{Numerical integration}

We require a  higher dimensional replacement of the Koksma
bound, Theorem 5.4~\cite{Harman}. The classical analog is 
called the Koksma--Hlawka Inequality and applies to functions of bounded
variation in the sense of Hardy and Krause.
\refcomment{92}{Let $\theta: U\rightarrow\IR$ be a function whose domain
$U$ is a non-empty subset of $\IR^d$. In this subsection we use the more
rudimentary \textit{modulus of continuity} of $\theta$ defined by}
\begin{equation}
  \label{def:omega}
  \omega(\theta;t) = \sup_{\substack{x,y\in U \\ |x-y|\le t}}
  |\theta(x)-\theta(y)|
\end{equation}
 for all
$t\ge 0$; as usual $|\cdot|$ denotes the maximum-norm on $\IR^d$.
\refcomment{92}{We define $\omega(\theta;t)=0$ if $U=\emptyset$.} We
will use it to estimate a mean in terms of the corresponding integral
in Proposition \ref{prop:numintegration}. 
Hlawka~\cite{Hlawka:71} has a related and more
precise result. For the reader's convenience we give a
self-contained treatment that suffices for our purposes.

\begin{propo}
  \label{prop:numintegration}
  Let $\theta:[0,1]^d\rightarrow\IR$ be a continuous function and
  let  $x_1,\ldots,x_n\in [0,1)^d$ with
    discrepancy $\cD = \cD(x_1,\ldots,x_n)$. Then
    \begin{equation}
      \label{eq:numintbound}
      \left|\frac 1n \sum_{i=1}^n \theta(x_i) - \int_{[0,1)^d}
        \theta(x) dx\right| 
        \le (1+2^{d+1}) \omega(\theta,\cD^{1/(d+1)}).        
    \end{equation}
\end{propo}
\begin{proof}
Both sides of (\ref{eq:numintbound}) are invariant under adding a
constant function to  $\theta$. So we may assume $\theta(0)=0$. 
  
  Let $T\ge 1$ be an integral parameter to be
  determined below. 
  We write $[0,1)^d$ as a disjoint union of $T^d$ half-open hypercubes
    $Q_j$ with side length $1/T$. Let $\overline{Q_j}$ denote the
    closure of $Q_j$ in $[0,1]^d$. 
The Mean Value Theorem tells us that
for each $j$  there exists $y_j\in \overline{Q_j}$
such that $\int_{Q_j} \theta(x) dx = \vol{Q_j} \theta(y_j)=T^{-d} \theta(y_j)$.

For each
$j$ we write $n_j=\#\left\{i \in \{1,\ldots,n\} : x_i \in Q_j\right\}$.
So 
\begin{equation}
\label{eq:numint1b}  \frac 1n \left|\sum_{i=1}^n \theta(x_i)
-\sum_{j} n_j \theta(y_j) \right|
  \le \frac 1n \sum_{j}
  \sum_{\substack{i=1 \\ x_i\in Q_j}}^n|\theta(x_i) -
  \theta(y_j)|
  \le \frac 1n \sum_j \omega(\theta;1/T) n_j =\omega(\theta;1/T). 
\end{equation}

On the other hand, $  \frac 1n \sum_{j} n_j\theta(y_j)$ equals
\begin{equation*}
 \sum_{j}  \frac {n_j}{n} T^d
  \int_{Q_j}\theta(x) dx  = \sum_{j} (1+\delta_jT^d )\int_{Q_j} \theta(x) dx
  =\int_{[0,1)^d}\theta(x) dx + T^d\sum_{j} \delta_j\int_{Q_j} \theta(x) dx
\end{equation*}
where $\delta_j = n_j/n-T^{-d}$.
The definition of discrepancy implies 
$|\delta_j|\le \cD$. Hence
\begin{equation}
  \label{eq:numint2b}
   \left|\frac 1n\sum_{j} n_j\theta(y_j) - \int_{[0,1)^d} \theta(x) dx\right|
     \le T^{d}\cD \int_{[0,1)^d} |\theta(x)| dx
    \le T^{d+1}\cD \omega(\theta;1/T)
\end{equation}
where we used 
$|\theta(x)|\le  T \omega(\theta;1/T)$ for all  $x\in
[0,1]^d$; recall that $\theta(0)=0$.

We apply the triangle inequality to  (\ref{eq:numint1b})
and (\ref{eq:numint2b}) and conclude that the left-hand side of
(\ref{eq:numintbound})
is at most $(1+T^{d+1}\cD)\omega(\theta;1/T)$.
To complete the proof observe that $0<\cD\le 1$ and fix
$T = \lceil \cD^{-1/(d+1)}\rceil$ which satisfies
$\cD^{-1/(d+1)}\le T\le \cD^{-1/(d+1)}+1$. 
\end{proof}

\subsection{Averaging the Mahler measure}

This subsection is  purely in the complex setting. Let
$P\in \IC[X_1,\ldots,X_d]\ssm\{0\}$ 
have  at most $k\ge 2$
non-zero terms, where $k$ is an integer.

Let  $l\in \{1,\ldots,d-1\}$.
For
$x\in\IR^l$ we define $P_{\be(x)} = P(\be(x),Y_{1},\ldots,Y_{d-l})\in
\IC[Y_{1},\ldots,Y_{d-l}]$.
Next we construct an auxiliary Laurent polynomial $\widehat P$ in $l$
variables whose value at $\be(x)$ is comparable to
$|P_{\be(x)}|$. 
 \refcomment{87}{
  For $i\in\IZ^{d-l}$ we  denote $p_i \in
\IC[X_{1},\ldots,X_l]$ the coefficients of $P$ taken
 as a Laurent polynomial
in  $X_{l+1},\ldots,X_d$ and define}
\begin{equation}
  \label{eq:defQpi}
  \widehat P=\sum_{i}
  p_i(X_1,\ldots,X_l)\overline{p_i}(X_1^{-1},\ldots,X_l^{-1}) \in  \IC[X_1^{\pm 1},\ldots,X_l^{\pm 1}]
\end{equation}
where the bar denotes complex conjugation.

\begin{lemma}
  \label{lem:Qprops0}
  \refcomment{103}{In the notation above the following properties hold true:}
  \begin{enumerate}
  \item [(i)] The Laurent polynomial $\widehat P$ has at most $k^2$
    non-zero terms.  
  \item[(ii)] The product  $(X_1\cdots X_l)^{\deg P}\widehat P$ is a
    polynomial of degree at most $(l+1)\deg P$.  
  \end{enumerate}
\end{lemma}
\begin{proof}
  If \refcomment{88}{each} $p_i$ consists of $k_i$ non-zero terms,
  then $\widehat P$ consists of at most $\sum_{i} k_i^2$ terms. Since
  $\sum_{i} k_i\le k$ we find that $\widehat P$ has at most $k^2$
  non-zero terms. This implies part (i). 

  Part (ii) follows from (\ref{eq:defQpi}).
\end{proof}

Observe that $\widehat P(\be(x)) = \sum_i |p_i(\be(x))|^2\ge 0$. As
  $|P_{\be(x)}|$ is the maximum of $|p_i(\be(x))|$ as $i$ varies, we
find
\begin{equation}
  \label{eq:cmpQbexPbex}
\frac{1}{k^{1/2}} \widehat P(\be(x))^{1/2}\le   |P_{\be(x)}| \le \widehat P(\be(x))^{1/2}.
\end{equation}
So $P_{\be(x)}=0$ if and only if $\widehat P(\be(x))=0$. 

  The main result of this subsection is 
\begin{propo}
  \label{prop:lehmermean}
Assume $P\in \IC[X_1,\ldots,X_d]\ssm \IC$ 
has at most $k$ non-zero terms for an integer $k\ge 2$. Let $l\in \{1,\ldots,d-1\}$ and let
$\widehat P$ be as above. 
    Suppose $x_1,\ldots,x_n\in [0,1)^l$ with discrepancy
    $\cD = \cD(x_1,\ldots,x_n)$.
If $P_{\be(x_i)}\not=0$ for all $i\in
    \{1,\ldots,n\}$, then
    \begin{equation}
      \label{lem:propmahleraverageassertion}
      \frac 1n \sum_{i=1}^n m(P_{\be(x_i)}) = m(P)
      +O_{d,k}\left( \deg(P)\cD^{1/(16(d+1)k^2)}+\left| m(\widehat P)-\frac 1n \sum_{i=1}^n \log
 \widehat P(\be(x_i))\right|\right).
    \end{equation}
\end{propo}

By a theorem of Boyd~\cite{Boyd}, the Mahler measure is a continuous
function in the coefficients of a non-zero polynomial of fixed degree
\refcomment{89}{(below in Lemma~\ref{lem:mahlerhoelder} we prove that
  it is even H\"older continuous).}
Therefore, if the $P_{\be(x_i)}$ in the proposition above are
uniformly bounded away from $0$, then the average on the left in
(\ref{lem:propmahleraverageassertion}) converges to the integral
$\int_{[0,1)^l} m(P_{\be(x)}) dx$ as the discrepancy tends to $0$. But
even when $|P|=1$ it is conceivable that $|P_{\be(x_i)}|$ is small for
some $x_i$, \refcomment{90}{then 
$P_{\be{(x_i)}}$ is near  the Mahler measure's logarithmic singularity.}
 This happens if and only if $\widehat P(\be(x_i))$ is
small by (\ref{eq:cmpQbexPbex}). The proposition states that we can
handle the mean for arbitrary $x_i$ if we can control the logarithmic
mean of $\widehat P$ over the $\be(x_i)$.

The proof follows a series of lemmas. \refcomment{92,}{
  We first note a
  useful property of the modulus of continuity as defined in
  (\ref{def:omega}). Let $\theta:[0,1]^d\rightarrow
  \IR\cup\{-\infty\}$ be a function and $c\in\IR$, such that
  $\theta_c(x) = \max\{c,\theta(x)\}$ defines 
  a continuous function $[0,1]^d\rightarrow\IR$.
  Then
  \begin{equation}
    \label{eq:maxcpsibound}
    \omega(\theta_c;t) \le \omega(\theta|_{\theta^{-1}((c,\infty))};t)
    \quad\text{for all}\quad t\ge 0;
  \end{equation}
  indeed, say $x,y\in [0,1]^d$ with $|x-y|\le t$.
  To bound $|\theta_c(x)-\theta_y(y)|$ from above by the right-hand side
  of 
   (\ref{eq:maxcpsibound}) we may assume
  $\theta_c(x)>c= \theta_c(y)$. By continuity of $\theta_c$ there is
  for all small enough $\epsilon>0$
  a  $z\in [0,1]^d$ on the line segment connecting $x$ and $y$ with
  $c+\epsilon = \theta_c(z)=\theta(z)$.
  Then $|\theta_c(x)-\theta_c(y)| =
  \theta(x)-c\le|\theta(x)-\theta(z)|+|\theta(z)-c| \le
  \omega(\theta|_{\theta^{-1}((c,\infty))};t) +\epsilon$. Our claim
  (\ref{eq:maxcpsibound}) follows as $\epsilon$ can be made arbitrarily small.
}

Let $P$ and $k$ be as in
Proposition \ref{prop:lehmermean} and assume in addition that $|P|=1$.

\begin{lemma}
  \label{lem:countsmallvalues}
  let  $x_1,\ldots,x_n\in [0,1)^d$ have
  discrepancy $\cD = \cD(x_1,\ldots,x_n)$.
  If \refcomment{91}{$r \in (0,1]$}, then
  \begin{equation}
    \label{eq:proportionsmallvalue}
    \frac 1n \#\{i\in \{1,\ldots,n\} : |P(\be(x_i))|\le r\} \ll_{d,k}
    r^{1/(2k)} + \deg(P) \cD^{1/(d+1)}/r.  
  \end{equation}
\end{lemma}
\begin{proof}
  For $x\in [0,1]^d$   we set
  \begin{equation*}
    \chi(x) = \max\{0,2 - |P(\be(x))|/r\}
  \end{equation*}
  and this defines a continuous  function
  on $[0,1]^d$ with values in $[0,2]$.

  We note that $\chi(x)\ge 1$ if
  $|P(\be(x))|\le r$.
  As $\chi$ is non-negative the average
  $\frac 1n \sum_{i=1}^n \chi(x_i)$ is at least the proportion of the
  $i$ among $\{1,\ldots,n\}$ such that $|P(\be(x_i))|\le r$. 
  On the other hand, Lemma \ref{lem:volSPep}(i) implies
  \begin{equation}
    \label{eq:volumepolysmallvaluebound}
    \int_{[0,1)^d} \chi(x)dx \le 2  \vol{\{ x\in [0,1)^d :
      |P(\be(x))|< 2r\}} \ll_{d,k}  r^{1/(2(k-1))}\ll_{d,k} r^{1/(2k)}.
  \end{equation}

  We will apply Proposition \ref{prop:numintegration} to bound the
  proportion on the left in (\ref{eq:proportionsmallvalue}). Say $t>0$,
  let us verify
  \begin{equation}
    \label{eq:mocbound1}
    \omega(\chi;t)\ll_{d,k} \deg(P)t/r.
  \end{equation}
  \refcomment{92}{We apply (\ref{eq:maxcpsibound}) to $\theta(x) =
  2-|P(\be(x))|/r$ and $c=0$.
  Say $x,y\in \theta^{-1}((0,\infty))$ with $|x-y|\le t$, so  in
  particular   $|P(\be(x))|<2r$
  and  $|P(\be(y))|<2r$.  Then  $|\theta(x)-\theta(y)|=
  |P(\be(x))-P(\be(y))|/r \ll_{d,k} \deg(P) t/r$, where we used
  $|x-y| \le t$ and $|P|=1$.  We obtain (\ref{eq:mocbound1}). 
}
  


  Let us set $t = \cD^{1/(d+1)}$. 
  We apply numerical integration, Proposition \ref{prop:numintegration}, and use (\ref{eq:volumepolysmallvaluebound})
  to conclude the proof. 
\end{proof}

In the next lemma we truncate the singularity of $x\mapsto \log
|P(\be(x))|$ using a parameter $r$ and bound the modulus of continuity
of the resulting function. 

\begin{lemma}
  \label{lem:logPmodcont}
  Let $r \in (0,1]$, for $x\in[0,1]^d$ we define
  $\psi(x) = \max\{\log r, \log|P(\be(x))|\}$ as above
  (\ref{eq:maxcpsibound}). 
  Then $\psi:[0,1]^d \rightarrow \IR$ is continuous and for all $t>0$
  we have
  \begin{equation*}
    \omega(\psi;t) \ll_{d,k} \frac{\deg(P)t}{r}.
  \end{equation*}
\end{lemma}
\begin{proof}
  Clearly, $\psi$ is continuous on $[0,1]^d$.
   \refcomment{92}{We apply (\ref{eq:maxcpsibound}) to $\theta(x) = \log|P(\be(x))|$
  and $c = \log r$. Say $x,y\in [0,1]^d$ with $|P(\be(x))|\ge
  |P(\be(y))| \ge r$ and $|x-y|\le t$. Then
  as in the proof of Lemma \ref{lem:countsmallvalues} we find
 $\bigl| |P(\be(y))/P(\be(x))| - 1\bigr|\ll_{d,k} \deg(P)t / |P(\be(x))|
  \ll_{d,k} \deg(P)t/r$. Applying the logarithm and using $0\le \log s \le
  s-1$ for all $s\ge 1$ yields
  \begin{equation*}
    \bigl|\log|P(\be(x))|-\log|P(\be(y))|\bigr|\ll_{d,k}
\frac{\deg(P)t}{r},
\end{equation*}
as desired. 
}
  
  


\end{proof}

\begin{lemma}
  \label{lem:logvsintpsi}
  We keep the notation of Lemma \ref{lem:logPmodcont}.
 Then
  \begin{equation*}
    \left|
      m(P)- \int_{[0,1)^d}
    \psi(x)dx\right| \ll_{d,k} r^{1/(4k)}. 
  \end{equation*}
\end{lemma}
\begin{proof}
The absolute value in question is
  \begin{equation*}
\cE = \left|\int_{\Sigma}\log|P(\be(x))|dx - \vol{\Sigma}\log r\right|
  \end{equation*}
  where $\Sigma=S(P,r)=\{x\in [0,1)^d : |P(\be(x))|< r\}$ in the
  notation of (\ref{def:SPr}). 
Hence $\vol{\Sigma}\ll_{d,k} r^{1/(2(k-1))}$ by Lemma
  \ref{lem:volSPep}(i). So
  \begin{equation*}
    \cE
    \ll_{d,k}  
\int_{\Sigma}\bigl|\log|P(\be(x))|\bigr|dx+r^{1/(2k)}
  \end{equation*}
  as $r\le 1$. 
To bound the final integral   we use Lemma
  \ref{lem:Spdeltaintegral} which implies
    $\cE\ll_{d,k} r^{1/(4(k-1))}+r^{1/(2k)}  \ll_{d,k} r^{1/(4k)}$.
\end{proof}

\begin{lemma}
  \label{lem:abslogPsumbound}
Let  $x_1,\ldots,x_n\in [0,1)^d$ with $P(\be(x_i))\not=0$ for all $i$ and
    discrepancy
    $\cD  =\cD(x_1,\ldots,x_n)$. 
    We set
    \begin{equation*}
\epsilon =      \left| m(P)-\frac 1n \sum_{i=1}^n \log| P(\be(x_i))|\right|.
    \end{equation*}
If $r\in (0,1]$, then 
    \begin{equation*}
  \frac 1n \sum_{|P(\be(x_i))|< r}   \bigl|\log|P(\be(x_i))|\bigr|
  \ll_{d,k} \deg(P)\cD^{1/(d+1)}r^{-2} +r^{1/(4k)} + \epsilon.
    \end{equation*}
\end{lemma}
\begin{proof}
  By the triangle inequality and with $\psi$ as in Lemma \ref{lem:logPmodcont} we have
  \begin{equation*}
    \left|\frac 1n \sum_{i=1}^n \psi(x_i)-\log|P(\be(x_i))|\right| \le
        \left|\frac 1n \sum_{i=1}^n \psi(x_i)-\int_{[0,1)^d} \psi(x)dx \right|
    +\left|  \int_{[0,1)^d} \psi(x)dx-m(P) \right|+    \epsilon.
  \end{equation*}
  We use Proposition \ref{prop:numintegration}
  and Lemma \ref{lem:logPmodcont} with
  $t=\cD^{1/(d+1)}$ to bound the first term on the right by
  $\ll_{d,k} \deg(P) \cD^{1/(d+1)}/r$. The second term is $\ll_{d,k}
  r^{1/(4k)}$ by Lemma \ref{lem:logvsintpsi}. 
  
The term on the  left equals $\frac 1n \sum_{|P(\be(x_i))|< r} \left(\log
r - \log|P(\be(x_i))|\right)$. Observe that $-\log|P(\be(x_i))| = \bigl|\log
|P(\be(x_i))|\bigr|$ in this sum as $r\le 1$. We rearrange and find
\begin{equation*}
  \frac 1n \sum_{|P(\be(x_i))|< r}   \bigl|\log|P(\be(x_i))|\bigr|
  \ll_{d,k} \deg(P)\cD^{1/(d+1)}r^{-1} +r^{1/(4k)} +\frac{|\!\log r|}{n}
\left(  \sum_{|P(\be(x_i))|< r}1\right) + \epsilon. 
\end{equation*}
By Lemma \ref{lem:countsmallvalues}, the term corresponding to
the sum  over $i$ on the right is
$\ll_{d,k} r^{1/(2k)}|\!\log r| + \deg(P)\cD^{1/(d+1)}r^{-1}|\!\log r|$.
Combining our bounds and absorbing $|\!\log r|$ in an appropriate power of
$r^{-1}$ 
we find
\begin{equation*}
  \frac 1n \sum_{|P(\be(x_i))|< r}   \bigl|\log|P(\be(x_i))|\bigr|
  \ll_{d,k} \deg(P)\cD^{1/(d+1)}r^{-2} +r^{1/(4k)} + \epsilon,
\end{equation*}
as desired. 
\end{proof}

After this warming-up we prove variants of
Lemmas \ref{lem:logPmodcont} and \ref{lem:logvsintpsi} where 
$\log|\cdot|$ is replaced by the Mahler measure.
We also truncate  at the  parameter $r$. 

\begin{lemma}
  \label{lem:mahlermeasuremodcont}
  Let $r \in (0,1]$, for $x\in[0,1]^l$ we define
  $\mu(x) = \max\{\log r, m(P_{\be(x)})\}$ as above
  (\ref{eq:maxcpsibound}) where we interpret the Mahler measure of $0$
  as $-\infty$. 
  Then $\mu:[0,1]^l \rightarrow \IR$ is continuous and for all $t>0$ we have
  \begin{equation*}
    \omega(\mu;t) \ll_{d,k}
    \left(\frac{\deg(P)t}{r}\right)^{1/(8k)}(1+|\!\log r|).
  \end{equation*}
\end{lemma}
\begin{proof}
  By Boyd's Theorem~\cite{Boyd} the Mahler measure is continuous on
  the space of non-zero polynomials of bounded degree. Thus $\mu$ is
  continuous on $[0,1]^l$. Observe that $\omega(\mu;t) \ll_k 1+|\!\log
  r|$ as $m(P_{\be(x)})\ll_k 1$ by (\ref{eq:mPlogPub}). So we may assume
  that $\deg(P)t/r$ is sufficiently small in terms of $d$ and $k$.

  \refcomment{92}{We again use (\ref{eq:maxcpsibound}), this time with
  $\theta(x) = m(P_{\be(x)})$ and $c= \log r$. Let $x,y\in [0,1]^d$
  with $m(P_{\be(x)}) \ge \log r$ and $m(P_{\be(y)}) \ge \log r$ and $|x-y|\le t$.
   Then $|P_{\be(x)}|\gg_k r$ and $|P_{\be(y)}|\gg_k r$ by
  (\ref{eq:mPlogPub}). 
    As in the proof of Lemma
  \ref{lem:countsmallvalues} we find $|P_{\be(x)} - P_{\be(y)}|\ll_{d,k}
  \deg(P)t$. Since $\deg(P)t/r$
  is smaller than some prescribed
  constant depending only on $d$ and $k$ we
  have $|P_{\be(x)}-P_{\be(y)}| /
  \min\{|P_{\be(x)}|,|P_{\be(y)}|\} \le 1/2$.
  Lemma \ref{lem:mahlerhoelder} implies
  \begin{equation*}
    |m(P_{\be(x)})-m(P_{\be(y)})|\ll_{d,k}
    \left(\frac{|P_{\be(x)}-P_{\be(y)}|}
      {\min\{|P_{\be(x)}|,|P_{\be(y)}|\}}\right)^{1/(8(k-1))}
    \ll_{d,k}
    \left(\frac{\deg(P)t}{r}\right)^{1/(8(k-1))}, 
  \end{equation*}
  as desired.  
  }
  %
  %
  %
  %
\end{proof}

\refcomment{103}{By Lemma~\ref{lem:Qprops0}(i) 
$\widehat P$ has at
most $k^2$ non-zero terms, so $\sup_{x\in [0,1]^l} |\widehat
P(\be(x))|\le k^2 |\widehat P|$.
We let $\widetilde P$ denote the polynomial from part (ii) of the said
lemma divided by $|P|$, so $|\widetilde P|=1$. }
There exists $i$ with $|p_i|=|P|=1$.
The definition of the Mahler measure implies
\begin{equation*}
  m(p_i)\le \sup_{x\in [0,1]^l} \log |p_i(\be(x))| \le \frac 12
  \sup_{x\in [0,1]^l} \log |\widehat P(\be(x))|\le \frac{1}{2}(2\log k +
  \log|\widehat P|). 
\end{equation*}
Using $|p_i|=1$ and the Theorem of Dobrowolski--Smyth, Theorem
\ref{thm:dobsmyth}, we conclude $m(p_i)\ge - (k-2)\log 2$. Thus
$|\widehat P|\gg_k 1$. Bounding $|\widehat P|$ from above is more
straight-forward. Indeed, $|\widehat P| \ll_k 1$ by (\ref{eq:defQpi})
and since $|P|=1$. Therefore,
\begin{equation}
  \label{eq:Qnormbound}
  1\ll_k |\widehat P| \ll_k 1. 
\end{equation}

\begin{lemma}
  \label{lem:mahlervsintvarphi}
  We keep the notation of Lemma \ref{lem:mahlermeasuremodcont}.
  Then
  \begin{equation*}
    \left|
      m(P)- \int_{[0,1)^l}
      \mu(x)dx\right| \ll_{d,k} r^{1/(2k^2)}. 
  \end{equation*}
\end{lemma}
\begin{proof}    
  We recall that $|\widehat P|\widetilde P$ equals $\widehat P$ up-to
  a monomial factor. By (\ref{eq:cmpQbexPbex}), (\ref{eq:Qnormbound}),
  and Theorem \ref{thm:dobsmyth} there exists $c>0$ depending only on
  $k$ such that $|\widetilde P(\be(x))|\ge cr^2$ implies
  $m(P_{\be(x)})\ge \log r$. By Fubini's Theorem we have
  $\int_{[0,1)^l}m(P_{\be(x)})dx = m(P)$, so the absolute value in
  question is
  \begin{equation*}
    \cE = \left|\int_{\Sigma}m(P_{\be(x)})dx -
      \vol{\Sigma}\log r\right|
  \end{equation*}
  where $\Sigma= S(\widetilde P,cr^2)$; indeed
  $m(P_{\be(x)})=\mu(x)$ for all $x\in [0,1]^l\ssm \Sigma$.

  Note that $\vol{\Sigma}\ll_{d,k} r^{1/(k^2-1)}$ by Lemma
  \ref{lem:volSPep}(i) applied to $\widetilde P$. So
  \begin{equation}
    \label{eq:bounddiffintegrals}
    \cE\ll_{d,k} r^{1/(k^2-1)}|\!\log r| +
    \left|\int_{\Sigma}m(P_{\be(x)})dx\right|
    \ll_{d,k} r^{1/k^2} +
    \int_{\Sigma}\left|m(P_{\be(x)})\right| dx. 
  \end{equation}

  To bound the integral in (\ref{eq:bounddiffintegrals}) from
  above we will
  replace $m(P_{\be(x)})$ by $\log|P_{\be(x)}|$. 
  Say $x\in \Sigma$ and $P_{\be(x)}\not=0$,
  then $\left|m(P_{\be(x)}) - \log|P_{\be(x)}| \right|\ll_k 1$ by
  (\ref{eq:diffmahlerlog}) and thus
  $|m(P_{\be(x)})| \ll_k 1 + \bigl|\log|P_{\be(x)}|\bigr|$.
  \refcomment{95}{The function $x\mapsto \bigl|\log|P_{\be(x)}|\bigr|$ is
    integrable over $[0,1)^l$ in the sense of Lebesgue  and  so is
    $x \mapsto |m(P_{\be(x)})|$; both take the value
    $+\infty$ on a measure zero subset of $[0,1)^l$. We find}
  \begin{alignat*}1
    \cE&\ll_{d,k}
    r^{1/k^2} +
    \int_{\Sigma}\left(1+\bigl|\log|P_{\be(x)}| \bigr|\right)dx.
  \end{alignat*}
  From (\ref{eq:cmpQbexPbex}) and (\ref{eq:Qnormbound}) we deduce
  $\bigl| \log |P_{\be(x)}|\bigr| \ll_k \bigl|\log |\widetilde
  P(\be(x))|\bigr| + 1$ if $\widehat P(\be(x))\not=0$. So $\cE\ll_{d,k}
  r^{1/k^2} + \int_{\Sigma}\left(1+\bigl|\log|\widetilde
    P(\be(x))|\bigr|\right)dx$. By Lemma \ref{lem:Spdeltaintegral} applied
  to $\widetilde P$ and the volume estimate for $\Sigma$, the integral
  on the right is $\ll_{d,k} r^{1/(2(k^2-1))} \ll_{d,k} r^{1/(2k^2)}$,
  as desired.
\end{proof}




\begin{proof}[Proof of Proposition \ref{prop:lehmermean}]
If we scale $P$  by a factor $\lambda$, then $\widehat P$ is scaled by
$|\lambda|^2$.
So the proposition is invariant under non-zero scaling and
we may assume $|P|=1$.
Later on we will choose the parameter $r$ in terms of $\deg(P)$ and
$\cD$. In the meantime we  assume that $r\in (0,1/2]$.

Observe that $\int_{[0,1)^l} m(P(\be(x)))dx = m(P)$. We want to bound
 $\cE = |m(P)- n^{-1} \sum_{i=1}^n m(P_{\be(x_i)})|$ from above.

We replace the Mahler measure with $\mu(\cdot)$ coming from
Lemma \ref{lem:mahlermeasuremodcont}. Indeed,  the
triangle inequality implies
\begin{alignat*}1
 \cE &\le \left|m(P)- \int_{[0,1)^l} \mu(x)dx\right| +
    \left|\int_{[0,1)^l} \mu(x)dx - \frac 1n \sum_{i=1}^n
      \mu(x_i)\right|+
          \left| \frac 1n \sum_{i=1}^n \mu(x_i)-m(P_{\be(x_i)})\right|.
\end{alignat*}
The first term on the right is $\ll_{d,k} r^{1/(2k^2)}$ by Lemma
\ref{lem:mahlervsintvarphi} applied to $P$. By Proposition \ref{prop:numintegration}
applied to $\mu$ and $t=\cD^{1/(d+1)}$
and Lemma \ref{lem:mahlermeasuremodcont}  the
second term is $\ll_{d,k} (\deg(P)\cD^{1/(d+1)}r^{-1})^{1/(8k)}|\!\log r|$.
So 
\begin{equation}
  \label{eq:firsterrorterm}
\cE\ll_{d,k} r^{1/(2k^2)} +
(\deg(P)\cD^{1/(d+1)}r^{-2})^{1/(8k)}+\cE'
\end{equation}
after absorbing $|\!\log r|$ in a multiple $r^{-1/(8k)}$ and
where $\cE'$ is the third term above.  
Only terms with
$m(P_{\be(x_i)})\le\log r$ contribute to the average, so $\cE'$ equals
\begin{equation*}
 \left|\frac 1n \sum_{m(P_{\be(x_i)})\le \log r} \log r -
 m(P_{\be(x_i)})\right| \le
 \frac{|\!\log r|}{n} \#\left \{ i : m(P_{\be(x_i)}) \le \log r\right\}
 + \frac 1n \sum_{m(P_{\be(x_i)})\le \log r}
 |m(P_{\be(x_i)})|. 
\end{equation*}
By (\ref{eq:diffmahlerlog}) we may replace $m(P_{\be(x_i)})$ by
$\log|P_{\be(x_i)}|$ at the cost of introducing  a constant $c_1>0$
depending only on  $k$, \textit{i.e.},
\begin{equation*}
\cE'\ll_{d,k} \frac{|\!\log r|}{n} \#\left\{ i : |P_{\be(x_i)}|
   \le c_1 r \right\} + \frac 1n \sum_{|P_{\be(x_i)}|\le c_1 r}
\bigl(1+ \bigl|\log|P_{\be(x_i)}|\bigr|\bigr).
\end{equation*}

If $|P_{\be(x_i)}|\le c_1r$, then
$|\widetilde P(\be(x_i))| = |\widehat P(\be(x_i))| / |\widehat P|\le c_2 r^2$ for some $c_2$ depending only on $k$  by
(\ref{eq:cmpQbexPbex}) and (\ref{eq:Qnormbound}).
The same inequalities imply  $\bigl|\log|P_{\be(x_i)}|\bigr| \ll_{k}
\bigl|\log|\widetilde P(\be(x_i))|\bigr|+1$, the ``$+1$'' is absorbed in
the first term in 
 \begin{equation*}
\cE'\ll_{d,k} \frac{|\!\log r|}{n} \#\left\{ i : |\widetilde P(\be(x_i))|
   \le c_2r^2 \right\} + \frac 1n \sum_{|\widetilde P(\be(x_i))|\le c_2r^2}
 \bigl|\log|\widetilde P(\be(x_i))|\bigr|. 
\end{equation*}

Recall  that $\deg{\widetilde P} \ll_d \deg P$
and that $\widetilde P$ has at most
$k^2$ terms and norm $1$.
Lemma \ref{lem:countsmallvalues} applied
to $\widetilde P$
implies
\begin{alignat*}1
  \cE'&\ll_{d,k} r^{1/k^2}|\!\log r|+\deg(P) \cD^{1/(l+1)}r^{-2}|\!\log r| + \frac 1n
  \sum_{|\widetilde P(\be(x_i))| \le c_2r^2}
  \bigl|\log|\widetilde P(\be(x_i))|\bigr|.
\end{alignat*}

We use Lemma \ref{lem:abslogPsumbound}, applied to $\widetilde P$ and $c_2r^2$, to bound the final sum and thus obtain
\begin{equation*}
  \cE'\ll_{d,k} r^{1/k^2}|\!\log r| + \deg(P)\cD^{1/(l+1)}r^{-2}|\!\log r| +
  \deg(P)\cD^{1/(l+1)}r^{-4} + r^{1/(2k^2)}+\epsilon,
\end{equation*}
here $\epsilon=\left| m( \widehat P)-\frac 1n \sum_{i=1}^n \log|
 \widehat P(\be(x_i))\right|$; note that  multiplying $\widehat P$ with a non-zero
 scalar and a monomial leaves $\epsilon$ invariant.

 We return to the total error term $\cE$. By
 (\ref{eq:firsterrorterm}) together with $l\le d,r\le 1,$ and $\cD\le 1$ we get
\begin{equation*}
  \cE\ll_{d,k} r^{1/(2k^2)} +
  (\deg(P)\cD^{1/(d+1)}r^{-4})^{1/(8k)}
  +\deg(P)\cD^{1/(d+1)} r^{-4} + \epsilon.
\end{equation*}

We choose $r = \frac 12 \cD^{1/(8(d+1))}$, then the proposition
follows as $\cD\le 1$.
%
%
\end{proof}


\section{Endgame}\label{sec:endgame}

In this section we prove a stronger version of Theorem \ref{thm:main}
from the introduction.

\subsection{Preliminaries}

Suppose $P\in \IQbar[X_1^{\pm 1},\ldots,X_d^{\pm 1}]\ssm \{0\}$. For
$V\in \gl{d}(\IZ)$ we set $Q\in \IQbar[X_1,\ldots,X_d]$ to be
$P(X^{V^{-1}})$ multiplied by a suitable monomial in $X_1,\ldots,X_d$
such that $Q$ is coprime to $X_1\cdots X_d$. Let $l\in
\{0,\ldots,d-1\}$. For $z=(z_1,\ldots,z_l)\in \IC^l$ we set
\begin{equation}
  \label{eq:definePVz}
  P_{V,z} = Q(z_1,\ldots,z_l,X_1,\ldots,X_{d-l})
\end{equation}
this is a polynomial in $d-l$ variables. \refcomment{97}{It
 is sometimes useful to  allow  $l=0$ in which case
 $P_{V,z}=Q$.}

The following lemma requires a result of Bombieri, Masser, and
Zannier~\cite{BMZGeometric} and relies crucially on the hypothesis
that $P$ is \essatoral{}.

\begin{lemma}
  \label{lem:Petaprops}
  \refcomment{101}{Suppose $P\in\IQbar[X_1^{\pm 1},\ldots,X_d^{\pm}]\ssm\{0\}$ is \essatoral{}.
    There exists $c\ge 1$ depending only on $P$ and $d$ such that for all
    $\bzeta\in\IG_m^d$ with $\delta(\bzeta)\ge c$, for all $V\in\gl{d}(\IZ)$, and all $l\in
    \{0,\ldots,d-1\}$, we have
    $P_{V,\bfeta}\not=0$ and
    $\cBprim(P_{V,\bfeta})\le c |V^{-1}|$ where $\bzeta^V =
    \{\bfeta\}\times\IG_m^{d-l}$.}
\end{lemma}
\begin{proof}
  The Zariski closure $W$ in $\IG_m^d$ of  all algebraic
  zeros of $P$  in $(S^1)^d$ is defined over $\IQbar$. 

  By hypothesis, $P$ is \essatoral{}. So each irreducible component of
  the Zariski closure of all complex roots of $P$ on $(S^1)^d$ is
  of codimension at least $2$ in $\IG_m^d$
  or a proper torsion coset of $\IG_m^d$.
  Therefore,
  each
  irreducible component of $W$ is also
  of this type. 

  Let $\bzeta\in\IG_m^d$ be of finite order with $\delta(\bzeta)\ge
  c$, where $c$ is to be determined, 
  and $\bzeta^V = \{\bfeta\}\times\IG_m^{d-l}$ with $V$ and $l$ as in
  the hypothesis.

  Let $\bfeta'\in\IG_m^{d-l}$ be of finite order, $z\in
  S^1\ssm\mu_\infty$ be algebraic, and
  $u\in\IZ^{d-l}$ with $P_{V,\bfeta}(\bfeta' z^u)=0$.
  We must find $v''\in\IZ^{d-l}\ssm\{0\}$ with  $|v''|\le c |V^{-1}|$
  such that $\langle u,v''\rangle=0$.
  \refcomment{101}{The existence of such a $v''$ establishes in
    particular $P_{V,\bfeta}\not=0$.}

  Now $P(x)=0$ for the algebraic point $x=(\bfeta,\bfeta' z^u)^{V^{-1}}
  \in (S^1)^d$. So $x$ is contained in an irreducible component $W'$ of
  $W$
  \refcomment{98}{and in a $1$-dimensional
    algebraic subgroup of $\IG_m^d$.}

  If $\dim W'\le d-2$,  we apply Bombieri, Masser, and Zannier's Theorem
  1.5~\cite{BMZGeometric} to $\mathcal{X} = W'$.
  We get a proper torsion coset of $\IG_m^d$ containing $x$ and coming from a finite
  set depending only on $W'$.  We find
  $v\in\IZ^d\ssm\{0\}$ with $|v|\ll_{d,P} 1$ and
  $x^v=1$.

  If $W'$ is a proper torsion coset of $\IG_m^d$
  there exists $v\in\IZ^d\ssm\{0\}$,
  depending only on $W'$ and thus only on $P$, such that
  $y^v=1$ holds  for all $y\in W'$. Again we find $|v|\ll_{d,P} 1$ and
  $x^v=1$.

  In either case we have
  \begin{equation}
    \label{eq:oneequalsxpowera}
    1=x^v = (\bfeta,\bfeta'z^u)^{V^{-1}v}= \bfeta^{v'} (\bfeta'
    z^u)^{v''}
    \text{ where }
    V^{-1}v = \left(
      \begin{array}{cc}
        v'\\ v'' \end{array}\right)\in\IZ^l\times\IZ^{d-l}.
  \end{equation}
  In particular, $\langle u,v''\rangle=0$ as $z$ has infinite order. 

  If $v''\not=0$, then we are done. Indeed, $|v''|\le |V^{-1}v| \le
  d|V^{-1}||v|$ and $|v|$
  is bounded from above solely in terms of $P$ and $d$. 

  \refcomment{99}{Let us assume $v''=0$ and derive a contradiction.}
  Note $l\ge 1$ as  $v$ cannot be $0$.
  Then $v'\not=0$ and  by equality (\ref{eq:oneequalsxpowera})  we find
  $\bfeta^{v'}=1$.
  Recall that $\bfeta$ consists of the first $l$ coordinates of $\bzeta^V$. 
  Thus $\bzeta^v = 1$ and hence $\delta(\bzeta)\le |v|$ 
  \refcomment{100}{where $|v|\ll_{d,P} 1$. But $\delta(\bzeta)\ge c$,
    a contradiction for large enough $c$.}
\end{proof}

\begin{defin}
  Let $c\ge 1$  be a real number. Suppose
  $P\in \IQbar[X_1^{\pm 1},\ldots,X_d^{\pm 1}]\ssm \{0\}$ and 
  $\bzeta\in\IG_m^d$ is of finite order. 
  The pair $(P,\bzeta)$ is called
  $c$-admissible if
  for all 
  $V\in\gl{d}(\IZ)$ and all $l\in \{0,\ldots,d-1\}$ we have 
  $P_{V,\bfeta}\not=0$ and $\cBprim(P_{V,\bfeta})\le c |V^{-1}|$
  where $\bzeta^V \in \{\bfeta\}
  \times \IG_m^{d-l}$.
\end{defin} 

The case $l=0$ yields in  particular $\cBprim(P)\le c$ if there exists
$\bzeta$ such that $(P,\bzeta)$
is $c$-admissible.

Let $P$ be an  \essatoral{} Laurent polynomial
with algebraic coefficient. 
By Lemma \ref{lem:Petaprops} there exists $c\ge 1$ such that
$(P,\bzeta)$ is $c$-admissible for all $\bzeta\in\IG_m^d$ of finite order with
$\delta(\bzeta)\ge c$. 

In the definition of admissibility, it will be useful to keep track of
$\bzeta$ when passing in down in an induction step. The next lemma
makes this precise. 

\begin{lemma}
  \label{lem:Qprops2}
  Let $P\in \IQbar[X_1^{\pm 1},\ldots,X_d^{\pm 1}]\ssm\{0\}$ and let
  $\bzeta\in\IG_m^d$ be of finite order such that $(P,\bzeta)$ is 
  $c$-admissible with $c\ge 1$. 
Say $l\in \{0,\ldots,d-1\},V\in\gl{d}(\IZ),$ and $\bzeta^V = (\bfeta,\bxi)$ with
$\bfeta\in\IG_m^l$ and $\bxi\in\IG_m^{d-l}$. Then 
$(P_{V,\bfeta},\bxi)$ is $(cd|V^{-1}|)$-admissible.
\end{lemma}
\begin{proof}
  Throughout the proof we use that $|\cdot|$ is the maximum-norm on
  matrices.

  We abbreviate $R = P((\bfeta,X_1,\ldots,X_{d-l})^{V^{-1}})$ which equals
  $P_{V,\bfeta}$ up-to a monomial factor. It suffices to show that
  $(R,\bxi)$ is $c$-admissible. 

  To this end say $k\in \{0,\ldots,d-l-1\},
  W\in\gl{d-l}(\IZ),$ and ${\bxi}^{W} =
  \{\bfeta'\}\times\IG_m^{d-l-k}$
  with $\bfeta'\in\IG_m^k$. We must bound $\cBprim(R_{W,\bfeta'})$. 
  So say $z\in S^1\ssm\mu_\infty, u\in\IZ^{d-l-k},$ and
  $\bfeta'' \in\IG_m^{d-l-k}$ is of finite order with
  $R_{W,\bfeta'}(\bfeta'' z^u)=0$. Thus $R((\bfeta',\bfeta''
  z^u)^{W^{-1}})=0$ and hence 
  $P\left((\bfeta,(\bfeta',\bfeta''z^u)^{W^{-1}})^{V^{-1}}\right)=0$. 
  We abbreviate $W' = \left(
    \begin{array}{cc}
      E_l & 0 \\ 0 & W
    \end{array}\right)$ with $E_l$ the $l\times l$ unit matrix.
  So $P\left((\bfeta,\bfeta',\bfeta'' z^u)^{(VW')^{-1}}\right)=0$ which
  means
  $P_{VW',(\bfeta,\bfeta')}(\bfeta'' z^u)=0$. 

  Observe that $\bzeta^{VW'} = (\bfeta,\bxi)^{W'} =
  (\bfeta,{\bxi}^{W})=(\bfeta,\bfeta',*)$. By hypothesis $(P,\bzeta)$ is
  $c$-admissible. Therefore, $\cBprim(P_{VW',(\bfeta,\bfeta')})\le c
  |(VW')^{-1}| = c|{W'}^{-1} V^{-1}|\le cd |V^{-1}||{W'}^{-1}|=cd
  |V^{-1}| |W^{-1}|$. In other words, there exists
  $v\in\IZ^{d-l-k}\ssm\{0\}$ with 
  $|v|\le cd
  |V^{-1}| |W^{-1}|$ and $\langle u,v\rangle=0$. Thus
  $\cBprim(R_{W,\bfeta'})\le cd|V^{-1}||W^{-1}|$, as desired.  In
  particular, $R_{W,\bfeta'}\not=0$. 
\end{proof}

\begin{lemma}
  \label{lem:Qprops1}
  Let $P\in \IQbar[X_1^{\pm 1},\ldots,X_d^{\pm 1}]\ssm\{0\}$ and let
  $\bzeta\in\IG_m^d$ be of finite order such that $(P,\bzeta)$ is 
  $c$-admissible with $c\ge 1$. 
  Say $l\in
  \{1,\ldots,d-1\}$ and
  let  $\widehat P \in \IQbar[X_1^{\pm
    1},\ldots,X_l^{\pm 1}]$ be as in (\ref{eq:defQpi}) and 
  $\bzeta\in \{\bfeta\} \times\IG_m^{d-l}$.
  Then $(\widehat P,\bfeta)$ is $c$-admissible.
\end{lemma}
\begin{proof}
Suppose $V\in
  \gl{l}(\IZ)$ such that $\bfeta^V = (\bfeta',*)$ where $\bfeta\in
  \IG_m^{l'}$ where $l' \in \{0,\ldots,l-1\}$.
  Following the definition of admissibility and recalling
  (\ref{eq:definePVz}) we are in the following situation.
  There is $\bfeta'' \in \IG_m^{l-l'}, z\in S^1\ssm\mu_\infty$ algebraic, and $u'\in \IZ^{l-l'}$ such that
  \begin{equation*}
    \widehat P((\bfeta',\bfeta'' z^{u'})^{V^{-1}})=0.
  \end{equation*}

It follows from the definition of $\widehat P$  that $P((\bfeta',\bfeta''
  z^{u'})^{V^{-1}},X_{l+1},\dots,X_d)=0$ as a polynomial in
  $X_{l+1},\ldots,X_d$.
  We extend $\widetilde V = \left(
  \begin{array}{cc}
    V & 0 \\ 0 & E_{d-l} 
  \end{array}\right)$ where $E_{d-l}$ is the $(d-l)\times(d-l)$ unit matrix.
  Then $P((\bfeta',\bfeta'' z^{u'},z^{u''})^{{\widetilde V}^{-1}})=0$
  for all $u''\in\IZ^{d-l}$.

  By hypothesis, $(P,\bzeta)$ is $c$-admissible and
  $\bzeta^{\widetilde V} = (\bfeta^V,*) = (\bfeta',*,*)$. 
  Now $P_{\widetilde V,\bfeta'}(\bfeta'' z^{u'},z^{u''})=0$, so by
  definition 
  there exist $v'\in\IZ^{l-l'},v''\in\IZ^{d-l}$, not both zero, 
  such that $\langle u',v'\rangle + \langle u'',v''\rangle =0$
and $  |(v',v'')|\le c
  |{\widetilde V}^{-1}| =c |V^{-1}|$ for the maximum-norm.

As we are free to vary $u''$ we see
 that $\{u'\}\times\IQ^{d-l}$ is contained in a finite union of proper
vector subspaces of $\IQ^d$, each defined as the kernel of
$\langle \cdot, (v',v'')\rangle$ with $v',v''$ as above.
So $\{u'\}\times\IQ^{d-l}\subset V$ for 
one of these vector spaces $V$ defined by some $(v',v'')$.
We must have $v''=0$  and hence $\langle u',v'\rangle = 0$.
Then $v'\not=0$ and as $|v'|\le c|V^{-1}|$  we conclude that  $\widehat P$ is
$c$-admissible. 
%
%
%
\end{proof}

Here are some basic estimates involving $P_{V,\bfeta}$. 
\begin{lemma}
  \label{lem:basicPVeta}
  Let $P\in \IQbar[X_1,\ldots,X_d]\ssm\{0\},l,$ and $V$ be as near the
  beginning of this subsection. Say $\bfeta\in\IG_m^l$ has finite order and
  $P_{V,\bfeta}\not=0$. The following hold true.
  \begin{enumerate}
  \item [(i)]  We have $\deg P_{V,\bfeta}\ll_d
    |V|^{d-1}\deg P$.
  \item [(ii)]
    We have $\hproj{P_{V,\bfeta}}\le \log(k)+\hproj{P}$ where $k\ge 2$
    is an upper bound for the number of non-zero terms of $P$.
  \end{enumerate}
\end{lemma}
\begin{proof}
  Both parts follow are  elementary consequences of the degree and the
  height of a polynomial. For (i) we require $|V^{-1}|\ll_d
  |V|^{d-1}$.
  
  For (ii) we note that $Q$ from the beginning of this
  \refcomment{102}{subsection}
  has
  the same coefficients and thus the same height as $P$. We decompose
  $\hproj{P_{V,\bfeta}}$ in local heights as into (\ref{def:hprojP}).
  The triangle inequality at the archimedean places leads to $\log k$.
\end{proof}

We continue with further basic estimates involving $\widehat P$
as in   (\ref{eq:defQpi}).

\begin{lemma}
  \label{lem:Qprops}
  \refcomment{103}{Let $K\subset\IC$ be a number field and suppose $P\in K[X_1,\ldots,X_d]\ssm\{0\}$
  has at most $k\ge 2$ terms, where $k$ is an integer.} Say $l\in \{1,\ldots,d-1\}$
and suppose $\widehat P \in \IC[X_1^{\pm 1},\ldots,X_l^{\pm 1}]$
then the following properties
  hold true:
  \begin{enumerate}
  \item[(i)] We have $\widehat P \in K'[X_1^{\pm 1},\ldots,X_l^{\pm 1}]$ where
    $K'$ is a number field such that $ K\subset K'\subset\IC$ and
    $[K':\IQ]\le [K:\IQ]^2$.
  \item[(ii)] We have $\hproj{\widehat P}\ll_k 1 + \hproj{P}$.
  \end{enumerate}
\end{lemma}
\begin{proof}
  %
  To see (i) we note that the coefficients of $\widehat P$ are contained in
  the subfield $K'$ of $\IC$  generated by all elements of $K$ and
  their complex conjugates. 

  Finally, for (ii)  we remark that each $p_i$ as in (\ref{eq:defQpi})
  has at most $k$ terms and that there are at most $k$ non-zero $p_i$.
  Using the local decomposition of the height together with the
  ultrametric and archimedean  triangle inequality
  yields the claim. 
\end{proof}

\subsection{Completion of proof}

The next lemma will setup a monomial change of coordinates. We recall
that $\widetilde \Lambda_\bxi(\nu)$ was defined in (\ref{def:tildeLambdazeta}) and
$\lambda_1(\widetilde \Lambda_\bxi(\nu))$ in (\ref{def:firstminimum}).

\begin{lemma}
\label{lem:monomialcoc}
Suppose $\bzeta\in\IG_m^d$ has  order $N$
and let $\delta\ge 1, \epsilon\in (0,1/2],\nu_1,\ldots,\nu_{d-1} \in
(0,1/2]$ with $\nu_1+\cdots+\nu_{d-1}\le 1/2$.
Then there exist $l\in \{0,\ldots,d-1\}$ and $V\in\gl{d}(\IZ)$
such that the following hold.
\begin{enumerate}
\item [(i)] We have $|V|\ll_d \delta^{2\epsilon^{d-l}}$   
and $V$ is the unit matrix if $l=0$. 
\item[(ii)] We have
$\bzeta^V = (\bfeta,\bxi)$ where $\bfeta\in\IG_m^l,
  \bxi\in\IG_m^{d-l},\ord(\bfeta)\le N^{\nu_1+\cdots+\nu_l},$ 
  and either $l=d-1$ and $\bxi\in\mu_\infty$ has order at least $N^{1/2}$ or
$l\le d-2$ and  $\lambda_1(\widetilde\Lambda_{\bxi_{l}}(\nu_{l+1}))>\delta^{\epsilon^{d-l-1}}$.
\end{enumerate}
\end{lemma}
\begin{proof}
Set $\bxi_1=\bzeta$ and let $V_0$ be the unit matrix in $\gl{d}(\IZ)$.
For all
$l\in \{1,\ldots,d-1\}$ with $\lambda_1(\widetilde\Lambda_{\bxi_l}(\nu_l))\le
\delta^{\epsilon^{d-l}}$ we will construct inductively $V_l\in\gl{d}(\IZ),\bxi_{l+1} \in
\IG_m^{d-l}$ of order at most $N,$ and $\eta_l \in\IG_m$ of order at most $N^{\nu_l}$ such
that $\bzeta^{V_l} = (\eta_1,\ldots,\eta_l,\bxi_{l+1})$ and
\begin{equation}
\label{eq:Vlbound}
|V_l| \ll_d \delta^{\epsilon^{d-1}+\cdots +\epsilon^{d-l}}.
\end{equation}

Suppose  $\lambda_1(\widetilde\Lambda_{\bxi_l}(\nu_l))\le \delta^{\epsilon^{d-l}}$, 
 there exists $v\in \widetilde\Lambda_{\bxi_l}(\nu_l)\ssm\{0\}$
 such that $|v|\le \delta^{\epsilon^{d-l}}$ and $v$ is primitive.
 Note that
 $[\widetilde\Lambda_{\bxi_l}(\nu_l):\Lambda_{\bxi_l}(\nu_l)]v$
 lies in $\Lambda_{\bxi_l}$, so 
 $\ord(\bxi_l^v)
 \le[\widetilde\Lambda_{\bxi_l}(\nu_l):\Lambda_{\bxi_l}(\nu_l)]\le
 \det(\Lambda_{\bxi_l}(\nu_l))\le N^{2\nu_l^{2}}\le N^{\nu_l}$ by
 (\ref{eq:detLambdaub}) and since $\det(\Lambda_\bxi) \le N$. 
As in the proof of Proposition  \ref{prop:geometryofnumbers} we can realize  $v$ as the first column of a matrix
$V'_l \in \gl{d+1-l}(\IZ)$ 
 with $|V'_l|\ll_d |v|\ll_d \delta^{\epsilon^{d-l}}$. 
 Let $E_{l-1}$
denote the $(l-1)\times(l-1)$ unit matrix and set
\begin{equation*}
V_l = V_{l-1}\left(
\begin{array}{cc}
E_{l-1} & 0\\
0 &  V'_l
\end{array}\right) \in \gl{d}(\IZ).
\end{equation*}
By  step $l-1$ we have $\bzeta^{V_{l-1}} = (\eta_1,\ldots,\eta_{l-1},\bxi_l)$.
We define $\bxi_{l+1}$ via 
$\bzeta^{V_l}=(\eta_1,\ldots,\eta_l,\bxi_{l+1})$,
note $\eta_l = \bxi_l^v$.
Finally, $|V_l|\ll_d
|V_{l-1}||V'_l|\ll_d \delta^{\epsilon^{d-1}+\cdots+\epsilon^{d-l}}$,
which  completes our construction.

Let us consider the largest $l\in \{0,\ldots,d-1\}$
for which $\lambda_1(\widetilde \Lambda_{\xi_{l}}(\nu_l))\le \delta^{\epsilon^{d-l}}$ and define $V=V_l$.
Then (i) holds by (\ref{eq:Vlbound}) as $\epsilon \le 1/2$.
To verify (ii) observe that 
 $\bzeta^{V}=(\eta_1,\ldots,\eta_{l},\bxi)$ with $\bxi=\bxi_{l+1}\in\IG_m^{d-l}$ has
 order $N$
and $(\eta_1,\ldots,\eta_l)$ has  order at most
$N^{\nu_1+\cdots+\nu_l}\le  N^{1/2}$.
If $l=d-1$, then $\bxi$ is a root of unity of order at least
$N^{1/2}$. Otherwise $l\le d-2$ and
$\lambda_1(\widetilde\Lambda_{\bxi_{l+1}}(\nu_{l+1}))> \delta^{\epsilon^{d-l-1}}$, because
the construction cannot continue.
\end{proof}

We are ready to prove a theorem that will quickly imply our
Theorem \ref{thm:main} and its refinements. 

\begin{thm}
  \label{thm:main2}
  Let $c\ge 1$,
let $K\subset\IC$ be a number field, and suppose  $P\in
K[X_1,\ldots,X_d]\ssm \{0\}$ 
has at most $k$ terms for an integer $k\ge 2$.
There are constants $C=C(d,k) \ge 1$ and $\kappa=\kappa(d,k)>0$ depending only on $d$ and $k$
with the following property. 
Let $\bzeta\in\IG_m^d$ have   order $N$ 
and suppose $G$
is a subgroup of $\GammaN$ 
with $P(\bzeta^\sigma)\not=0$ for all $\sigma\in G$.
If $(P,\bzeta^\sigma)$ is
$c$-admissible for all $\sigma\in G$ and if
\begin{equation}
\label{eq:deltazetahypinthm}
  \delta(\bzeta) \ge C \max\{c,\deg P\}^{C}
\end{equation}
then
\begin{equation*} 
  \frac{1}{\#G} \sum_{\sigma \in G} \log |P(\bzeta^\sigma)|
   = m(P) + O_{d,k}\left(\frac{[K:\IQ]^{2^d} [\GammaN:G]^2 \mf_G
     \deg(P)^{16d^2} (1+\hproj{P})}{\delta(\bzeta)^\kappa}\right).
\end{equation*}
\end{thm}
\begin{proof}  
  The case $d=1$ follows from
  Proposition~\ref{prop:etaxifactorization}(i) as $\delta(\bzeta)=N$
  in this case and as $\cBprim(P)<\infty$. 
  So we may assume
  $d\ge 2$. We may also assume that $P$ is non-constant.
  
  We work with the parameters
  $\nu_1,\ldots,\nu_{d-1} \in (0,1/(128d^2)], \epsilon \in (0,1/2]$ in
  this proof. They are assumed to be small in terms of $d$ and $k$ but
  independent of $P$ and $\bzeta$. We may assume that $\epsilon$ is
  small in terms of the $\nu_l$, \textit{e.g.}, $\epsilon \le \nu_l^d$
  for all $l$. We determine them during the argument.

 We apply Lemma \ref{lem:monomialcoc}
to $\bzeta,\delta=\delta(\bzeta),\epsilon,$ and the $\nu_l$. Say
$l,V,\bfeta,$ and $\bxi$ are given by this
lemma, in particular $\bzeta^V = (\bfeta,\bxi)$ and $|V|\ll_d \delta(\bzeta)^{2\epsilon^{d-l}}$. We have
\begin{equation*}
  \mathrm{ord}(\bfeta)\le N^{\nu_1+\cdots+\nu_l}. 
\end{equation*}

The case $l=0$ is straightforward. Here $V$ is the unit matrix, $\bxi=\bzeta,$ and
$\lambda_1(\widetilde\Lambda_\bzeta(\nu_1))
> \delta(\bzeta)^{\epsilon^{d-1}}$ as we are in case
$d-l = d\ge 2$ of Lemma~\ref{lem:monomialcoc}(ii).
So $\widetilde\lambda(\bzeta;\nu_1) \ge \delta(\bzeta)^{\min\{\epsilon^{d-1},\nu_1^d/2\}}$
using (\ref{def:tildelambdanu}). 
As $(P,\bzeta)$ is $c$-admissible we have $\cBprim(P) \le c$. 
We will apply
 Proposition
\ref{prop:etaxifactorization}(ii) to $P$ and $\nu=\nu_1$, so we must verify
$\widetilde\lambda(\bzeta;\nu_1) >d^{1/2} \max \{c,\deg P\}.$ This
inequality is satisfied if $\delta(\bzeta)$ is as in
(\ref{eq:deltazetahypinthm}) with $C$ large in terms of
$\epsilon,\nu_1,d,$ and $k$.

\smallskip
{\bf Step 1: A monomial change of coordinates.}
From now on we assume $l\ge 1$, \textit{i.e.}, $l\in\{1,\ldots,d-1\}$.
We have $\zeta^V = (\bfeta,\bxi) \in \IG_m^l \times\IG_m^{d-l}$.
This time we apply Proposition \ref{prop:etaxifactorization} to
$P_{V,\bfeta} \in K(\bfeta)[X_1,\ldots,X_{d-l}], \bxi,$ and
$\nu_{l+1}$. 
\refcomment{101}{Note that $P_{V,\bfeta}\not=0$ as $(P,\bzeta)$ is
  $c$-admissible; this polynomial has at most $k$ non-zero terms.}
By Lemma \ref{lem:Qprops2} the pair $\left(P_{V,\bfeta},\bxi\right)$ is $(cd|V^{-1}|)$-admissible.
Observe  $|V^{-1}|\ll_d |V|^{d-1} \ll_d \delta(\bzeta)^{2\epsilon^{d-l}(d-1)}$. So the said pair is
 $c_1 c \delta(\bzeta)^{2\epsilon^{d-l} d}$-admissible; here
and below $c_1,c_2,\ldots$ denote positive constants that depend only
on $d$. In particular, $\cBprim(P_{V,\bfeta})\le c_1 c
\delta(\bzeta)^{2 \epsilon^{d-l} d}$ and $\cBprim(P_{V,\bfeta})<\infty$. 

\refcomment{107}{If $d-l=1$ we will apply
Proposition \ref{prop:etaxifactorization}(i) and there is nothing
further to check. But for $d-l\ge 2$ we must verify the hypothesis in
the second part of this proposition.} This step is similar
as in the case $l=0$. Indeed, by
Lemma \ref{lem:monomialcoc} we have
\begin{equation}
\label{eq:lambda1lb}
  \widetilde\lambda(\bxi;\nu_{l+1}) \ge
 \delta(\bzeta)^{\min\{\epsilon^{d-l-1},\nu_{l+1}^d/2\}}=\delta(\bzeta)^{\epsilon^{d-l-1}}
\end{equation}
as $\epsilon\le \nu_{l+1}^d/2$. 
Observe that
\begin{alignat*}1
  \deg P_{V,\bfeta} &\ll_d |V|^{d-1} \deg P \ll_d  \delta(\bzeta)^{2\epsilon^{d-l}d} \deg P,
\end{alignat*}
by Lemma \ref{lem:basicPVeta}(i). To apply Proposition
\ref{prop:etaxifactorization}(ii) we must verify 
$$\widetilde\lambda(\bxi;\nu_{l+1})> c_2
\delta(\bzeta)^{2\epsilon^{d-l} d}
\max\{c,\deg P\}.$$
We may  assume  $\epsilon^{d-l-1} - 2\epsilon^{d-l}d \ge
\epsilon^{d-l-1}/2$. 
By (\ref{eq:deltazetahypinthm}) and (\ref{eq:lambda1lb})
the desired inequality is satisfied 
when $C$ is large in terms of $\epsilon,d,$ and $k$.
We may thus apply Proposition~\ref{prop:etaxifactorization}.

Observe that
\begin{alignat*}1
  h(P_{V,\bfeta}) &\ll_{k} 1+h(P) \\
[K(\bfeta):\IQ]&\le \mathrm{ord}(\bfeta)[K:\IQ] \le N^{\nu_1+\cdots+\nu_l} [K:\IQ]\\
\mathrm{ord}(\bxi) &\ge N / \mathrm{ord}(\bfeta) \ge N^{1-(\nu_1+\cdots+\nu_l)}\ge N^{1/2}
\end{alignat*}
by Lemma \ref{lem:basicPVeta} (ii) and Lemma
\ref{lem:monomialcoc}.
\refcomment{108}{Recall that $\bzeta^V =(\bfeta,\bxi)$. So 
  $\bxi^u=1$ for some $u\in\IZ^{d-l}$ with $|u|=\delta(\bxi)$, hence
  $\bzeta^{V
    \left(\begin{smallmatrix}0\\u\end{smallmatrix}\right)
}=1$. We conclude $\delta(\bzeta)\le
|V\left(\begin{smallmatrix}0\\u\end{smallmatrix}\right)|\le
d|V|\delta(\bxi)$.} We find
\begin{equation}
\label{eq:deltaxibound}
  \delta(\bxi)\gg_d \delta(\bzeta)^{1-2 \epsilon^{d-l}}\gg_d \delta(\bzeta)^{1/2}
\end{equation}
as we may assume $\epsilon^{d-l} \le 1/4$. 

We must choose a group $G$ in Proposition
\ref{prop:etaxifactorization}; we will denote it by $H$ here. Let $L$ denote the fixed field of $G$
in $\IQ(\bzeta)$. We may naturally identify $\mathrm{Gal}(L(\bxi)/L)$
with a subgroup of $\GammaM=\mathrm{Gal}(\IQ(\bxi)/\IQ)$ with
$M=\mathrm{ord}(\bxi)$. We take $H$ the subgroup of $\GammaM$ thus
identified with $\mathrm{Gal}(L(\bxi)/L(\bxi)\cap L(\bfeta))\cong \mathrm{Gal}(L(\bzeta)/L(\bfeta))$. 
The fixed field of $H$ in $\IQ(\bxi)$ is contained in $L(\bfeta)$, so
\begin{alignat*}1
  [\GammaM:H]  &\le [L(\bfeta):\IQ] \le \mathrm{ord}(\bfeta) [L:\IQ] 
  = \mathrm{ord}(\bfeta) [\GammaN:G]  \\
  &\le
N^{\nu_1+\cdots+\nu_l}[\GammaN:G]
\end{alignat*}
having used the bound for the order of $\bfeta$ from Lemma \ref{lem:monomialcoc}. 
Moreover, the conductor of $H$ satisfies
\begin{equation*}
  \mf_H \le \lcm{\mf_G,\mathrm{ord}(\bfeta)} \le \mf_G
  \mathrm{ord}(\bfeta)\le \mf_G N^{\nu_1+\cdots+\nu_l}.
\end{equation*}

To cover the case $l=d-1$  we set $\nu_{d} = 1/(128d^2)$. 
By applying Proposition \ref{prop:etaxifactorization} to
$P_{V,\bfeta},\bxi,$ and $\nu_{l+1}$ and using the various
estimates above, in particular  (\ref{eq:lambda1lb}), we find 
\begin{alignat*}1
&\frac{1}{\#H}  \sum_{\sigma\in H} \log |P_{V,\bfeta}(\bxi^\sigma)|  =
m(P_{V,\bfeta}) \\
&+ O_{d,k}\left(\frac{[K:\IQ]^2[\GammaN:G]^2 \mf_G
    \deg(P)^2(1+h(P))N^{(2+2+1)(\nu_1+\cdots+\nu_l)}\delta(\bzeta)^{4\epsilon^{d-l}d}}{N^{\nu_{l+1}^d/(40d)}}\right)
\\
&+O_{d,k}\left(\frac{\deg(P)^{16d^2}}{\delta(\bzeta)^{\epsilon^{d-l-1}/(16k)-32\epsilon^{d-l}d^3}}\right)
\end{alignat*}
here we used  $\vi\le d$ and $M\ge N^{1/2}$; the third line can be omitted
\refcomment{107}{if $l=d-1$ as then we apply Proposition~\ref{prop:etaxifactorization}(i).}

At this point we reap the benefit of having split the
error term in Proposition \ref{prop:etaxifactorization} into a part
depending on $N$ and a part depending on $\delta(\bzeta)$. Indeed, the
order of $\bfeta$, which we bound in terms of $N$, does not affect the
term involving $\delta(\bzeta)$. Recall that $\delta(\bzeta)\le N$,
but there can be no meaningful  lower bound for $\delta(\bzeta)$ in
terms of $N$. Introducing a dependency on $N$ in the part containing
$\delta(\bzeta)$ would spoil the result.

We use the crude bound $\delta(\bzeta)\le N$ and we assume the parameters satisfy
 \begin{equation*}
   5(\nu_1+\cdots+\nu_l) + 4\epsilon^{d-l}d \le \frac{\nu_{l+1}^d}{80d}
 \end{equation*}
and
\begin{equation*}
  32\epsilon^{d-l}d^3 \le \frac{\epsilon^{d-l-1}}{32k}. 
\end{equation*}

We now combine both contributions to the error term and get
\begin{alignat}1
\label{eq:thmfinalsum1}
&\frac{1}{\#H}  \sum_{\sigma\in H} \log |P_{V,\bfeta}(\bxi^\sigma)|  =
m(P_{V,\bfeta}) + O_{d,k}\left( \frac{[K:\IQ]^2 [\GammaN:G]^2 \mf_G
    \deg(P)^{16d^2} (1+h(P))}{\delta(\bzeta)^\kappa}\right)
\end{alignat}
if
\begin{equation*}
  \kappa\le \min\left\{\frac{\nu_{l+1}^d}{80},
    \frac{\epsilon^{d-l-1}}{32k} \right\}.
\end{equation*}
Later we may shrink $\kappa$.

\smallskip
{\bf Step 2: Induction on $d$.} \refcomment{109}{Recall that $\bzeta^V = (\bfeta,\bxi)$}
and $|P(\bzeta^\sigma)| = |P_{V,\bfeta^\sigma}(\bxi^\sigma)|$ for all
$\sigma\in G$. We still assume $l\ge 1$ and we find, as in  
 (\ref{eq:galoisdecompositionsum2}), that
\begin{equation}
\label{eq:thmfinalsum2}
  \frac{1}{\#G} \sum_{\sigma \in G} \log |P(\bzeta^\sigma)| 
=\frac{1}{[L(\bfeta):L]} \sum_{\tau \in\gal{L(\bfeta)/L}} \frac{1}{\#H}\sum_{\sigma\in 
H} \log|P_{V,\bfeta^\tau}(\bxi^{\widetilde\tau \sigma})|
\end{equation}
with $\widetilde\tau$ a lift of $\tau$ to $\mathrm{Gal}(L(\bzeta)/L)$.

The estimates from Step 1  hold equally for all conjugates
$\bfeta^\sigma,\bxi^\sigma$. Indeed, for example
$\widetilde\Lambda_{\bxi}(\nu_{l+1})$ and  $\delta(\bxi)$  are Galois
invariant and $(P,\bzeta^\sigma)$ is $c$-admissible. So we may apply (\ref{eq:thmfinalsum1}) to 
$P_{V,\bfeta^\sigma}$ when summing over $\sigma\in H$.

We set $Q$ to equal $P(X^{V^{-1}})$ times a monomial such that $Q$ is
a polynomial coprime to $X_1\cdots X_d$. So $\deg Q \ll_d |V^{-1}|\deg
P$ and recall that $|V^{-1}|\ll_d
\delta(\bzeta)^{2\epsilon^{d-l}d}$.
 We apply the construction
(\ref{eq:defQpi}) to $Q$ and $l$  and obtain $\widehat Q$. 
Recall Lemma \ref{lem:Qprops0} and write $\widetilde Q$ for $\widehat
Q$ times the monomial from part (ii) of this lemma. Then $\widetilde
Q$  has at most $k^2$ non-zero terms and using also 
Lemma \ref{lem:Qprops}
we find
\begin{equation}
\label{eq:tildeQbounds}
\begin{aligned}
   \widetilde Q &\in K'[X_1^{\pm 1},\ldots,X_l^{\pm 1}] \text{ where }
   [K':\IQ]\le [K:\IQ]^2, \\ 
\deg \widetilde Q &\ll_d \deg Q\ll_d |V^{-1}| \deg P \ll
\delta(\bzeta)^{2\epsilon^{d-l}d} \deg P,\text{ and}\\
h(\widetilde Q)  &\ll_k 1 + h(Q) \ll_k 1 + h(P).
 \end{aligned}
\end{equation}

By Lemma \ref{lem:Qprops2}, with $l=0$, the pair $(Q,(\bfeta^\sigma,\bxi^\sigma))$ is
$c_3 c \delta(\bzeta)^{2\epsilon^{d-l}d}$-admissible for all
$\sigma\in G$.
Now $(\widetilde Q,\bfeta^\sigma)$ is also $c_3 c
\delta(\bzeta)^{2\epsilon^{d-l}d}$-admissible by Lemma
\ref{lem:Qprops1} for all $\sigma$.

We want to apply this theorem to $\widetilde Q$ and $\bfeta\in\IG_m^l$ by
induction, recall $l\le d-1$. For this
we must verify
\begin{equation*}
  \delta(\bfeta) \ge c_4 C(l,k^2) \delta(\bzeta)^{2\epsilon^{d-l}dC(l,k^2)}
  \max\{c ,\deg P\}^{C(l,k^2)}
\end{equation*}
having used the bound for $\deg\widetilde Q$ in (\ref{eq:tildeQbounds}).
\refcomment{108}{By symmetry we have, as above and in
  (\ref{eq:deltaxibound}), the bound $\delta(\bfeta)\gg_d \delta(\bzeta)/|V|\gg_d
\delta(\bzeta)^{1/2}$.}  So it suffices
to check
\begin{equation}
\label{eq:wanttoshowforind}
  \delta(\bzeta)^{1-4\epsilon^{d-l}d C(l,k^2)} \ge c_5 C(l,k^2)^2
  \max\{c,\deg P\}^{2C(l,k^2)}. 
\end{equation}
We may assume that $1-4\epsilon^{d-l}d C(l,k^2)\ge 1/2$ as we may
choose $\epsilon$ small in terms of $d$ and $C(l,k^2)$. So (\ref{eq:wanttoshowforind}) follows from
(\ref{eq:deltazetahypinthm}) if $C=C(d,k)$ is large enough in terms of
$d$ and $k$.

By induction and (\ref{eq:tildeQbounds}) we have
\begin{alignat*}1
&  \frac{1}{\#H'} \sum_{\tau\in H'} \log|\widetilde Q(\bfeta^\tau)|  =
  m(\widetilde Q)
  \\ &\quad + O_{d,k}\left(\frac{[K:\IQ]^{2^d} [\GammaE:H']^2
      \mf_{H'}\deg(P)^{16d^2}\delta(\bzeta)^{32\epsilon^{d-l}d^3} (1+h(P))}{\delta(\bfeta)^{\kappa(l,k^2)}}\right)
\end{alignat*}
here $E = \mathrm{ord}(\bfeta)$ and $H'\subset \GammaE$ is the
subgroup identified with $\mathrm{Gal}(\IQ(\bfeta)/\IQ(\bfeta)\cap
L)\cong \mathrm{Gal}(L(\bfeta)/L)$. 
Note that $[\GammaE:H'] = [\IQ(\bfeta)\cap L:\IQ]\le [L:\IQ] =
[\GammaN:G]$ and 
$\mf_{H'} \le \mf_G$. Using again $\delta(\bfeta)\gg_d
\delta(\bzeta)^{1/2}$ we get 
\begin{equation}
\label{eq:endgameaverHprime}
  \frac{1}{\#H'} \sum_{\tau\in H'} \log|\widehat Q(\bfeta^\tau)|  =
  m(\widehat Q) + O_{d,k}\left(\frac{[K:\IQ]^{2^d} [\GammaN:G]^2
      \mf_{G}\deg(P)^{16d^2} (1+h(P))}{\delta(\bzeta)^{\kappa(l,k^2)/2
      - 32 \epsilon^{d-l}d^3}}\right)
\end{equation}
as passing from $\widetilde Q$ to $\widehat Q$ is harmless. We may
assume that $\kappa(l,k^2)/4 \ge 32 \epsilon^{d-l}d^3$. 

Recall that $Q$ equals $P(X^{V^{-1}})$ up-to a
monomial factor.
We will soon apply
 Proposition \ref{prop:lehmermean} to $Q$.
 Consider  $(x_1,\ldots,x_{\#H'})$, with each $x_i\in [0,1)^l$,
\refcomment{112}{a tuple of discrepancy $\cD$ as in (\ref{def:discrepancy}),}
  where the
$\be(x_i)$ are the $\bfeta^\tau$.
 Proposition \ref{prop:lehmermean} together with (\ref{eq:endgameaverHprime}) imply
\begin{alignat*}1
 \frac{1}{\#H'} &\sum_{\tau\in H'} m(P_{V,\bfeta^\tau}) = m(Q) + \\
  &O_{d,k}\left(\deg(Q)\cD^{1/(32(d+1)k^2)}+\frac{[K:\IQ]^{2^d} [\GammaN:G]^2
      \mf_{G}\deg(P)^{16d^2} (1+h(P))}{\delta(\bzeta)^{\kappa(l,k^2)/4}}\right).
      \end{alignat*}
 By Proposition \ref{propo:discbound}, parts (i) and (ii), we find 
\begin{equation*}
 \cD \ll_d [\GammaE:H'] \mf_{H'}^{1/2} \delta(\bfeta)^{-1/3} \ll_d
[\GammaN:G] \mf_{G}^{1/2} \delta(\bzeta)^{-1/6}.
\end{equation*}
From above we find $\deg Q \ll_d |V^{-1}|\deg P \ll_d
\delta(\bzeta)^{2\epsilon^{d-l}d}\deg P$.
The Mahler measure is invariant under a monomial change of
coordinates by Corollary 8, Chapter 3.4~\cite{Schinzel}, thus $m(P)=m(Q)$.
As $\#H' = [L(\bfeta):L]$ we get 
\begin{equation*}
  \frac{1}{[L(\bfeta):L]} \sum_{\tau\in H'} m(P_{V,\bfeta^\tau}) = m(P) + 
O_{d,k}\left(\frac{[K:\IQ]^{2^d} [\GammaN:G]^2
      \mf_{G}\deg(P)^{16d^2}
      (1+h(P))}{\delta(\bzeta)^{\min\{1/(192(d+1)k^2)-2\epsilon^{d-l}d\},\kappa(l,k^2)/4\}}}\right).
\end{equation*}

We shrink $\epsilon$ a final time to achieve $1/(192(d+1)k^2)
-2\epsilon^{d-l}d > 1/(200(d+1)k^2)$. 
The theorem follows on combining this asymptotic estimate with
(\ref{eq:thmfinalsum1}) and (\ref{eq:thmfinalsum2}), when $\kappa =
\kappa(d,k)$ is small in terms of $\kappa(l,k^2),d,$ and $k$. 
\end{proof}

To prove Theorem \ref{thm:main} we can multiply $P$ by any monomial,
so we may assume that it is a polynomial. Thus the theorem
is a direct consequence of the following more
precise corollary. 

\begin{corol}
\label{cor:limiterrorterm}
  Let $K\subset \IC$ be a number field and suppose
  $P\in K[X_1,\ldots,X_d]\ssm \{0\}$ is \essatoral{} and has at most
  $k$ non-zero terms for an integer $k\ge 2$. There exists $\kappa =
  \kappa(d,k)>0$ with the following property.
  Suppose $\bzeta\in\IG_m^d$ has  order $N$ and suppose $G$
is a subgroup of $\GammaN$ and $\delta(\bzeta)$ is large in
terms of $d, P,[K:\IQ],\mf_G,$ and $[\GammaN:G]$. Then $P(\bzeta^\sigma)\not=0$ for all $\sigma\in
G$ and
\begin{equation*} 
\frac{1}{\#G} \sum_{\sigma \in G} \log |P(\bzeta^\sigma)|
   = m(P) + O_{d,k}\left(\frac{[K:\IQ]^{2^d} [\GammaN:G]^2 \mf_G
     \deg(P)^{16d^2} (1+\hproj{P})}{\delta(\bzeta)^\kappa}\right).   
\end{equation*}
\end{corol}
\begin{proof}
By Lemma \ref{lem:Petaprops} there is $c\ge 1$, depending only on $P$,
such that $(P,\bzeta)$ is $c$-admissible for all $\bzeta\in\IG_m^d$ of
finite order with $\delta(\bzeta)\ge c$. 

Suppose $\bzeta\in\IG_m^d$ has finite order and $P(\bzeta)=0$.
By the Manin--Mumford Conjecture, 
$\delta(\bzeta)$ is bounded in terms of $d$ and $P$ only. 
Hence for $\delta(\bzeta)$ sufficiently large in terms of these
quantities  we have
$P(\bzeta)\not=0$ and the same also holds with $\bzeta$ replaced by a
Galois conjugate. Our corollary now follows from Theorem \ref{thm:main2}.
\end{proof}


\begin{proof}[Proof of Corollary \ref{cor:Ih}]
We may assume that $K/\IQ$ is Galois and that $P$ is a polynomial.
The product $P'$ for $\tau(P)$ as $\tau$ ranges over $\gal{K/\IQ}$ has
rational coefficients. The coefficients are even integers as the coefficients of
$P$ lie in $\IZ_K$. 

The Mahler measure of any non-zero, integral polynomial is
non-negative. By a theorem attributed to Boyd~\cite{Boyd:Kronecker},
Lawton~\cite{Lawton77}, Smyth~\cite{SmythMahler}, the fact that 
 the zero set of $P$ in $\IG_m^d$ has an irreducible component not
 equal to  the translate of an algebraic
subgroup by a point of finite order implies
$m(P')>0$.

Suppose $\bzeta\in\IG_m^d$ has order $N$.
Take for $G$ the subgroup of $\GammaN$ associated to
$\gal{\IQ(\bzeta)/K\cap \IQ(\bzeta)}$. Then $[\GammaN:G] \le
[K:\IQ]$. As $\bzeta$ varies, there are only finitely many possibilities for
the number field $K\cap \IQ(\bzeta)$, being a subfield of the field
$K$. So $\mf_G$ is bounded from above solely in terms of $K$. 
For any $\tau\in \gal{K/\IQ}$ 
choose an extension $\widetilde\tau \in \gal{K(\bzeta)/\IQ}$.
We apply Corollary \ref{cor:limiterrorterm} to the
polynomial $\tau(P)$ which is  \essatoral{} by hypothesis.
If $\delta(\bzeta^{\widetilde \tau})=\delta(\bzeta)$ is large
enough in terms of the fixed data, then 
  \begin{equation*}
    \frac{1}{\#G}\sum_{\sigma\in G} \log
    |\tau(P)(\bzeta^{\widetilde\tau\sigma})| = m(\tau(P)) + o(1)
  \end{equation*}
as $\delta(\bzeta)\rightarrow\infty$, here and below the implied
constant
is independent of $\bzeta$.

The average of the left-hand side over $\tau\in \gal{K/\IQ}$ equals
the left-hand side in 
\begin{equation*}
 \frac{1}{[K(\bzeta):\IQ]} \sum_{\sigma:K(\zeta)\rightarrow \IC}
 \log|\sigma(P(\bzeta))|
 = \frac{1}{[K:\IQ]} \sum_{\tau\in \gal{K/\IQ}} m(\tau(P)) + o(1).
\end{equation*}

As the Mahler measure is additive, the average on the right-hand side is
$m(P')/[K:\IQ]>0$. 
But the left-hand side vanishes if 
 $P(\bzeta)$ is an algebraic unit. 
 In this case, we see that $\delta(\bzeta)$ is bounded from above.
\end{proof}


\appendix
\section{A theorem of Lawton re-revisited}\label{app:lawton}

The following theorem makes explicit a result of Lawton~\cite{Lawton}.
It is a more precise version of the second-named author's
result~\cite{hab:gaussian} which is unfortunately insufficient for our
purposes. We closely follow the proof presented in~\cite{hab:gaussian}
which itself is based on Lawton's approach~\cite{Lawton}. We also
show how to correct an inaccuracy in the proof of Lemma A.4(i)~\cite{hab:gaussian}.

Recall the definition of $\rho(\cdot)$ in  (\ref{def:rhoa}) where
$d\ge 1$ is an integer.

\begin{thm}
  \label{thm:lawtonquant}
  Suppose $P\in K[X_1,\ldots,X_d]\ssm \{0\}$ has at most $k$ non-zero terms for an integer
  $k\ge 2$.
  For $a=(a_1,\ldots,a_d)\in\IZ^d\ssm\{0\}$ with $\rho(a)>\deg P$  we have
  \begin{equation}
    \label{eq:lawtonlimit}
    \ma{P(X^{a_1},\ldots,X^{a_d})} = \ma{P}+
    O_{d,k}\left( \frac{\deg(P)^{16d^2}}
      {\rho(a)^{1/(16(k-1))}}\right)
  \end{equation}
  where the implicit constant depends only on $d$ and $k$.
\end{thm}

In the univariate case $d=1$ we have $\rho(a)=\infty$ for all
$a\in\IZ\ssm\{0\}$ by definition.
Then we should interpret (\ref{eq:lawtonlimit}) as stating
$m(P(X^a)) = m(P)$. This identity is an  easy consequence of
(\ref{eq:mahlerjensen}).
So throughout this subsection we assume  $d\ge 2$.

We did not strive to obtain the
best-possible
exponent in $\rho(a)^{1/(16(k-1))}$ that our method can produce.

We must assume $\rho(a)>\deg P$ to avoid \refcomment{115}{interaction}
of coefficients in $P(X^{a_1},\ldots,X^{a_d})$. Indeed, take for
example $P = X_1 (X_2-1+\epsilon)$ with $\epsilon\in (0,1)$ small and
$a = (1,0)$. Then $P(X,1) = X \epsilon$ whose Mahler measure is $\log
\epsilon$. On the other hand $m(P) = m(X_2-1+\epsilon) =
\log\max\{1,|1-\epsilon|\} = 0$ by Jensen's formula. The difference
\begin{equation*}
  m(P(X,1)) - m(P) = \log\epsilon 
\end{equation*}
is unbounded as $\epsilon \rightarrow 0$. This does not contradict our
theorem as $\rho(a)=1$.

The Lebesgue measure on $\IR^d$ is $\vol{\cdot}$. For $P\in
\IC[X_1^{\pm 1},\ldots,X_d^{\pm 1}]$ and $\sv > 0$ we define
\begin{equation}
  \label{def:SPr}
  S(P,\sv)  =
  \{ x \in [0,1)^d : |P(\be(x))|< \sv \}
\end{equation}
where $\be$ is as in (\ref{eq:defbex}). 

Dobrowolski extended Lawton's Theorem 1~\cite{Lawton} to polynomials
that are not necessarily monic.

\begin{thm}[Dobrowolski, Theorem 1.1~\cite{Dobrowolski:Lawton}]
  \label{thm:dobrowolski}
  Suppose $P\in \IC[X]\ssm\{0\}$ has at most $k$ non-zero terms for an
  integer $k\ge 2$. Then $ \vol{ S(P,\sv) }\ll_k
   \min\{1,\sv/|P|\}^{1/(k-1)}$ for all $\sv > 0$.
\end{thm}

Dobrowolski requires that $P$ as at least $2$ non-zero terms. But it
is convenient to allow $P$ to have a single term, as above. It is also
convenient to apply the estimate in the case $P=0$, we then interpret
the minimum to be $1$.

Until the end of this appendix and if not stated otherwise we assume
that $P\in \IC[X_1,\ldots,X_d]\ssm\IC$ has at most $k$ non-zero terms
for an integer $k\ge 2$ and $|P|=1$.

\begin{lemma}\
  \label{lem:volSPep}
  \begin{enumerate}
  \item [(i)]
    If $\sv>0$ then 
    $\vol{S(P,\sv)}
    \ll_{d,k} \sv^{1/(2(k-1))}$.
  \item[(ii)] We have 
    $\int_{[0,1)^d}\bigl|\log |P(\be(x))|\bigr|^2
    dx\ll_{d,k}1$.
  \end{enumerate}
\end{lemma}
\begin{proof}
  To ease notation we drop $d,k$ in the subscript $\ll_{d,k}$.
  
  Because of the trivial bound
  $\vol{S(P,\sv)}\le 1$ we may assume $\sv\le 1$. 

  The case $d=1$ follows from Theorem~\ref{thm:dobrowolski}.
  So let us now assume $d\ge 2$. We consider $P$ as a polynomial in the unknown
  $X_d$ and coefficients among $\IC[X_1,\ldots,X_{d-1}]$.
  We pick a coefficient $P_i$ with maximal norm, \textit{i.e.},
  $P$ has a term $P_i X_d^i$ such that $P_i\in\IC[X_1,\ldots,X_{d-1}]$
  and $|P_i| = |P| = 1$. 

  For $x'\in\IR^{d-1}$ we let  $P_{\be(x')}$ denote $P(\be(x'),X)\in\IC[X]$. 
  Recall that
  \begin{equation*}
    S(P,\sv) 
    =\{(x',t) \in [0,1)^{d-1}\times [0,1) :
    |P_{\be(x')}(\be(t))| < \sv\}.
  \end{equation*}
  We splice  the hypercube and apply Fubini's Theorem
  to find
  \begin{alignat*}1
    \vol{S(P,\sv)} 
    &= \int_{[0,1)^{d-1}} \vol{S(P_{\be(x')},\sv)}dx'.
  \end{alignat*}
  %
  %
  The measure zero set of $x'\in [0,1)^{d-1}$ with  $P_{\be(x')}$ is
  harmless. 
  By Theorem~\ref{thm:dobrowolski} 
  we find
  \begin{equation*}
    \vol{S(P,\sv)}\ll \int_{[0,1)^{d-1}}
    \min\left\{1,\frac{\sv}{|P_{\be(x')}|}\right\}^{1/(k-1)} dx'. 
  \end{equation*}
  The coefficient of
  $X^i$ in $P_{\be(x')}$ is $P_i(\be(x'))$. 
  So $|P_{\be(x')}|\ge |P_i(\be(x'))|$ and 
  \begin{alignat}1
    \label{eq:I1I2}
    \vol{S(P,\sv)}&\ll \int_{[0,1)^{d-1}}
    \min\left\{1,\frac{\sv}{|P_i(\be(x'))|}\right\}^{1/(k-1)} dx' =
    I_1 + \sv^{1/(k-1)} I_2
  \end{alignat}
  where
  \begin{equation*}
    I_1 = \int_{|P_i(\be(x'))|< \sv} dx' 
    \quad\text{and}\quad
    I_2 = \int_{|P_i(\be(x'))|\ge  \sv}\frac{dx'}{|P_i(\be(x'))|^{1/(k-1)}};
  \end{equation*}
  both integrals are over subsets of $[0,1)^{d-1}$. 
  We will bound $I_1$ and $I_2$ from above.

  We have
  $  I_1 = \vol{S(P_i,\sv)}$.
  This lemma applied by induction to $P_i$, a polynomial in $d-1$
  variables with at most $k$ non-zero terms and $|P_i|=1$, yields
  \begin{equation}
    \label{eq:I1bound}
    I_1 \ll \sv^{1/(2(k-1))}.
  \end{equation}

  To bound $I_2$ we consider real numbers $\sv = r_0 < r_1 < \cdots < r_{N+1}=k+1$, with
  $r_{n+1}\le r_n + \delta$ where $\delta>0$ is a small parameter.
  We split the
  domain of integration
  up into measurable parts
  \begin{equation*}
    \Sigma_n = \left\{x'\in [0,1)^{d-1} :
      r_n \le |P_i(\be(x'))|< r_{n+1}\right\}
    \quad\text{for}\quad n\in \{0,\ldots,N\}. 
  \end{equation*}
  Observe that $|P_i(\be(x'))|\le k<r_{N+1}$ for all $x'$. 
  Thus 
  \begin{equation}
    \label{eq:firstI2bound}
    I_2 =  \sum_{n=0}^N \int_{\Sigma_n}
    \frac{dx'}{|P_i(\be(x'))|^{1/(k-1)}}
    \le  \sum_{n=0}^N \frac{\vol{\Sigma_n}}{r_n^{1/(k-1)}}
    = \sum_{n=0}^N a_nb_n
  \end{equation}
  where  $a_n =
  r_n^{-1/(k-1)}$ and $b_n = \vol{\Sigma_n}$.

  As the $\Sigma_n$ are pairwise disjoint, the partial sums satisfy 
  $$B_n =\sum_{l=0}^n b_l= \vol{\bigcup_{l=0}^n \Sigma_l}
  \le \vol{\{x'\in [0,1)^{d-1} :
    |P_i(\be(x'))|< r_{n+1}\}} = \vol{S(P_i,r_{n+1})}.$$  
  Hence we have the trivial bound $B_n\le 1$.
  As in the bound for $I_1$
  we apply this lemma by induction to $P_i$ and find
  \begin{equation}
    \label{eq:partsumBn}
    B_n \le \vol{S(P_i,r_{n+1})} \ll r_{n+1}^{1/(2(k-1))}.
  \end{equation}

  Summation by parts implies
  \begin{equation*}
    I_2\le \sum_{n=0}^N a_nb_n = a_N B_N - \sum_{n=0}^{N-1} B_n(a_{n+1}-a_n)
    \le (k+1)^{1/(k-1)} + \sum_{n=0}^{N-1} B_n(a_n-a_{n+1});
  \end{equation*}
  we used the trivial bounds $a_N = r_N^{1/(k-1)} \le (k+1)^{1/(k-1)}$
  and  $B_N\le 1$. 
  By (\ref{eq:partsumBn}) and the definition of $a_n$ we find
  \begin{equation*}
    I_2 \ll 1
    +\sum_{n=0}^{N-1} r_{n+1}^{1/(2(k-1))}(r_n^{-1/(k-1)}-r_{n+1}^{-1/(k-1)}). 
  \end{equation*}

  We use the  mean value theorem to bound 
  $$r_n^{-1/(k-1)}-r_{n+1}^{-1/(k-1)}
  \ll r_n^{-1/(k-1)-1} (r_{n+1}-r_n) \ll r_{n+1}^{-1/(k-1)-1}
  (r_{n+1}-r_n);$$ for the second  bound we assume, as we may, that
  $\delta \le \sv$ and so $r_{n+1}\le r_n+\delta \le 2r_n$. 
  Thus $I_2 \ll 1+\int_{\sv }^{k+1} t^{-1/(2(k-1))-1}dt
  \ll \sv^{-1/(2(k-1))}$. 
  
  



  This bound together with (\ref{eq:I1bound}) implies
  $I_1+\sv^{1/(k-1)}I_2 \ll \sv^{1/(2(k-1))}$.
  Therefore, $\vol{S(P,\sv)}\ll \sv^{1/(2(k-1))}$ by
  (\ref{eq:I1I2}), completing the induction step and the proof of (i).
  




\refcomment{119}{We define
$p_n(x) = \min \{n, |\!\log |P(\be(x))||^2\}\ge 0$ where $n\ge 0$ is
an integer.}
We must find an upper bound for the non-decreasing sequence
$I_n = \int_{[0,1)^d} p_n(x)dx$.
Observe that $|P(\be(x))|\le k|P|=k$, so if
$n\ge (\log k)^2$, 
then
 $|P(\be(x))|\le e^{\sqrt n}$.
We fix  $m$ to be the least integer with
$m\ge 1+ (\log k)^2$, so $m\ge 2$.
Say $n\ge m$. 
Then
 $p_n$ equals $n$ on $S(P,e^{-\sqrt n})$ and it equals
$p_{n+1}$ outside this set. Thus
  \begin{equation*}
 I_{n+1}-I_n = \int_{S(P,e^{-\sqrt n})}
      (p_{n+1}(x)-p_n(x))dx 
 \le \vol{S(P,e^{-\sqrt n})}
 \ll  e^{-\lambda \sqrt n}
  \end{equation*}
from part (i),  here $\lambda = 1/(2(k-1))$. 
A telescoping sum trick shows
\begin{equation*}
I_n-I_m \ll
\sum_{l\ge m} e^{-\lambda\sqrt{l}}
\ll \int_{m-1}^\infty e^{-\lambda\sqrt{l}} dl \ll 1.
\end{equation*}
The initial term satisfies $I_m\le m  \ll 1$ as $m$ depends only on
$k$,  this completes the proof. 
\end{proof}

A more careful analysis should lead to $\mathrm{vol}(S(P,\sv))\ll_{d,k}
(1+|\!\log \sv|)^{d-1}\sv^{1/(k-1)}$ for all $\sv > 0$ in part (i) of
Lemma~\ref{lem:volSPep}. But this improvement has little effect on
the main results of the current work.

Brunault, Guilloux, Mehrabdollahi, and Pengo pointed out that the
second-named author's argument for Lemma A.4(i) \cite{hab:gaussian}
leads (for $k\ge 2$)
to an estimate $O(y^{f(n)/(2(k-1))})$ where $f(n)$ depends on 
the number of variables $n$, as opposed to the claimed
bound $O(y^{1/(2(k-1))})$. However, the claimed bound holds true by
Lemma~\ref{lem:volSPep}(i). Alternatively and in the proof of
Lemma~\ref{lem:volSPep}(i) one can replace Dobrowolski's Theorem 1.1 \cite{Dobrowolski:Lawton} by
Lawton's Theorem 1~\cite{Lawton} which is sufficient for the
applications in \cite{hab:gaussian}.

\begin{lemma}
  \label{lem:Spdeltaintegral}
If $\sv>0$ then
 \begin{equation*}
  \int_{S(P,\sv)} \bigl|\log|P(\be(x))|\bigr| dx  \ll_{d,k} \sv^{1/(4(k-1))}.
\end{equation*}
\end{lemma}
\begin{proof}
As $|P(\be(x))|\le |P|k\le k$ for all $x\in[0,1)^d$
  we may assume $\sv\le 1$. 
  
With $\Sigma = S(P,r)$   we find
\begin{alignat*}1
0\le -\int_{\Sigma} \log|P(\be(x))| dx
 &= -\sum_{n=0}^\infty \int_{\frac{\sv}{2^{n+1}}\le
 |P(\be(x))| <
 \frac{\sv}{2^{n}}} \log |P(\be(x))| dx
\\
 &\le \sum_{n=0}^\infty \log\left(\frac{2^{n+1}}{\sv}\right) \vol{S(P,\sv/2^n)}. 
 \end{alignat*}
Let
$\lambda =
1/(2(k-1))\le 1/2$.
We use Lemma \ref{lem:volSPep}(i) to bound 
$\vol{S(P,\sv/2^n)}\ll_{d,k} (\sv/2^n)^{\lambda}$.
Note that 
$\log(2t)\ll_k t^{\lambda/2}$
 on $t\in [1,\infty)$.
 We take $t = 2^n/\sv \ge 1$ and conclude
\begin{equation*}
  -\int_{\Sigma} \log|P(\be(x))| dx
  \ll_{d,k} \sum_{n=0}^\infty
  \left(\frac{\sv}{2^n}\right)^{\lambda/2}
  \ll_{d,k} \sv^{\lambda/2}.\qedhere
\end{equation*}
\end{proof}

Boyd~\cite{Boyd} proved  that the Mahler measure is continuous on the non-zero
polynomials of fixed degree.
Here we show that the Mahler measure is H\"older continuous away from
$0$. For the next lemma we momentarily drop our usual assumptions on $P$. 

\begin{lemma}
  \label{lem:mahlerhoelder} 
Suppose $P,Q\in \IC[X_1,\ldots,X_d]
\ssm\{0\}$ such that $P$ and $Q$ both have at most $k$ non-zero terms for an integer
$k\ge 2$. 
If $\delta=|P-Q|/|Q|\le 1/2$, then 
\begin{equation*}
\ma{P}\le \ma{Q} +
C(d,k)\delta^{1/(8(k-1))}
\end{equation*}
where $C(d,k)>0$ is effective and depends only on $d$ and $k$.
\end{lemma}
\begin{proof}
  It suffices to prove the lemma when $|Q|=1$; indeed,
  just replace $P$ and $Q$ by $P/|Q|$ and $Q/|Q|$,
  respectively, to reduce to this case.
  
Suppose for the moment that $x\in\IR^d$ with $P(\be(x))Q(\be(x))\not=0$. Then
$|P(\be(x))-Q(\be(x))|\le 2k |P-Q|$
and so
\begin{equation}
  \label{eq:mahlerlogbound}
\log \left|\frac{P(\be(x))}{Q(\be(x))}\right|\le  \left|\frac{P(\be(x))}{Q(\be(x))}\right| -1  \le 2k\frac{\delta}{|Q(\be(x))|}.
\end{equation}
where the first inequality used  $\log t \le
t-1$ for all $t>0$.
The difference of Mahler measures $\ma{P}-\ma{Q}$ can thus be written as
\begin{equation*}  
 \int_{[0,1)^d \ssm \Sigma} \bigl(\log|{P(\be(x))}|-\log|{Q(\be(x))}|\bigr) dx  
   +\int_{\Sigma} \bigl(\log|{P(\be(x))}|-\log|{Q(\be(x))}|\bigr) dx
\end{equation*}
with $\Sigma = S(Q,\delta^{1/2})$. 

The first integral is at most $2k \delta^{1/2}$ by
(\ref{eq:mahlerlogbound}).
We proceed by bounding the second integral $I$ from above.
First, we note that $|P(\be(x))|\le k|P| \le 3k/2$ as
$|P-Q|\le\delta\le 1/2$ and thus $|P|\le 3/2$. So
\begin{equation*}
I 
\le \log\left(3k/2\right)\vol{\Sigma}-
\int_{\Sigma} \log|Q(\be(x))| dx
\le \log\left(3k/2\right)\vol{\Sigma}
+c\delta^{1/(8(k-1))}
\end{equation*}
where we applied  Lemma \ref{lem:Spdeltaintegral}  to $Q$ and
$\delta^{1/2}$, the case $Q$ constant being trivial;
here  $c=c(d,k)>0$. Finally, 
Lemma \ref{lem:volSPep}(i) yields
$\vol{\Sigma}=\vol{S(Q,\delta^{1/2})} \ll_{d,k}
\delta^{1/(4(k-1))}$ and
the lemma follows as
$\delta\le 1$. 
\end{proof}

Let $\IN_0 =\IN\cup \{0\}$. For 
$b\in\IN_0$  let  $C^b(\IR^d)$ denote 
the set of real valued
functions on $\IR^d$ 
whose derivatives exist and are continuous up-to and including order $b$.
For a multiindex $i=(i_1,\ldots,i_d)\in\IN_0^d$
we set $\ell(i) = i_1+\cdots+i_d$. If  
 $g\in  C^b(\IR^d)$ and $\ell(i)\le b$, we set
$\partial^i g = (\partial/\partial x_1)^{i_1}\cdots
(\partial/\partial x_d)^{i_d}g \in C^0(\IR^d)$ 
and 
\begin{equation*}
  |g|_{C^b} = \max_{\substack{i\in\IN_0^d \\  \ell(i)\le b}}
\sup_{x\in\IR^d} |\partial^i g(x)| \in \IR\cup \{\infty\}. 
\end{equation*}

We recall the construction of $f_\sv$ in~\cite{hab:gaussian}
depending on the parameter $\sv \in (0,1/2]$. This
function lies in $C^b(\IR^d)$ and equals $\log |P(\be(\cdot))|$ away from the
singularity, \textit{i.e.}, the locus where $P(\be(\cdot))$ vanishes.

\refcomment{122}{We fix the anti-derivative $\phi$  of $x^b(1-x)^b$ on $[0,1]$ with
$\phi(0)=0$ and multiply it with a positive number to ensure $\phi(1)=1$. Then we 
 extend it by $0$ on $x<0$ and by $1$ for $x>1$ to obtain
 a non-decreasing step function $\phi\in C^b(\IR)$ with
support $[0,1]$. }
Finally, we rescale and define $\phi_\sv (x) =
\phi(((2/\sv)^2x-1)/3)$.
So $\phi_\sv$ is a  \refcomment{123}{non-decreasing function which vanishes on
  $(-\infty,(\sv/2)^r]$, equals $1$ on $[\sv^2,\infty)$, and satisfies}
\begin{equation*}
  \left|\frac{d^i\phi_\sv}{dx^i}\right|_{C^0} \ll_b \sv^{-2i}
  \quad\text{for all}\quad 0\le i\le b,\quad\text{hence}\quad
  |\phi_\sv|_{C^b}\ll_b \sv^{-2b}.
\end{equation*}
The function $\phi_r$ takes values in $[0,1]$. Moreover, we define
\begin{equation*}
  \psi_\sv(x)  = \left\{
  \begin{array}{ll}
    \frac 12 \phi_\sv(x) \log x &: x > 0, \\
    0 &: x \le 0.
  \end{array}\right.
\end{equation*}
\refcomment{125}{which vanishes on $(-\infty,(\sv/2)^2]$, coincides with $\frac 12 \log
x$ on $[r^2,\infty)$, and satisfies}
\begin{equation}
\label{eq:psiCbbound}
  |\psi_\sv|_{C^b} \ll_{b}\sv^{-2b}|\!\log\sv|.
\end{equation}

We consider  $g:x\mapsto  |P(\be(x))|^2$, then
\begin{equation}
\label{eq:gCbbound}
  |g|_{C^b} \ll_{k,b} (\deg P)^b.
\end{equation}
Next we compose  $f_\sv = \psi_\sv \circ g\in C^b(\IR^d)$, so
for $x\in\IR^d$ we have
\begin{equation*}
f_\sv(x) = \left\{
\begin{array}{ll}
  0 & :\text{ if $|P(\be(x))|\le \sv/2$,}\\
  \log |P(\be(x))| & :\text{ if $|P(\be(x))|\ge\sv$.}
\end{array}\right. 
\end{equation*}
By Lemma A.5~\cite{hab:gaussian},
which follows from the chain rule, together with  
(\ref{eq:psiCbbound}) and
(\ref{eq:gCbbound})  we find
\begin{equation}
\label{eq:fyCbbound}
  |f_\sv|_{C^b}  \ll_{k,b} \sv^{-2b}|\!\log\sv| (\deg
   P)^{b^2}.
\end{equation}

For the following lemmas we  suppose  $b\ge d+1$. As above we have $r\in
(0,1/2]$.

\begin{lemma}
  \label{lem:lawton1}
  Suppose $a\in\IZ^d\ssm\{0\}$, then
  \begin{equation*}
    \int_0^1 f_\sv(as) ds = \int_{[0,1)^d} f_\sv(x)
    dx+ O_{d,k,b}\left(
      \frac{|\!\log\sv|}{\sv^{2b}} \frac{(\deg P)^{b^2}}{\rho(a)^{b-d}}\right).
  \end{equation*}
\end{lemma}
\refcomment{127}{We follow and adapt the proof of Lemma A.6~\cite{hab:gaussian}.}
\begin{proof}
  For $m\in\IZ^d$ let  $\widehat
  f_r(m)$ denote the Fourier coefficient of  $f_r$.
  By Theorem
  3.2.9(a)~\cite{Grafakos}
  with derivative up-to order $b$ and using
  $|\widehat{\partial^i f_\sv} (m)|\le |\partial^i f_\sv|_{C^0} \le
  |f_\sv|_{C^b}$ where $\ell(i)=b$
  we conclude
  $  |\widehat{f_\sv}(m)| \ll_{d,b}{|f_\sv|_{C^b}}{|m|^{-b}}$ if $m\not=0$. So
  $|\widehat{f_\sv}(m)|\ll_{d,b}
  \sv^{-2b}{|\!\log\sv|}(\deg P)^{b^2}|m|^{-b}$
  for all $m\in\IZ^d\ssm\{0\}$ by (\ref{eq:fyCbbound}).
  Then
  \begin{alignat}1
    \label{eq:fouriercoeffsumbound}
    \sum_{|m|\ge \rho(a)} |\widehat{f_\sv}(m)| 
    &\ll_{d,k,b} \frac{|\!\log \sv|}{\sv^{2b}}(\deg P)^{b^2}
    \sum_{|m|\ge \rho(a)} \frac{1}{|m|^{b}} 
    \ll_{k,b} \frac{|\!\log \sv|}{\sv^{2b}} \frac{(\deg P)^{b^2}
    }{\rho(a)^{b-d}}
  \end{alignat}
  as $b\ge d+1$. In particular, the Fourier coefficients of $f_\sv$ are absolutely summable and 
  the    Fourier series converges absolutely and uniformly to $f_\sv$,
  \refcomment{129}{see Proposition 3.1.14~\cite{Grafakos}.}
  Hence
  \begin{equation*}
    \int_0^1 f_\sv(as) ds = \sum_{m\in\IZ^d} \int_0^1
    \widehat{f_\sv}(m)e^{2\pi \sqrt{-1}\langle a,m\rangle s}ds
    = \int_{[0,1)^d} f_\sv(x)dx+
    \sum_{\substack{m\in\IZ^d\ssm\{0\} \\ \langle a,m\rangle=0}} 
    \widehat{f_\sv}(m).
  \end{equation*}
  The lemma follows from (\ref{eq:fouriercoeffsumbound}) as only those $m$
  with $|m|\ge \rho(a)$ contribute to the final sum.
\end{proof}

\begin{lemma}
  \label{lem:lawton2}
  Suppose $a\in\IZ^d\ssm\{0\}$ such that
  $\rho(a)>\deg P$. For all $s\in [0,1)$, up-to finitely many
  exceptions, we have
  $|P(\be(as))|\not=0$   and
  \begin{equation*}
    \int_0^1 \log|P(\be(as)))|ds = \int_0^1  f_\sv(as) ds + O_{k}\left(
      \sv^{1/(k-1)}|\!\log\sv|\right).
  \end{equation*}
\end{lemma}
\refcomment{127}{We follow and adapt the proof of Lemma A.7~\cite{hab:gaussian}.}
\begin{proof}
  Say $a=(a_1,\ldots,a_d)$ with $\rho(a)>\deg P$. Then the coefficients
  of the univariate Laurent polynomial $Q=P(X^{a_1},\ldots,X^{a_d})$ are precisely the
  coefficients of $P$. Hence $|Q|=|P|=1$ and $Q$ has at most $k$
  non-zero terms.

  The first claim follows as $P(\be(as))=Q(\be(s))$ for all $s\in\IR$
  and since $Q\not=0$.
  For the second claim we note that the difference of the two integrals
  equals
  \begin{equation*}
    \int_{S(Q,\sv)} (\log|Q(\be(s))| - f_r(as))ds   
  \end{equation*}
  \refcomment{130}{with $S(Q,\sv)$ as in (\ref{def:SPr}).}
  Note that
  $\int_{S(Q,\sv)}\log|Q(\be(s))| ds \le 0$ as $\sv\le 1$. Recall
  Theorem \ref{thm:dobrowolski}
  which yields $\vol{S(Q,\sv)}\ll_k
  \sv^{1/(k-1)}$. As in the proof of Lemma 4~\cite{Lawton}, cf. also
  Theorem 7, Appendix G~\cite{Schinzel}, we find
  \begin{equation*}
    \int_{S(Q,\sv)} \log|Q(\be(s))|ds \ge -C \sv^{1/(k-1)}|\!\log \sv|,
  \end{equation*}
  where $C>0$ depends only on $k$.
  Finally, by the definition of $f_r$ we find $\log(\sv/2) \le f_{\sv}(as)
  \le 0$ if $|Q(\be(s))|< \sv$. Thus $\int_{S(Q,\sv)} f_r(as)ds$ is also
  $O_{k}(r^{1/(k-1)}|\!\log r|)$.
\end{proof}

\begin{lemma}
  \label{lem:lawton3}
  We have
  \begin{equation*}
    \left| \int_{[0,1)^d} \bigl(f_\sv(x)-\log |P(\be(x))|\bigr)dx\right|
    \ll_{d,k} \sv^{1/(4(k-1))}.
  \end{equation*}
\end{lemma}
\refcomment{127}{We follow and adapt the proof of Lemma A.8~\cite{hab:gaussian}.}
\begin{proof}
  We have
  \begin{alignat*}1
    &\left| \int_{[0,1)^d} (f_\sv(x)-\log |P(\be(x))|)dx\right|
    = 
    \left| \int_{[0,1)^d} (\phi_\sv(|P(\be(x))|^2)-1)\log |P(\be(x))|dx\right|\\
    &\qquad\qquad\le 
    \int_{[0,1)^d} |\phi_\sv(|P(\be(x))|^2)-1|\left|\log |P(\be(x))|\right|
    dx \\ &\qquad\qquad\le
    \left(\int_{[0,1)^d}|\phi_\sv(|P(\be(x))|^2)-1|^2 dx\right)^{1/2}
    \left(\int_{[0,1)^d}\left|\log |P(\be(x))|\right|^2 dx\right)^{1/2}
  \end{alignat*}
  by definition and where we used the Cauchy-Schwarz inequality in the last step. The second
  integral on the final line is $\ll_{d,k} 1$ 
  by Lemma \ref{lem:volSPep}(ii). The first integral  is 
  \begin{equation*}
    \int_{S(P,\sv)}  |\phi_\sv(|P(\be(x))|^2)-1|^2 dx
    \le\vol{S(P,\sv)}\ll_{d,k} {\sv}^{1/(2(k-1))}
  \end{equation*}
  by Lemma \ref{lem:volSPep}(i) and $|P|=1$. We take the square
  root to complete the proof.
\end{proof}

\begin{proof}[Proof of Theorem \ref{thm:lawtonquant}]
As stated below Theorem \ref{thm:lawtonquant} we may assume $d\ge 2$. 
As we have seen in the proof of Lemma \ref{lem:lawton2}, the condition
$\rho(a) > \deg P$ guarantees $P(X^{a_1},\ldots,X^{a_d})\not=0$.
We may also assume that $P$ is non-constant.
Moreover, replacing $P$ by $P/|P|$
leaves $m(P(X^{a_1},\ldots,X^{a_d}))-m(P)$ invariant. So it suffices
to prove the theorem if $|P|=1$.

We fix the parameters $b=4d\ge d+1$ and
 $\sv = 
\rho(a)^{-1/4}/2\le 1/2$.

We write  $|\ma{P(X^{a_1},\ldots,X^{a_d})} - \ma{P}|$ as $\left|\int_{0}^1 \log |P(\be(as))| ds - \int_{[0,1)^d}
\log |P(\be(x))|dx \right|$ and find that it is at most 
\begin{alignat*}1
  &\left|\int_{0}^1 f_\sv(as) ds - \int_{[0,1)^d}
f_\sv(x) dx \right| + 
  \left|\int_{0}^1 (\log |P(\be(as))| - f_\sv(as)) ds\right| \\
  &\qquad\qquad
  +
  \left|\int_{[0,1)^d} (f_\sv(x)-\log |P(\be(x))| ) dx\right|.
\end{alignat*}
Then by Lemmas \ref{lem:lawton1},
\ref{lem:lawton2},  and \ref{lem:lawton3} this
 sum is
\begin{equation*}
  \ll_{d,k}
\frac{|\!\log\sv|}{\sv^{2b}}  \frac{(\deg P)^{b^2}}{\rho(a)^{b-d}}+
  \sv^{{1}/{(k-1)}}|\!\log \sv|
  +\sv^{1/(4(k-1))}.
\end{equation*}
By our  choice of  $\sv$ and $\rho(a)\ge 2$, the sum is
\begin{equation*}
\ll_{d,k}
\frac{\log \rho(a)}{\rho(a)^{b-d-b/2}}
(\deg P)^{b^2} 
+ \frac{\log \rho(a)}{\rho(a)^{1/(4(k-1))}}
+  \frac{1}{\rho(a)^{1/(16(k-1))}}.
\end{equation*}
Finally, as
$b=4d$   the sum is
\begin{equation*}
  \ll_{d,k}  {(\deg P)^{16d^2}}\frac{\log \rho(a)}{\rho(a)^d}
  +  \frac{\log \rho(a)}{\rho(a)^{1/(4(k-1))}}
  + \frac{1}{\rho(a)^{1/(16(k-1))}}.\qedhere
\end{equation*}
 \end{proof}


\section{Recovering the theorem of Lind, Schmidt, and Verbitskiy}

In this appendix we recover from our work a variant of
Lind, Schmidt, and Verbitskiy's Theorem 1.1~\cite{LSV:13}. This variant is
stated in the introduction as Theorem \ref{thm:LSV}
For a finite subgroup $G\subset\IG_m^d$. Recall that we defined $\delta(G)$
in (\ref{def:deltaG}).

\begin{lemma}
  \label{lem:finitegroupcount}
  Let $G$ be a finite subgroup of $\IG_m^d$.
  If $a\in\IZ^d\ssm\{0\}$, then
  \begin{equation*}
    \frac{1}{\# G}\# \left\{ \bzeta \in G : \bzeta^a = 1 \right\} \le \frac{|a|}{\delta(G)}. 
  \end{equation*}
\end{lemma}
\begin{proof}
  We will detect $\bzeta^a=1$ using the character  $\chi(\bzeta) = \bzeta^a$ of $G$. 
  The image $\chi(G)$ is a cyclic subgroup of $\IC^\times$ of order $E$, say. 
  For $\bzeta\in G$,  the sum
  $\sum_{k=0}^{E-1} \chi(\bzeta)^k = 0$ equals $E$ if $\bzeta^a=1$ and vanishes
  otherwise. The number of solutions $\bzeta\in G$ of
  $\bzeta^a=1$ is thus
  \begin{equation*}
    \sum_{\bzeta\in G}      \frac 1E \sum_{k=0}^{E-1} \chi(\bzeta^k) =
    \frac 1E \sum_{k=0}^{E-1}\sum_{\bzeta\in G}   \chi(\bzeta)^k
    = \frac 1E \sum_{k=0}^{E-1}  \frac{\# G}{E} \sum_{\xi \in \chi(G)}
    \xi^k 
    =  \frac{\#G}{E}.
  \end{equation*}
  We conclude the proof as    $\bzeta^{aE}=\chi(\bzeta)^E=1$ for all
  $\bzeta\in G$  and hence
  $E\ge \delta(G)/|a|$.
\end{proof}

\begin{lemma}
  \label{lem:finitegroupcount2}
  Let $G$ be a finite subgroup of $\IG_m^d$.
  \begin{enumerate}
  \item [(i)] If $T\ge 1$, then
      \begin{equation*}
    \frac{1}{\#G} \# \left\{\bzeta\in G : \delta(\bzeta)\le
    T\right\}\le \frac{3^dT^{d+1}}{ \delta(G)}.
      \end{equation*}
  \item[(ii)] If $\kappa > 0$, then       
    \begin{equation*}
      \frac{1}{\#G} \sum_{\bzeta\in G} \delta(\bzeta)^{-\kappa} \le
\frac{4^d}{\delta(G)^{{\kappa}/({d+1+\kappa})}}. 
    \end{equation*}
  \end{enumerate}
\end{lemma}
\begin{proof}
Any  $\bzeta \in G$ with $\delta(\bzeta)\le T$  satisfies $\bzeta^a =1$ for some
$a\in\IZ^d\ssm\{0\}$ and $|a|\le T$.
The number of such $a$ is at most $(2T+1)^d \le 3^dT^d$ and each $a$ leads to
at most $|a|\#G /\delta(G) \le T\#G/ \delta(G)$ different $\bzeta$ by
Lemma \ref{lem:finitegroupcount}. This implies (i).

For the second assertion
  we split up the elements in $G$ into those
  with $\delta(\bzeta)\le T$ and those with $\delta(\bzeta)>T$; here
  $T\ge 1$ is a parameter to be chosen. 

For the lower range, we use the trivial lower bound $\delta(\bzeta)\ge
1$ and part (i) to obtain
\begin{equation*}
  \frac{1}{\#G} \sum_{\substack{\bzeta\in G \\ \delta(\bzeta)\le T}}
  \delta(\bzeta)^{-\kappa} \le \frac{3^dT^{d+1}}{\delta(G)}.
\end{equation*}

For the higher range,  we have
\begin{equation*}
  \frac{1}{\#G} \sum_{\substack{\bzeta\in G \\ \delta(\bzeta)> T}}
  \delta(\bzeta)^{-\kappa} \le \frac{1}{T^\kappa}.
\end{equation*}

The lemma follows by taking the sum of these two bounds with
$T=\delta(G)^{1/(d+1+\kappa)}$. 
\end{proof}

\begin{proof}[Proof of Theorem \ref{thm:LSV}]
  Without loss of generality we can assume that $P$ is a polynomial.

Any finite subgroup of $\IG_m^d$ is defined over $\IQ$, \textit{i.e.},
it is map to itself under the action of the absolute Galois group of
$\IQ$, see Corollary 3.2.15~\cite{BG}. 
We decompose $G$ into a disjoint union
$G_1\cup\cdots\cup G_m$ of Galois
orbits.  It is useful to fix a
representative $\bzeta_i\in G_i$ for each $i\in \{1,\ldots,m\}$
and define $N_i = \ord(\bzeta_i)$. For these $i$ all elements in $G_i$
have the same order and the Galois action is the
natural action of $(\IZ/N_i\IZ)^\times $  on $G_i$. Moreover, $\#G_i =
\varphi(N_i)$. 
Note that $\delta$ is constant on each $G_i$ as 
$\delta(\bzeta^\sigma)=\delta(\bzeta)$ for all field automorphisms
$\sigma$.

  Let $T\ge 1$ be a parameter depending on $\delta(G)$ and large in
  terms of $P,d$ which we will
  fix in due time.
We split our average (\ref{eq:LSVaverage}) up into those $\bzeta$ with $\delta(\bzeta)\le T$
and those with $\delta(\bzeta)>T$.

First, we  will show  that the sum
\begin{equation}
  \label{eq:LSVreductionsum1}
  \frac{1}{\#G} \sum_{\substack{\bzeta\in G \\ \delta(\bzeta)\le T, P(\bzeta)\not=0}}
  \log|P(\bzeta)| = 
  \frac{1}{\#G} \sum_{\substack{i=1 \\ \delta(\bzeta_i)\le T , P(\bzeta_i)\not=0}}^m
\sum_{ \sigma \in \GammaNi }
 \log|P(\bzeta_i^\sigma)|
\end{equation}
 is negligible.
Say $P(\bzeta_i)\not=0$. Then $P(\bzeta_i)$ lies in a number field of
degree $\varphi(N_i)$ over $\IQ$. So
 \begin{equation*}
 \left|\sum_{\sigma\in \GammaNi}
\log|P(\bzeta_i^\sigma)|\right|
\le \sum_{\sigma\in \GammaNi}
\left|\log|P(\bzeta_i^\sigma)|\right| \le 2\varphi(N_i)
\hproj{P(\bzeta_i)}
\ll_P \varphi(N_i)
 \end{equation*}
 where we used the height (\ref{def:heightx})
and its basic  properties. So the
 absolute value of (\ref{eq:LSVreductionsum1})
 is at most
 \begin{equation}
   \label{eq:LSVdeltaleT}
  \ll_P  \frac{1}{\#G} \sum_{\substack{i=1 \\ \delta(\bzeta_i)\le T}}^m \varphi(N_i)
  \ll_P \frac{1}{\#G} \sum_{\substack{\bzeta\in G\\ \delta(\bzeta)\le T}}1\ll_{d,P} \frac{T^{d+1}}{ \delta(G)}.
\end{equation}
by  Lemma \ref{lem:finitegroupcount2}(i).  

The remaining sum
is
\begin{equation*}
 \frac{1}{\#G} \sum_{\substack{i=1 \\\delta(\bzeta_i)> T}}^m \sum_{ \sigma \in \GammaNi} \log|P(\bzeta_i^\sigma)|;
\end{equation*}
note that $P(\bzeta_i^\sigma)\not=0$ for $T$ large enough by
Theorem \ref{thm:main}.
We use this theorem to 
\refcomment{135}{rewrite the inner sum} as
\begin{alignat*}1
  \frac{1}{\#G} &\sum_{\substack{i=1 \\\delta(\bzeta_i)> T}}^m
       {\varphi(N_i)} \left(m(P) + O_{d,P}(\delta(\bzeta_i)^{-\kappa})\right)
       \\ &=
         \frac{1}{\#G} \left(\sum_{\bzeta\in G:\delta(\bzeta)>T}
              1\right)m(P) +
              O_{d,P}\left(\frac{1}{\#G}\sum_{\bzeta\in G:
                \delta(\bzeta)>T}\delta(\bzeta)^{-\kappa}\right)
       \\&=\left(1-\frac{1}{\#G} \sum_{\bzeta\in G, \delta(\bzeta)\le
         T}1\right) m(P)  + O_{d,P}\left(\delta(G)^{-\frac{\kappa}{d+1+\kappa}}\right)
\end{alignat*}
where we used Lemma \ref{lem:finitegroupcount2}(ii).
The remaining average in
the last line is $O_d(T^{d+1}/\delta(G))$
by Lemma \ref{lem:finitegroupcount2}(i). 

We combine this estimate with the first bound (\ref{eq:LSVdeltaleT})
to conclude that the average (\ref{eq:LSVaverage}) equals
\begin{equation*}
 m(P) + O_{d,P}(T^{d+1} \delta(G)^{-1} +
\delta(G)^{-\frac{\kappa}{d+1+\kappa}}) 
\end{equation*}
The theorem follows with the choice $T = c \delta(G)^{1/(2(d+1))}$
where $c\ge 1$ is sufficiently large in terms of $d$ and $P$. 
\refcomment{137}{The exponent $\kappa$ in (\ref{eq:LSVaverage}) is $\min\{1/2,\kappa/(d+1+\kappa)\}$
in the notation here.}
\end{proof}

We leave to the interested reader the task of generalizing the previous theorem to polynomials
defined over an arbitrary number field. 


\bibliographystyle{amsplain}
\bibliography{literature}

\vfill\hfill\today
\end{document}